\documentclass[11pt]{amsart}
\usepackage{amssymb,amsthm,amsfonts,mathrsfs}
\usepackage{amsmath,amscd}
\usepackage[hmargin=3cm,vmargin=3.5cm]{geometry}
\usepackage{hyperref}
\usepackage{color}
\usepackage{graphicx}
\usepackage{stmaryrd}  
\usepackage{tikz-cd}

\newtheorem{theorem}{Theorem}[section]

\newtheorem{prop}[theorem]{Proposition}

\newtheorem{corollary}{Corollary}

\let\oldtocsection=\tocsection
\let\oldtocsubsection=\tocsubsection
\renewcommand{\tocsection}[2]{\hspace{0em}\oldtocsection{#1}{#2}}
\renewcommand{\tocsubsection}[2]{\hspace{1em}\oldtocsubsection{#1}{#2}}

\usepackage{chngcntr}
\counterwithin{figure}{subsection}

\title{Universal construction of topological theories  in two dimensions}
\author{Mikhail Khovanov}
 \address{Department of Mathematics, Columbia University, New York, NY 10027, USA}
 \email{\href{mailto:khovanov@math.columbia.edu}{khovanov@math.columbia.edu}}
\date{July  9, 2020}

\begin{document}

\begin{abstract}
    We consider  Blanchet, Habegger, Masbaum and  Vogel's universal  construction  of topological theories in dimension two, using it to produce interesting theories that do not satisfy the usual two-dimensional TQFT axioms. Kronecker's  characterization of rational functions allows us to classify theories over a field with finite-dimensional state spaces and introduce their extension to theories with the ground ring the product of rings of symmetric functions in N and M variables. We look at several examples of  non-multiplicative theories and see Hankel matrices, Schur and supersymmetric Schur polynomials quickly emerge from these structures. The last section explains how an extension of the Robert-Wagner foam evaluation to overlapping foams gives the Sergeev-Pragacz formula for the supersymmetric Schur polynomials and the Day formula for the Toeplitz determinant of rational power series as special cases. 
\end{abstract}

\def\R{\mathbb R}
\def\Q{\mathbb Q}
\def\Z{\mathbb Z}
\def\N{\mathbb N} 
\def\C{\mathbb C}
\def\CP{\mathbb P}
\renewcommand\SS{\ensuremath{\mathbb{S}}}
\def\l{\lbrace}
\def\r{\rbrace}
\def\lra{\longrightarrow}
\def\Hom{\mathrm{Hom}}
\def\Id{\mathrm{Id}}
\def\mc{\mathcal}
\def\mf{\mathfrak} 
\def\Ext{\mathrm{Ext}}
\def\End{\mathrm{End}}
\def\mfgl{\mathfrak{gl}}
\def\mfglN{\mathfrak{gl}_N} 

\def\Cob{\mathrm{Cob}} 
\def\Cobn{\Cob_n} 
\def\Cobk{\mathrm{Cob}_{2,k}}  
\newcommand\Cobkk[1]{\mathrm{Cob}_{2,#1}}
\def\CobOne{\Cob_1}
\def\CobTwo{\Cob_2}
\def\Pa{\mathrm{Pa}} 
\def\Ka{\mathrm{Ka}}
\def\sym{\mathrm{Sym}}
\def\undone{\underline{1}}
\def\undalpha{\underline{\alpha}}
\def\wtS{\widetilde{S}}
\def\mcD{\mathcal{D}}

\def\lra{\longrightarrow}
\def\kk{\mathbf{k}}  
\def\kR{\widetilde{R}} 
\def\lF{\langle F\rangle}  
\def\lG{\langle G\rangle} 
\def\gdim{\mathrm{gdim}}  
\def\rk{\mathrm{rk}}
\def\wchi{\widetilde{\chi}}
\def\Free{\mathrm{Fr}}
\def\ualpha{\underline{\alpha}}
\def\ovh{\overline{h}}
\def\ove{\overline{e}}
\def\iind{\mathrm{ind}} 
\def\symf{\mathrm{Sym}}

\let\oldemptyset\emptyset
\let\emptyset\varnothing

\newcommand{\brak}[1]{\ensuremath{\left\langle #1\right\rangle}}
\newcommand{\oplusop}[1]{{\mathop{\oplus}\limits_{#1}}}
\newcommand{\ang}[1]{\langle #1 \rangle_{RW} } 
\newcommand{\angf}[1]{\langle #1 \rangle } . 
\newcommand{\pseries}[1]{\kk\llbracket #1 \rrbracket}
\newcommand{\rseries}[1]{R \llbracket #1 \rrbracket}
\newcommand{\addfigure}{\vspace{0.1in} \begin{center} {\color{red} ADD FIGURE} \end{center} \vspace{0.1in} }
\newcommand{\add}[1]{\vspace{0.1in} \begin{center} {\color{red} ADD FIGURE #1} \end{center} \vspace{0.1in} }

\maketitle
\tableofcontents

%
%

\section{Universal construction in \texorpdfstring{$n$}{n} dimensions}
\label{sec_uni_n}

Consider the tensor category $\Cobn$ of oriented $n$-dimensional cobordisms. Its objects are oriented closed $(n-1)$-manifolds $N$. Morphisms from $N_0$ to $N_1$ are equivalence classes of oriented compact $n$-manifolds $M$ with $\partial M = (-N_0)\sqcup N_1$, and the equivalence relation is diffeomorphism rel boundary. This is a symmetric tensor category, and by a tensor category in this note we will mean a symmetric tensor category.  

Monoidal functors from $\Cobn$ into algebraic tensor categories $\mathcal{C}$, such as the category of vector spaces over a field $\kk$ or the category of projective modules over a commutative ring $R$, are known as  $n$-dimensional TQFTs (topological quantum field theories) and play an important role in mathematical physics and related fields~\cite{At}. 

There are many examples of interesting TQFT-type functors $\alpha$ that 
do not satisfy the tensor product condition. Instead of a family of  isomorphism $\alpha(N_0)\otimes \alpha(N_1)\cong \alpha(N_0\sqcup N_1)$ giving rise to isomorphisms of bifunctors $\Cobn \times \Cobn \lra \mathcal{C}$, there may exist a compatible family of homomorphisms
\begin{equation}
    \alpha(N_0)\otimes \alpha(N_1) \lra \alpha(N_0\sqcup N_1) 
\end{equation}
that form a natural transformation of bifunctors between these categories. In many examples these homomorphisms are injective for all $N_0, N_1$, so that the natural transformation is an inclusion. 
  
 We will call these functors, that are usually not monoidal, 
 \emph{topological theories}. There may already exist an established terminology in the literature, but we are not aware of it. 
 
Topological theories naturally emerge from the \emph{universal construction} as described by Blanchet, Habegger, Masbaum and Vogel~\cite{BHMV} and used in their approach to the  Witten-Reshetikhin-Turaev  3-manifold invariants. A very similar universal pairing construction was studied by Freedman, Kitaev, Nayak, Slingerland, Walker and Wang~\cite{FKNSWW} in the context of positive-definite forms, see also~\cite{CFW,KT,Fr} and the review~\cite{W}. 

A variation of the universal construction, for foams embedded in $\R^3$, was used in~\cite{Kh1} to categorify the Kuperberg invariant of closed $A_2$-webs, as a step in a  categorification of the quantum $sl(3)$ link invariant, a.k.a. the Kuperberg bracket. Mackaay and Vaz~\cite{MV} generalized this setup to the equivariant $sl(3)$ case.  Robert-Wagner evaluation formula for closed $GL(N)$ foams extends, via the universal construction for foams in $\R^3$, to homology groups (or state spaces) for planar MOY graphs~\cite{RW1}. 

The following  is the original setup for the  universal  construction~\cite{BHMV}. An invariant $\alpha$ of closed oriented $n$-manifolds is given by assigning to each such manifold $M$ an element $\alpha(M) \in R$, where $R$ is a fixed commutative  ring,  such that $\alpha(M)=\alpha(M')$ if the manifolds $M,M'$ are diffeomorphic. In this paper we impose  multiplicativity assumptions on the invariant: 
\begin{enumerate}
    \item $\alpha(\emptyset_n)=1$, where $\emptyset_n$ is the empty $n$-manifold; 
    \item $\alpha(M_1\sqcup M_2)=\alpha(M_1)\alpha(M_2),$ for $n$-manifolds $M_1,M_2$. 
\end{enumerate}

In case of general $n$, it is convenient  to introduce an  involution $\kappa$ on $R$ (analogous to complex conjugation on $\C$) to match the  operation of orientation reversal on a manifold,  $M \longmapsto -M$. We assume that $\kappa$ is an involution of the ring $R$ and require
\[ \alpha(-M) = \kappa(\alpha(M)),
\] 
for all $M$.

Now to each closed oriented $(n-1)$-manifold $N$ associate an $R$-module. First, define the free $R$-module $\Free(N)$ as a module with the basis $[M]$, for all oriented $n$-manifolds $M$ such that $\partial M = N$ (another natural notation for $[M]$ is $\alpha(M)$). We think of $[M]$ as a formal symbol associated to $M$. There is an  $R$-valued  $\kappa$-semilinear form on $\Free(N)$ given by
\begin{eqnarray*}
    & &  ([M_1],[M_2]) \ = \ \alpha ( -M_1 \cup_{N} M_2 ), \ \ \partial(M_1)=\partial(M_2)=N,   \\
    & & (\sum_i a_i [M_i], \sum_j b_j [M'_j]) \ = \ \sum_{i,j} 
    \kappa(a_i)b_j ([M_i],[M'_j]).
\end{eqnarray*}
The form is $\kappa$-symmetric 
\begin{equation}
    ([M_2],[M_1]) \ = \ \kappa(([M_1],[M_2]) ).  
\end{equation}
Define the  state space $\alpha(N)$ as the quotient of $\Free(N)$ by the  kernel of this semilinear form, 
\begin{equation}
    \alpha(N) = \Free(N)/\mathrm{ker}((,)) . 
\end{equation}
Since $(,)$ is $\kappa$-symmetric, it does not matter whether we consider left or right kernel. It can also be denoted $(,)_N$ to emphasize the dependence on $N$. This bilinear form induces a non-degenerate bilinear form, also denoted $(,)_N$, on $\alpha(N)$. The form is non-denenerate on $\alpha(N)$ but is not always unimodular. 

Clearly, $\alpha(\emptyset_{n-1}) \cong R[\emptyset_n], $ where $\emptyset_m$ denotes the empty $m$-manifold. Namely, the $R$-module $\alpha(\emptyset_{n-1})$ of the empty $(n-1)$-manifold $\emptyset_{n-1}$ is free of rank one with a generator $[\emptyset_n]$. For a closed $n$-manifold $M$ we have $[M]=\alpha(M)[\emptyset_n]$ in $\alpha(\emptyset_{n-1})$. 

\begin{prop}  For $(n-1)$-manifolds $N_0,N_1$ there is a canonical map of $R$-modules 
    \begin{equation} \label{eq:can_map_tensor}
        \alpha_{N_0,N_1} \ : \  \alpha(N_0)\otimes_R \alpha(N_1) \lra \alpha(N_0\sqcup N_1). 
    \end{equation}
\end{prop}
\emph{Proof:}
Given $M_i$ with $\partial M_i\cong N_i$, $i=0,1$, the map sends $[M_0]\otimes [M_1]$ to $[M_0\sqcup M_1]$. It's obviously well-defined. $\square$

\vspace{0.07in} 

In many cases, but not always, maps $\alpha_{N_0,N_1}$ are injective. 

\begin{prop} \label{prop_injective} Maps $\alpha_{N_0,N_1}$ are injective when $R=\kk$ is a field. 
\end{prop} 

\emph{Proof:} Assume that 
\[\alpha_{N_0,N_1}(\sum_i  \lambda_i [a_i]\otimes [b_i]) = \sum_i \lambda_i [a_i\sqcup b_i] = 0
\]
in  $\alpha(N_0\sqcup N_1)$,  where $a_i$, resp. $b_i$, are cobordisms from the empty $n$-manifold to $N_0$, resp. $N_1$ and $\lambda_i \in \kk$. Then for any cobordism $c$ from $N_0\sqcup N_1$ to $\emptyset_n$, 
\[ \sum_i \lambda_i \alpha(c (a_i \sqcup b_i)) = 0. 
\] 
Specializing to $c$ which are disjoint unions of cobordisms $a_i'$ and $b_i'$ from $N_0$, resp. $N_1$, to the empty $n$-manifold, we get 
\[ \sum_i \lambda_i \alpha(a_i'a_i)\alpha(b_i' b_i) = 0  \in R.  
\] 
But this is equivalent to 
\[
\sum_i  \lambda_i [a_i] \otimes [b_i] = 0  
\]
in $\alpha(N_0)\otimes_R \alpha(N_1).$

$\square$

The last step in the proof is problematic over more general commutative rings, with counterexamples discussed in Section~\ref{subsub_nonin}. We do expect that for a large class of interesting theories in various dimensions $n$ over rather general commutative  rings maps $\alpha_{N_0,N_1}$ will be injective for all $N_0,N_1$. 

\vspace{0.07in}

To an $n$-cobordism $M$ between $N_0$ and $N_1$ such that $\partial(M) = (-N_0)\sqcup N_1$, associate a map $\alpha(M):\alpha(N_0)\lra \alpha(N_1)$. 
This map takes 
a generator $[M_0]$ of $\alpha(N_0)$ corresponding to an $n$-manifold $M_0$ with $\partial M_0 = N_0$ to $[M M_0]\in \alpha(N_1)$. It's easy to see that $\alpha(M)$ is a well-defined map of $R$-modules. The following statement is obvious. 

\begin{prop} Assigning $R$-modules $\alpha(N)$ to closed oriented $(n-1)$-manifolds $N$ and $R$-module maps $\alpha(M)$ to oriented $n$-cobordisms $M$ is a functor from the category $\Cobn$ of oriented $n$-cobordisms to the category of $R$-modules. 
\end{prop} 

This functor is also denoted $\alpha: \Cobn\lra R\mathrm{-mod}$. It's not monoidal, in general, but satisfies a weaker monoidality property, where injective maps (\ref{eq:can_map_tensor}) are compatible with morphisms $\alpha(M)$ associated to $n$-manifolds $M$, so that the diagram below commutes. 
\[
\begin{tikzcd}[nodes={inner sep=10pt}]
\alpha(N'_0)\otimes_R \alpha(N'_1)   
\arrow{r}{\alpha_{N'_0,N'_1}}
& \alpha(N'_0\sqcup N'_1) \\
\alpha(N_0)\otimes_R \alpha(N_1) 
\arrow{u}{\alpha(M_0)\otimes \alpha(M_1)}
\arrow{r}{\alpha_{N_0,N_1}} & 
\alpha(N_0\sqcup N_1) 
\arrow[swap]{u}{\alpha(M_0\sqcup M_1)}
\end{tikzcd} 
\]
where $M_i$ is a cobordism from $N_i$ to $N_i'$, for $i=0,1$. 

\vspace{0.1in} 

For each  $N$ there are tube cobordisms $1_{tu(N)}$ and $\epsilon_{tu(N)}$ given by  taking the identity cobordism $\mathrm{id}_N$ and bending it so that both boundary copies of $N$ are on one side of the cobordism. Tube cobordism $1_{tu(N)}$ goes from the empty $(n-1)$-manifold $\emptyset_{n-1}$ to $N\sqcup (-N)$ and the cotube cobordism 
$\epsilon_{tu(N)}$ goes in the opposite direction. For $n=2$ and $N$ a circle these cobordism are depicted in Figure~\ref{fig_ell_cob2} below. 

The cotube cobordism $\epsilon_{tu(N)}$ can  be used to redefine the $R$-semilinear form $(,)_N$ on $\alpha(N)$. To describe $(,)_N$ via the cotube, take $n$-manifolds $M_1,M_2$ with boundary $N$, compose $M_1\sqcup (-M_2)$ with $\epsilon_{tu(N)}$ and evaluate the resulting closed $n$-manifold. 

The semilinear form $(,)_N$  defines an injective $R$-module homomorphism 
\begin{equation} \label{dual_map}
\phi_N: \alpha(-N) \lra \alpha(N)^{\ast}= \Hom_R(\alpha(N),R)
\end{equation}
which may not be surjective. 
Homomorphism  (\ref{eq:can_map_tensor}), while always injective, may not be an isomorphism either. It's an isomorphism for all $N_0,N_1$ iff each tube $1_{tu(N)}$ is in the image of $\alpha_{N,-N}$, that is, can be written as a finite $R$-linear combination 
\begin{equation}\label{eq_1_dec}
 [1_{tu(N)}] = \sum_{i=1}^m \lambda_i [M_i]\otimes [-M_i'] 
\end{equation}
for some $m$, $\lambda_i\in R$ and manifolds $M_i,M_i'$ with boundary  $N$. This condition need to be checked for connected $N$ only. 

Blanchet, Habegger, Masbaum and Vogel~\cite{BHMV} define the following three properties of $\alpha$, extended below by two more properties: 
\begin{itemize}
    \item (I) The map  $\phi_N$ in (\ref{dual_map})  is an isomorphism for all $N$. 
    \item (M) The map $\alpha_{N_0,N_1}$ in (\ref{eq:can_map_tensor}) is  an isomorphism for all  $N_0,N_1$. 
    \item (F) $R$-module $\alpha(N)$ is free of  finite rank and the form $(,)_N$ is unimodular, for all $N$. 
    \item (M') $[1_{tu(N)}]$ decomposes as a linear combination (\ref{eq_1_dec}) for any $N$. 
    \item (M'') The map $\alpha_{N_0,N_1}$ in (\ref{eq:can_map_tensor}) is an isomorphism onto an $R$-module  direct summand for all $N_0,N_1$. 
\end{itemize}
Paper~\cite{BHMV} defines (I), (M), (F) and mentions property (M'') without labeling it. 
Blanchet et al.~\cite{BHMV} point out that (F) implies (I) and (M''). We observe that, when $R$  is  a field, property (M') is equivalent to (M). Indeed, such a decomposition, done near one of the necks $N_i\times [0,1]$ of a cobordism representing a vector in $\alpha(N_0\sqcup N_1)$ shows that it comes from 
an element of $\alpha(N_0)\otimes \alpha(N_1)$. Consequently, (M') implies (M), over a  field. Vice versa, (M) applied to $N_0=N, N_1=-N$ implies that the tube cobordism $1_{tu(N)}$ decomposes as in (\ref{eq_1_dec}), giving (M'). 

Condition (M) says that $\alpha$ is an $n$-dimensional TQFT as defined in~\cite{At}, for instance. We will  see in this paper that already for $n=2$ there are interesting theories that  don't satisfy property (M).

\vspace{0.07in} 

$R$-module $\alpha(N)$ is a quotient of a free countably-generated $R$-module $\Free(N)$. It's also a submodule of the $R$-module $\alpha(-N)^{\ast}$. Consequently, if $R$ is an integral domain, $\alpha(N)$ has no torsion. 

\vspace{0.07in} 

Blanchet et al.~\cite{BHMV} use multiplicativity assumptions (1),(2) on the invariant, listed at the beginning of this section, which say that the invariant is multiplicative under the disjoint union of $n$-manifolds. We then say that the theory $\alpha$ is $n$-multiplicative. If, in addition, maps (\ref{eq:can_map_tensor}) are isomorphisms, that  is, condition (M) above holds, we say that $\alpha$ is $(n-1)$-multiplicative, which  is  a stronger property. Equivalently, $\alpha$ is an $n$-dimensional TQFT as defined by Atiyah~\cite{A}.  

\vspace{0.07in} 

Freedman, Kitaev, Naya, Slingerland, Walker and Wang~\cite{FKNSWW} (see also follow-up papers, including~\cite{CFW,KT,Fr,W}) study a similar semilinear pairing. They form commutative  $\C$-algebra with a  basis of diffeomorphism  classes of  all closed oriented $n$-manifolds and multiplication given by the   disjoint  union. Then, for a closed oriented  $(n-1)$-manifold $N$, they consider a $\C$-semilinear pairing on the $\C$-vector  space of all $n$-manifolds with boundary $N$, using orientation reversal in conjunction with complex involution for semilinearity. This $\C$-vector space can be alternatively viewed as a  free module over the above algebra with a basis given by manifolds with boundary $N$ and without closed components. Authors of the above papers prove the absence of null-vectors in this  space (vectors $v$  with $(v,v)=0$) in dimensions two and three, and show existence of null vectors in dimensions four and  higher. Similar to~\cite{BHMV}, theories given by  Freedman et al.'s construction   are not  $(n-1)$-multiplicative, that is, do not satisfy the Atiyah tensor product axiom or, equivalently, fail property (M) and related property (M'). 

\vspace{0.07in} 

This paper deals with the $n=2$ case only. We discuss  the basics of universal theories in dimension two and show that, over a field, finite rank theories correspond to rational functions that encode values of the invariant on connected closed surfaces over all genera. In a less topological language this can be traced back to the work of Kronecker~\cite{Fh} and has reappeared in Dwork~\cite{Dw}, also see a detailed exposition in Koblitz~\cite[Chapter V.5]{Kb}. Classification of finite rank theories can be restated via the finite or Sweedler dual of the polynomial  algebra, see  Section~\ref{sec_recur}. We single out families of these theories with coefficients in the tensor products of two rings of symmetric functions and discuss simplest examples that are not 1-multiplicative, that is, fail the Atiyah tensor product axiom (\ref{eq:can_map_tensor}). Hankel matrices, Schur functions and supersymmetric Schur functions naturally appear in this story. 

Sergeev-Pragacz determinant formula for the  supersymmetric Schur function generalizes the Jacobi-Trudi determinant expression for the Schur function. In Section~\ref{sec_SP_formula} we show how to interpret the Sergeev-Pragacz determinant via an extension of the Robert-Wagner foam evaluation formula~\cite{RW1} to overlapping foams carrying different sets of variables. The same extension recovers, as special cases,  the classical notion of the resultant  of  two  polynomials and  the formula of  Michael Day~\cite{Da} for the Toeplitz determinant of rational power series. We do not give new proofs of the Sergeev-Pragacz or Day formulas and only interpret the expressions in these formulas through evaluation of overlapping foams. 

\vspace{0.1in} 

{\bf Acknowledgments:} The author is  grateful to Yakov Kononov, Lev Rozansky and Anton Zeitlin for interesting discussions. The author would like to thank Victor Shuvalov for help with creating the figures and Yakov Kononov for valuable input on an earlier version of the paper. The author was partially supported by NSF grants DMS-1664240 and DMS-1807425 while working on this paper. 

%
%

\section{Universal construction in two dimensions}\label{sec_univ_con_two} 

%
%

\subsection{State space of a circle and the generating function} \label{subset_state_circle}
\quad 
\vspace{0.07in}

In this note we  restrict to the case $n=2$.
Closed oriented surfaces $M$ have the property $(-M)\cong M$, that is, $M$  is diffeomorphic to itself with the orientation reversed. Thus, in dimension two, the invariant $\alpha(M)$  takes values in the subring of $\psi$-invariants of $R$, and we assume without loss of generality that the involution $\psi$ is the identity. 

\vspace{0.03in} 

\emph{Remark:} In the situation when $M$ is embedded into a larger space (or is built by gluing oriented patches and has singularities or seams) it may make sense to keep $\psi$ nontrivial even for $n\le 2$. 

\vspace{0.03in} 

We specialize to $n=2$ and restrict to the case $\psi=\mathrm{id}$. The bilinear form is symmetric, 
\begin{equation}
    ([M_2],[M_1]) \ = \ ([M_1],[M_2]).  
\end{equation}

We start with a multiplicative invariant $\alpha$ of closed oriented 2-manifolds. Due to multiplicativity,  the invariant is determined by its values on connected 2-manifolds. Diffeomorphism classes of such manifolds are in a bijection with non-negative integers $g\ge 0$, where to  $g$ one associates a closed oriented  surface $S_g$ of that genus. Let $\alpha_g=\alpha(S_g)\in R$. Invariant $\alpha$ is determined by the infinite sequence 
$\ualpha=(\alpha_0,\alpha_1, \dots )$ of elements of $R$. We also use $\alpha$ in place of $\ualpha$ to denote this sequence, when it does not lead to confusion. Form the generating function 
\begin{equation}\label{eq_gen_fun}
   Z(T)= Z_{\undalpha}(T) = \sum_{g\ge 0}\alpha_g T^g \ \in \ \rseries{T},
\end{equation}
where $\rseries{T}$ is the ring of power series in $T$ with coefficients in $R$. 

We'll mostly consider the case when $R$ is a field, $R=\kk$, or $R$ is an integral domain with the field of fractions $Q=Q(R)$. 

To the generating function $Z(T)$ encoding information about invariants at all genera via the  universal construction we associate a collection of state spaces 
\[ A(k) := \alpha(\sqcup_{k} \SS^1), \ k\ge 0,
\] 
for the disjoint union $\sqcup_{k} \SS^1 $ of $k$ copies of $\SS^1$. 
The state space $A(k)$ is spanned by oriented 2-manifolds $M$ with $\partial(M) = \sqcup_k \SS^1$. $A(k)$ is an $R$-module quotient of the free $R$-module generated by these 2-manifolds modulo the kernel of the bilinear form given by gluing cobordisms along the boundary and evaluating via coefficients of  $Z(T)$. 

In particular, $A(0)=R[\emptyset_2]$ is the free $R$-module of rank one generated by the symbol of the empty 2-cobordism into the empty 1-manifold. Symmetric group $S_k$ acts on $A(k)$, the action induced by the permutation cobordisms, and there are multiplication homomorphisms 
\begin{equation}\label{eq_hom_parallel}
     \alpha_{k_1,k_2}: A(k_1) \otimes_R A(k_2) \lra A(k_1+k_2) 
\end{equation}
given by putting cobordisms with $k_1$ and $k_2$ boundary circles next to each other. 

We may exclude the case $Z(T)=0$, for then all state spaces are zero.  

We denote by 
\begin{equation}\label{eq_A_def} 
    A = A(1) = \alpha(\SS^1)
\end{equation}
the state space of the circle. 
The pants and cup cobordisms, see Figure~\ref{fig_ell_cob1}, turn this state space $A$ into an associative commutative unital algebra over $R$. Multiplication is provided by the pants cobordism, and the cup gives the unit element of $A$.

\begin{figure}[h]
\begin{center}
\includegraphics[scale=0.8]{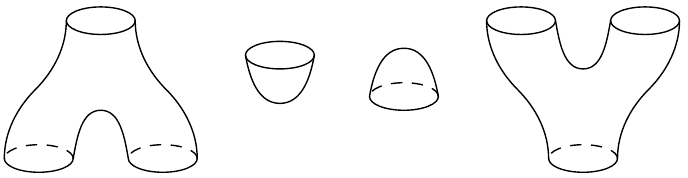}
\caption{Left to right: pants, cup, cap, and  copants cobordisms}
\label{fig_ell_cob1}
\end{center}
\end{figure}

More precisely, the pants cobordism induces a map $A(2)\stackrel{m_2^1}{\lra} A(1)$ and defines a multiplication on $A=A(1)$ via the composition 
\begin{equation}\label{eq_maps_A12} 
    A(1) \otimes A(1) \stackrel{\alpha_{1,1}}{\lra} A(2) \stackrel{m_2^1}{\lra} A(1),
\end{equation}
where we denote $\alpha_{1,1}=\alpha_{\SS^1, \SS^1}$, see (\ref{eq:can_map_tensor}) and (\ref{eq_hom_parallel}), and the map $m_n^1:A(n)\lra A(1)$ is induced by the multipants cobordism merging $n$ circles into one. 
  
The algebra $A=\alpha(\SS^1)$ is spanned by elements $x^n=[\SS^1_n]$, $n\ge 0$ representing the surface $\SS^1_n$ of genus $n$ with one boundary component.  The unit element $1=[\SS^1_0]$ is represented by a disk, and a generator $x=[\SS^1_1]$ is a torus with one boundary component, see Figure~\ref{fig_ell_cob3}.


\begin{figure}[h]
\begin{center}
\includegraphics[scale=1.0]{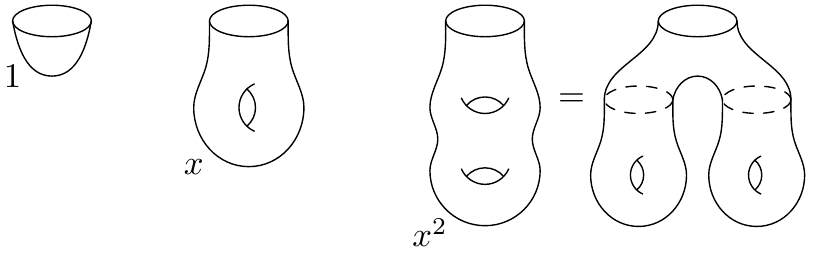}
\caption{Disk, torus and genus two surface with one boundary component, giving elements $1,x,x^2$ of $A=A(1)$.}
\label{fig_ell_cob3}
\end{center}
\end{figure}

A genus $n$ surface with one boundary component is $[\SS^1_n]=[\SS^1_1]^n=x^n.$ The natural homomorphism 
\begin{equation} \label{eq_nat_hom}
\rho_{\alpha}: R[x]\lra A=\alpha(\SS^1)
\end{equation} 
is either an isomorphism or identifies $A$ with the quotient of $R[x]$ by a nontrivial ideal. 
\vspace{0.05in} 

Each $A(k)$ is an associative commutative 
unital $R$-algebra. A spanning set of element of $A(k)$ is given by all cobordisms $M$ with  the  boundary $\partial(M)=\sqcup_k\SS^1$ being the union of $k$ circles (it's enough to choose one representative from each homeomorphism class of cobordisms rel boundary). Taking disjoint union of two such cobordisms $M_1,M_2$ and composing with $k$ pants cobordisms merging $2k$ boundary circles in pairs induces the algebra structure on $A(k)$.
Representatives of homeomorphism classes may be selected over all  possible decompositions of $k$ boundary circles into non-empty disjoint subsets, specifying connected components with those subsets as the boundary, and choosing the genus ($0,1,\dots$)  of each component, see also Proposition~\ref{prop_rel_Ak} below that describes a basis in algebras $A_k$ that surject onto $A(k)$ for any  $\undalpha$.  

\vspace{0.05in} 

The tube cobordism $1_{tu}$ from $\emptyset_1$ to $\SS^1\sqcup \SS^1$, see Figure~\ref{fig_ell_cob2}, gives an element in $A(2)$ that may not belong to the image of $A(1)^{\otimes 2}$ under the homomorphism $\alpha_{1,1}$ induced by  
 the 'disjoint union' map 
\begin{equation}\label{eq_alpha_one_one}
    \alpha_{1,1}:A(1)^{\otimes 2}= \alpha(\SS^1)\otimes \alpha(\SS^1)\lra \alpha(\SS^1\sqcup \SS^1) = A(2).
\end{equation}
This homomorphism potentially misses linear combinations of cobordisms given by the tube $1_{tu}$, possibly with handles added, see Figure~\ref{fig_ell_cob4} right. 
 
 The tube element $1_{tu}\in \alpha(\SS^1\sqcup  \SS^1)=A(2)$ satisfies a suitable duality property shown in Figure~\ref{fig_ell_cob2} right.  
 
 \vspace{0.1in} 
\begin{figure}[h]
\begin{center}
\includegraphics[scale=1.0]{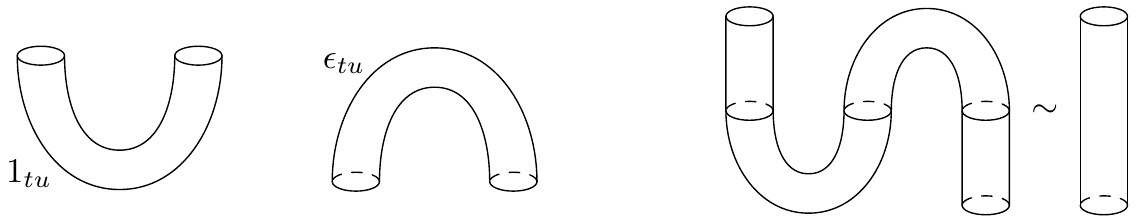}
\caption{On the left: tube cobordism $1_{tu}$ and  the cotube cobordism $\epsilon_{tu}$. On the right: an arrangement of these cobordisms homeomorphic rel boundary to the identity cobordism: $(\mathrm{id}_{\SS^1}\otimes \epsilon_{tu})\circ (1_{tu}\otimes \mathrm{id}_{\SS^1})=\mathrm{id}_{\SS^1}.$}
\label{fig_ell_cob2}
\end{center}
\end{figure}


\begin{figure}[h]
\begin{center}
\includegraphics[scale=1.0]{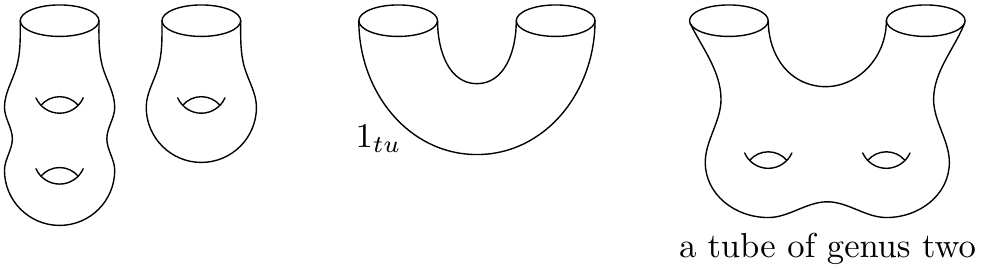}
\caption{Left to right: A decomposable cobordism which represents an element in the image of $A(1)^{\otimes 2}$ in $A(2)$ under $m^1_2$; the tube cobordism $1_{tu}$, a tube with two handles cobordism.}
\label{fig_ell_cob4}
\end{center}
\end{figure}

The cap cobordism in  Figure~\ref{fig_ell_cob1}  produces the trace map 
\begin{equation}\label{eq_trace_map} 
\epsilon: A=\alpha(\SS^1)\lra R.
\end{equation}
Likewise, the union of $k$ caps  gives the trace map $A(k)\lra R.$

The trace map, however, may not turn $A(1)=\alpha(\SS^1)$ into a Frobenius algebra. This is due to the absence of a neck-cutting relation for general $\alpha$, that is, $\alpha_{1,1}$ not being surjective, so that the tube $1_{tu}$ in Figure~\ref{fig_ell_cob4} center may not be an image of an element of $A(1)^{\otimes 2}$ under $\alpha_{1,1}$. 

As an example, 
Figure~\ref{fig_ell_cob4} depicts on the left an element of $\alpha_{1,1}(A(1)^{\otimes 2})\subset A(2)$ given  by the disjoint union of cobordisms with boundary $\SS^1$ each. 
To summarize, 
commutative algebra $A=\alpha(\SS^1)$ may be neither Frobenius nor of finite rank over $R$. For instance, this happens if the homomorphism (\ref{eq_nat_hom}) is an isomorphism. 

Thus, the situation may be more subtle than that of a usual 2-dimensional TQFT, which corresponds to a commutative Frobenius algebra over the ground ring $R$, see~\cite{A,Kc1, Kc2}. 

\vspace{0.07in} 

If $R$ is a field $\kk$, algebra $A$ is Frobenius iff it is finite dimensional over $\kk$, that is, iff the map (\ref{eq_nat_hom}) has a nontrivial kernel, for the simple reason  that any quotient of $\kk[x]$ by a non-trivial ideal is a Frobenius algebra.

\vspace{0.1in} 


\subsection{State space and the Hankel matrix}\label{subsec_state_hankel}
\quad
\vspace{0.07in} 

Recall that  $A$ is naturally  a quotient of $R[x]$ by  an ideal, via surjection (\ref{eq_nat_hom}). 
Equivalently, 
there's a homomorphism of $R$-modules 
\begin{equation}\label{eq_Rx_hom} 
     R[x] \lra \Hom_R(R[x],R) 
\end{equation} 
taking $f\in R[x]$ to the $R$-linear map $h\longmapsto (f,h)$, and we 
can identify $A=\alpha(\SS^1)$ with the image of $R[x]$ under this homomorphism. 

\vspace{0.1in} 

Let us come back to the generating function $Z(T)$ and its coefficients, see equation (\ref{eq_gen_fun}). 
The $R$-algebra $A=\alpha(\SS^1)$ is the quotient of the algebra $R[x]$ by the kernel of 
the $R$-bilinear form 
\begin{equation}\label{eq_bilin_form} 
(x^n,x^m) = \alpha(S_{n+m}) = \alpha_{n+m} \in R,
\end{equation}
\begin{equation}\label{eq_bilin_form_2} 
A=\alpha(\SS^1) \ \cong \  R[x]/\ker((,)).
\end{equation}
Indeed, gluing punctured tori $S^1_n$ and $S^1_m$ along the common boundary circle results in a closed surface of genus $n+m$ which evaluates to $\alpha_{n+m}$. 

\vspace{0.07in}

The matrix of the inner products in the spanning set $\{x^n\}_{n\in \N}$ is the Hankel matrix corresponding to the infinite sequence $\ualpha=(\alpha_0,\alpha_1, \dots )$:

\begin{equation}
     H =  H_{\ualpha}=
  \left( {\begin{array}{ccccc}
   \alpha_0 & \alpha_1 & \alpha_2 & \alpha_3 & \dots \\
   \alpha_1 & \alpha_2 & \alpha_3 & \alpha_4 & \dots \\
   \alpha_2 & \alpha_3 & \alpha_4 & \alpha_5 & \dots \\
   \alpha_3 & \alpha_4 & \alpha_5 & \alpha_6 & \dots \\
   \vdots & \vdots & \vdots & \vdots & \ddots \\
  \end{array} } \right)
\end{equation}
We view $H$ as a $\N\times \N$ matrix with rows and columns enumerated by $\N=\{0,1,2,\dots \}$ and entries in the commutative ring $R$. Matrix $H$ describes the map (\ref{eq_Rx_hom}) in the monomial basis of $R[x]$. Hankel matrices, also known  as  \emph{catalecticant} matrices in the invariant theory~\cite{IK}, are a starting point in the theory of orthogonal polynomials~\cite{Ch} and have many other applications~\cite{BS,Fa,G,Sch,T}.

We can identify the state space $\alpha(\SS^1)$ with $H(R^{\N})$, the image of the free $R$-module
$R^{\N}$ on generators $x^n$, $n\in \N$, under the endomorphism of $R^{\N}$ induced by the matrix $H$, and with the quotient of $R^{\N}$ by the kernel of this endomorphism: 
\[ A = \alpha(\SS^1) \ \cong \ H(R^{\N})  \ \cong \   R^{\N}/\mathrm{ker}(H).  
\]

Consider the ring $\sym=\Z[h_1,h_2,\dots]$ of symmetric functions in infinitely many variables with generators -- complete symmetric functions $h_n$. Assume that $\alpha_0\in R$ is invertible and form the homomorphism  
\begin{equation}\label{eq_hom_xi_p}
    \xi_0: \sym\lra R, \ \xi_0(h_n) = \alpha_0^{-1}\cdot\alpha_n , \ n\ge 1.
\end{equation}
Consider the Hankel matrix with the entries -- complete symmetric functions $h_n$ and $1$ in the upper left corner. A common convention is to set $h_0=1$, then $h_0$ can be entered there instead. 

\begin{equation}\label{eq_hankel_h}
     H_{h}=
  \left( {\begin{array}{ccccc}
   1 & h_1 & h_2 & h_3 & \dots \\
   h_1 & h_2 & h_3 & h_4 & \dots \\
   h_2 & h_3 & h_4 & h_5 & \dots \\
   h_3 & h_4 & h_5 & h_6 & \dots \\
   \vdots & \vdots & \vdots & \vdots & \ddots \\
  \end{array} } \right)
\end{equation}
Take the submatrix that consists of the first $N$ rows and columns of this matrix: 

\begin{equation} \label{eq_hankel_N_h}
      H_{[0,N-1],h} =  
  \left( {\begin{array}{ccccc}
   h_{0} & h_{1} & h_{2} & \dots & h_{N-1} \\
   h_{1} & h_{2} & h_{3} & \dots & h_{N} \\
   h_{2} & h_{3} & h_{4} & \dots & h_{N+1} \\
   \dots & \dots & \dots & \ddots & \dots \\
   h_{N-1} & h_{N} & h_{N+1} & \dots & h_{2N-2} \\
  \end{array} } \right)
\end{equation}
The upper left corner is $1=h_0$. Multiplying by the matrix $J_N$ of the  longest permutation (the  matrix of $1$'s on the antidiagonal and  zeros  elsewhere) results in the Toeplitz matrix 
\begin{equation} \label{eq_toeplitz_N}
      J_N H_{[0,N-1],h} =  
  \left( \begin{array}{ccccc}
   h_{N-1} & h_{N} & h_{N+1} & \dots & h_{2N-2} \\
   h_{N-2} & h_{N-1} & h_{N} & \dots & h_{2N-3} \\
   h_{N-3} & h_{N-2} & h_{N-1} & \dots & h_{2N-4} \\
   \dots & \dots & \dots & \ddots & \dots \\
   h_{0} & h_{1} & h_{2} & \dots & h_{N-1} \\
  \end{array}  \right)  . 
\end{equation}
The determinant of this matrix is given by  the Jacobi-Trudi formula, as the Schur function for the partition $\lambda_N:=((N-1)^N)=(N-1,N-1,\dots,N-1)$. Consequently, 
\begin{equation} 
  \det H_{[0,N-1],h}  =  (-1)^{N(N-1)/2}s_{\lambda_N}(h) \in \sym.
\end{equation}



Under the homomorphism $\xi_0$ above, the Schur function $s_{\lambda_{N}}(h)$ goes to 
\begin{equation}
s_{\lambda_{N}}(\xi_0(h)) = s_{\lambda_{N}}(\frac{\alpha_1}{\alpha_0},\dots, \frac{\alpha_{2N-2}}{\alpha_0}).
 \end{equation}
  
 \begin{prop} \label{prop_indep_schur} 
 Suppose that $R$ is an integral domain and $\alpha_0$ is invertible. Then vectors $1,x,\dots, x^{N-1}$ are $R$-linearly independent  in $A$ iff the Schur function  $s_{\lambda_{N}}(\xi_0(h))$ is not zero in $R$. 
 \end{prop} 
 \begin{proof}
 The Gram determinant 
 \[ \det H_{[0,N-1]} = \alpha_0^N \cdot s_{\lambda_N}(\xi_0(h))
 \] 
 is non-zero precisely when there is no linear relation on $1,x,\dots, x^{N-1}$ with coefficients in $R.$
  \end{proof}
  
 If we pick $N$ consecutive indices  and look at monomials 
 $x^k,x^{k+1},\dots, x^{k+N-1}$, the corresponding matrix of complete symmetric functions  
\begin{equation} \label{eq_hankel_N_k_h}
      H_{[k,k+N-1],h} =  
  \left( {\begin{array}{ccccc}
   h_{2k} & h_{2k+1} & h_{2k+2} & \dots & h_{2k+N-1} \\
   h_{2k+1} & h_{2k+2} & h_{2k+3} & \dots & h_{2k+N} \\
   h_{2k+2} & h_{2k+3} & h_{2k+4} & \dots & h_{2k+N+1} \\
   \dots & \dots & \dots & \ddots & \dots \\
   h_{2k+N-1} & h_{2k+N} & h_{2k+N+1} & \dots & h_{2(k+N-1)} \\
  \end{array} } \right)
\end{equation}
has determinant 
\begin{equation}
    \det(H_{[k,k+N-1],h}) =  (-1)^{N(N-1)/2} s_{\lambda_{N-1,k}}(h),
\end{equation}
where $s_{\lambda_{N,k}}(h)$ is the Schur function for the partition 
$\lambda_{N,k}= ((N+2k-1)^N)=(N+2k-1,\dots,N+2k-1)$. 
Assume that $k>0$ and consider homomorphism 
\begin{equation}\label{eq_hom_xi}
     \xi: \sym\lra R, \ \xi(h_n) = \alpha_n , \ n\ge 1.
\end{equation}
Under $\xi$, matrix  $H_{[k,k+N-1],h}$ goes to the Gram matrix 
\begin{equation} \label{eq_hankel_N_k}
      H_{[k,k+N-1]} =  
  \left( {\begin{array}{ccccc}
   \alpha_{2k} & \alpha_{2k+1} & \alpha_{2k+2} & \dots & \alpha_{2k+N-1} \\
   \alpha_{2k+1} & \alpha_{2k+2} & \alpha_{2k+3} & \dots & \alpha_{2k+N} \\
   \alpha_{2k+2} & \alpha_{2k+3} & \alpha_{2k+4} & \dots & \alpha_{2k+N+1} \\
   \dots & \dots & \dots & \ddots & \dots \\
   \alpha_{2k+N-1} & \alpha_{2k+N} & \alpha_{2k+N+1} & \dots & \alpha_{2(k+N-1)} \\
  \end{array} } \right)
\end{equation}
of the set of vectors $x^k,\dots, x^{k+N-1}$ in $A$, with the determinant
\begin{equation}
    \det(H_{[k,k+N-1]}) = (-1)^{N(N-1)/2} s_{\lambda_{N,k}}(\alpha),
\end{equation}
where we substitute $\alpha_m$ for the complete symmetric function $h_m$ in the expression for the Schur function. 

\begin{prop} \label{prop_indep} 
For an integral domain $R$, $\alpha$ as in Section~\ref{subset_state_circle}, and $k,N>0$, elements $x^k,x^{k+1},\dots, x^{k+N-1}$ are $R$-linearly  independent in $A$ if the Schur function $s_{\lambda_{N,k}}(\alpha)\not= 0$. 
\end{prop}
Taking $k>0$ avoids the need to restrict to invertible $\alpha_0$ and rescale homomorphism $\xi$ to $\xi_0$. 

\vspace{0.1in} 

We see that for a 'sufficiently large' integral domain $R$ and 'generic' $\alpha$, so that $s_{\lambda_{N,1}}(\alpha)\not= 0$ for $N>0$, vectors $x,x^2,\dots$ are  linearly independent in $A$. In particular, such $A$ has infinite rank as an $R$-module and is not a Frobenius $R$-algebra. 
For $A$ to have infinite rank over an integral domain $R$ it suffices to require that for any $N>0$ there is $k>0$ such that $s_{\lambda_{N,k}}(\alpha)\not= 0$. By a module $M$ of 'infinite rank' over an integral domain $R$ we may mean, for instance, a module such that the $Q(R)$-vector space $M\otimes_R Q(R)$ is infinite-dimensional, where $Q(R)$ is the field of fractions of $R$. 

Linear independence of vectors $x,x^2,\dots$ is equivalent to linear independence of vectors $1,x,x^2,\dots$, due to our construction of the bilinear form.  
Having $R[x]\lra A$ an isomorphism is equivalent to not having any skein relations on cobordisms with boundary $\SS^1$, 
modulo skein relations on closed cobordisms, which are $\alpha(S_g)=\alpha_g$, $g\ge 0$. In such theories $A=A(1)\cong R[x]$ is not Frobenius over $R$.
 
\vspace{0.07in} 

Note that the implication in Proposition~\ref{prop_indep} is only one way -- the matrix in (\ref{eq_hankel_N_k}) having determinant zero does not imply linear dependence of the vectors $x^k,\dots, x^{k+N-1}$. For that we would need linear dependence of the columns of the infinite $\N \times N$ matrix that contains  all inner products  $\alpha_{i+j}=(x^i,x^j)$ for $i\ge 0$ and $k\le j\le k+N-1$ as the entries. 

\vspace{0.07in} 

More general Schur functions can be recovered, up to sign, as Gram determinants of inner products of a sequence of consecutive monomials $x^k,x^{k+1},\dots, x^{k+N-1}$ and an arbitrary sequence of monomials 
$x^{k_1},\dots, x^{k_N}$ with increasing exponents $k_1<\dots<k_N$, via the general Jacobi-Trudi complete symmetric functions determinant. General Gram determinants between two sequences of monomials in $x$ of equal length will produce skew Schur functions~\cite{Mc} via the corresponding  Jacobi-Trudi formula. 

The relation between Toeplitz (or Hankel) determinants  and  the Jacobi-Trudi formula has been rediscovered many times, 
see the introduction section in Maximenko and  Moctezuma-Salazar~\cite{MMS} 
for a review of the literature. It's more common to substitute elementary rather then complete symmetric functions for the entries of Toeplitz matrices. When the number of variables is finite, elementary symmetric functions eventually vanish and the matrix consists of zeros outside of several diagonals surrounding the main diagonal, becoming a banded Toeplitz matrix. 

\vspace{0.07in} 

\emph{Remark:}
If we specialize from the ring $\sym$ of symmetric functions in infinitely many variables to the quotient ring $\sym_N$ of $N$-variable symmetric functions, Schur functions become characters of irreducible $GL(N)$-modules. Partition $\lambda_N=((N-1)^N)$ describes the character of the $(N-1)$-st tensor power of the determinantal one-dimensional representation $\Lambda^N(\C^N)$, with 
\[ s_{\lambda_N}(h) = (\gamma_1\dots\gamma_N)^{N-1}, 
\]
and the Gram matrix determinant for $1,x, \dots, x^{N-1}$ is given by this product, up to sign.

\vspace{0.1in} 


\subsection{Rational theories over a field.}\label{subset_rat_th}
\quad
\vspace{0.07in} 

We assume in this section that the ground ring $R$ is a field $\kk$. 
Let  $I=\{i_1, \dots, i_n\}$, $0\le i_1< i_2< \dots < i_n,$ be a finite subset of $\Z_+$ of cardinality $n$. 

Consider the principal minor $H_{I,\ualpha}$ of $H_{\ualpha}$ for the sequence $I$ of indices, 

\begin{equation} \label{eq_hankel_minor}
      H_I = H_{I,\ualpha} = 
  \left( {\begin{array}{ccccc}
   \alpha_{2i_1} & \alpha_{i_1+i_2} & \alpha_{i_1+i_3} & \dots & \alpha_{i_1+i_n} \\
   \alpha_{i_2+i_1} & \alpha_{2 i_2} & \alpha_{i_2+i_3} & \dots & \alpha_{i_2+i_n} \\
   \alpha_{i_3+i_1} & \alpha_{i_3+i_2} & \alpha_{2 i_3} & \dots & \alpha_{i_3+i_n} \\
   \dots & \dots & \dots & \ddots & \dots \\
   \alpha_{i_n+i_1} & \alpha_{i_n+i_2} & \alpha_{i_n+i_3} & \dots & \alpha_{2i_n} \\
  \end{array} } \right)
\end{equation}

As we've discussed above, if 
$\det(H_I)\not= 0$, then the set of monomials 
$\{x^{i_1},\dots,x^{i_n}\}$ is linearly independent in the state space $\alpha(\SS^1)$ of $\SS^1$.  This is a sufficient  condition. For the implication  the  other way, we need a linear dependence of columns $i_1,\dots, i_n$ of the infinite  matrix  $H_{\alpha}$. 

Denote by $H_{M,N}$ the submatrix of $H$ of size $(N+1)\times (N+1)$ that selects entries on the intersection of  consequent rows $\{i, i+1,\dots, i+N\}$ and consequent columns $\{j,j+1,\dots, j+N$ with $i+j=M$: 

\begin{equation}
      H_{M,N} =
  \left( {\begin{array}{cccc}
   \alpha_{M} & \alpha_{M+1} & \dots & \alpha_{M+N} \\
   \alpha_{M+1} & \alpha_{M+2}  & \dots & \alpha_{M+N+1} \\
    \vdots & \vdots & \ddots &  \vdots \\
   \alpha_{M+N} & \alpha_{M+N+1} &  \dots & \alpha_{M+2N} \\
  \end{array} } \right)
\end{equation}
This matrix does not  depend on a choice  of $i$ between $0$ and  $M$. 

\begin{theorem} \label{thm_rational}
 The following are equivalent: 
 \begin{enumerate}
     \item $Z_{\ualpha}(T)$ is the Taylor series of a rational function in $T$. 
     \item There is a finite sequence $q_0,\dots,q_N\in\kk$, with $q_N\not= 0$, such that $q_0 \alpha_m + q_1\alpha_{m+1}+\dots + q_N \alpha_{m+N}=0$ for all $m>>0$.
     \item There exist $N$ and $M$ such that $\det(H_{m,N})=0$ for all $m>M$. 
     \item There exist $N$ and $M$ such that $\det(H_{m,n})=0$ for all $m>M$ and $n>N$. 
     \item $\kk$-vector space $A=\alpha(\SS^1)$ is finite-dimensional. 
 \end{enumerate}
\end{theorem}
\emph{Proof:} Equivalence of (1)-(4) is proved in~\cite[Chapter V.5, Lemma 5]{Kb}, also see~\cite[Theorem 5.1]{Di}.
Condition (2) implies that for $m>>0$, $\alpha_{m+N}$ is a linear combination of $\alpha_i$'s with a smaller index than $m+N$, so that $\alpha_0, \dots, \alpha_{m+N-1}$ span $\alpha(\SS^1)$ as a $\kk$-vector space. 
Vice versa, if $\alpha(\SS^1)$ has a finite dimension $K$, 
then $\deg(H_{m,n})=0$ for all $n\ge K$ and all $m$, implying condition (4). 
$\square$  

\emph{Remark:} Essentially this theorem (without  the topological  theory interpretation (5)) appears in Drowk's proof of rationality of zeta function of an algebraic variety over a finite field via p-adic analysis, see \cite{Dw} and references cited in the proof above.  Proof of Proposition~\ref{prop_image_fd} below explains another approach to this theorem. The theorem, in fact, goes back to Kronecker, see~\cite[Section 8.3.1]{Fh},~\cite[Lemma~I.3.III]{S}, and~\cite{O}. 

\vspace{0.1in}

\emph{Remark:} Over a  field $\kk$, it's natural to split two-dimensional topological theories as considered in this paper into two types: 
\begin{enumerate}
    \item  $A$ is finite-dimensional over $\kk$. Equivalently, $Z(T)$ is a rational function.
    \item $A$ is infinite-dimensional over $\kk$. Equivalently, $Z(T)$ is not a rational function.
\end{enumerate}
Theories of type (1) naturally  split into two classes: 
\begin{enumerate}
    \item[(a)] Maps $\alpha_{1,1}$ in (\ref{eq_alpha_one_one}) is an isomorphism. Equivalently, the theory is a genuine 2D TQFT, a  neck-cutting relation exists for the theory, $1_{tu}$ is in the image  of $\alpha_{1,1}$.  
    \item[(b)] Map $\alpha_{1,1}$ is not surjective. Equivalently, $1_{tu}$ is  not in the  image of $\alpha_{1,1}$. 
\end{enumerate}

\vspace{0.1in}


\subsection{Semi-universal rational theories \texorpdfstring{$\alpha_{M,N}$}{alpha(M,N)}} \label{subsec_semiu}
\quad
\vspace{0.07in} 

Over an algebraically closed field, the  numerator  and denominator of  a rational function both decompose into products of linear factors. Without assuming that the field is algebraically closed, consider such  a  factorizable rational function 
\begin{equation}\label{eq_Z_rat_1} 
    Z(T)=\frac{P(T)}{Q(T)} = \frac{(1+\beta_1 T)\dots (1+\beta_N T)}{(1-\gamma_1 T)\dots (1-\gamma_M T)},
\end{equation}
where we restricted to the case $Z(0)=1$. 
We can write 
\begin{equation}
    P(T)= e_N(\beta)T^N + e_{N-1}(\beta)T^{N-1}+\dots + e_1(\beta)T+e_0(\beta),
\end{equation}
where 
\[ e_k(\beta) = \sum_{1\le  i_1<\dots < i_k\le N} \beta_{i_1}\dots \beta_{i_k} 
\] 
is the $k$-elementary symmetric function in $\beta_1, \dots, \beta_N$. The convention  is that  $e_k(\beta)=0$ for $k>N$ and  $k<0$ and $e_0=1$. 

Likewise, 
\begin{equation}
    \frac{1}{Q(T)} = \sum_{k\ge 0} h_k(\gamma) T^k,
\end{equation}
where 
\[ h_k(\gamma) = \sum_{i_1\le \dots \le i_k\le M} \gamma_{i_1}\dots \gamma_{i_k} 
\] 
is the $k$-th complete symmetric function of $\gamma_1, \dots, \gamma_M$. Set  $h_0(\gamma)=1$ and $h_{k}(\gamma)=0$ for $k<0$. 

Then  
\begin{equation}
    Z(T) = \left( \sum_{i=0}^N e_i(\beta) T^{i}\right) 
    \cdot \left( \sum_{k\ge 0} h_k(\gamma) T^k \right),
\end{equation}
and 
\begin{equation} \label{eq_Z_N_M_1}
    Z(T) =  \sum_{k\ge 0} \left(\sum_{i=0}^{\min(k,N)} 
    e_{i}(\beta) h_{k-i}(\gamma) \right)  T^{k}. 
\end{equation}
Note that coefficients at powers of $T$ are symmetric functions in variables $\beta_1, \dots, \beta_N$ and in $\gamma_1, \dots, \gamma_M$. 

It's convenient to introduce coefficient rings 
\begin{eqnarray} \label{eq_ring_Rp_1}
 R'_{M,N} & = & \kk[\gamma_1, \dots, \gamma_M,\beta_1, \dots, \beta_N], \\ \label{eq_ring_R_2}
 R_{M,N} & = & \kk[h_1(\gamma), \dots, h_M(\gamma),e_1(\beta),\dots, e_N(\beta)],
 \end{eqnarray}
 where $\kk$ is either a ground field or $\Z$ and work over $R_{M,N}$ and $R'_{M,N}$ as the ground ring $R$ rather than over a field. In other words, we start over a field where the numerator and the denominator factorize, 
 but  then turn (signed  or inverse) roots of these factorizations into formal variables  and work over the corresponding polynomial ring $R_{M,N}$ or over its subring of $S_M\times S_N$-invariant functions. 
 
 Denote by by $\sym_M(\gamma)\subset \kk[\gamma_1, \dots, \gamma_M]$ the subring of symmetric functions in $\gamma_1,\dots ,\gamma_M$ and by $\sym_N(\beta)\subset \kk[\beta_1, \dots, \beta_N]$ the subring of symmetric functions in $\beta_1,\dots ,\beta_N$. Then
 \begin{equation}
     R_{M,N} = \sym_M(\gamma)\otimes \sym_N(\beta)
 \end{equation}
 is the  tensor  product of these  rings. 
 It is the subring of $R'_{M,N}$ of $S_M\times S_N$-invariants, under the permutation action, and  can also  be written as in (\ref{eq_ring_R_2}). As an $R_{M,N}$-module, $R'_{M,N}$ is free of rank $N!M!$. 
 
 We use formula (\ref{eq_Z_N_M_1}) to define two topological theories, $\alpha'_{M,N}$ and $\alpha_{M,N}$. 
 Topological theory $\alpha'_{M,N}$ is defined over the ring  $R'_{M,N}$ and to a closed surface of genus $k$ it assigns the coefficient at $T^k$ in (\ref{eq_Z_N_M_1}). TQFT $\alpha_{M,N}$ is defined over $R_{M,N}$ but otherwise is given by the same power series. When there's no possibility of confusion, we may denote $\alpha_{M,N}$ simply by $\alpha$, and denote  $\alpha'_{M,N}$ by  $\alpha'$.  
 
 In $\alpha_{M,N}$, since the ground ring $R_{M,N}$ does not contain $\beta_i$ and $\gamma_j$'s, only their symmetric functions, it may be convenient to denote $e_i(\beta)$ and $h_i(\beta)$ simply by  $e_i$ and $h_i$ and denote $e_j(\gamma)$ and $h_j(\gamma)$ by $\ove_i$ and  $\ovh_j$. 
 
 In particular, for the first few coefficients of $Z(T)$,
 \begin{eqnarray*}
 \alpha_0 &  = &  e_0 \ = \ 1 , \\
 \alpha_1 & = &  \ovh_1 + e_1, \\
 \alpha_2 &  = &  \ovh_2 + e_{1}\ovh_1 + e_{2},  
\end{eqnarray*}
and, in general, 
\begin{equation}\label{eq_alpha_k} 
\alpha_k = \sum_{i=0}^{\min(k,N)} 
    e_{i}(\beta) h_{k-i}(\gamma) = \sum_{i=0}^{\min(k,N)} e_i \ovh _{k-i}. 
\end{equation} 
As we've mentioned, setting $e_i=0$ for $i>N$, $e_0=1$ and $\overline{h}_i=e_i=0$  for $i<0$ allows to remove 
the limits in the sum formula above and 
write 
\begin{equation} 
\alpha_k = \sum_{i} 
    e_i \overline{h}_{k-i}. 
\end{equation}
 The number of non-zero terms grows until we reach 
\[ \alpha_N = \ovh_N+e_1\ovh_{N-1}+\dots + e_{N-1} \ovh_1 +e_N,
\] 
where it stabilizes at $N+1$ terms, with 
\[ \alpha_k = \ovh_{k}+e_{1}\ovh_{k-1}+\dots + e_{N-1} \ovh_{k-N+1} +e_N\ovh_{k-N},
\]
for $k\ge N$. 

\vspace{0.1in}

Over $R_{M,N}$ and $R'_{M,N}$ these theories can be made graded, with 
 \begin{equation} 
 \deg(T)= -2, \ \deg(\beta_i)=\deg(\gamma_j)=2,
\end{equation}
so that power series $Z(T)$ is homogeneous of zero degree, and the state spaces $A(k)$ associated to the union of $k$ circles of the theory $\alpha_{M,N}$ and $A'(k)$ of the theory $\alpha'_{M,N}$ are graded modules over these ground rings. 

\vspace{0.1in}

We have
\begin{eqnarray*}
Z(T)Q(T) & = & \left( \sum_{k\ge 0}\alpha_k T^k \right) \cdot 
\left( \sum_{i=0}^M  (-1)^i e_i(\gamma) T^i\right) \\
 & = & \sum_{n\ge 0} \left( \sum_{i=0}^{\min(M,n)} (-1)^i e_i(
 \gamma) \alpha_{n-i}\right) T^n
\end{eqnarray*}
and 
\begin{equation}
Z(T)Q(T)  =  P(T) = \sum_{n=0}^N e_{n}(\beta)T^n . 
\end{equation}
Hence, for $n\ge \max(N+1,M)$, 
\begin{eqnarray*}
    0 & = & \sum_{i=0}^M  (-1)^i e_i(\gamma) \alpha_{n-i} = \alpha_n - e_1(\gamma) \alpha_{n-1}+\dots + (-1)^M e_{M}(\gamma)\alpha_{n-M} \\
    & = & \alpha_n - \ove_1\alpha_{n-1}+\dots + (-1)^M\ove_M\alpha_{n-M}, 
\end{eqnarray*}
using notation  $\overline{e}_k =e_k(\gamma)$, by analogy with $\overline{h}_k=h_k(\gamma)$. The coefficients of this equation take values in $R_{M,N}$ and they do not depend on $n$. Let 
\begin{equation} \label{eq_K} 
K= K_{M,N} = \max(N+1,M).
\end{equation} 
We can conclude that, for $n\ge K$, element $x^n=[\SS^1_n]$ of $A=\alpha_{M,N}(\SS^1)$ is a linear combination, with coefficients in $R_{M,N}$, of elements of smaller index:
\begin{equation} \label{eq_linear1}
 x^n - e_1(\gamma) x^{n-1}+\dots +(-1)^M e_M(\gamma) x^{n-M}=0. 
\end{equation}
 Consequently,  
\begin{equation} \label{eq_linear2}
 x^{K} - e_1(\gamma) x^{K-1}+\dots +(-1)^M e_M(\gamma) x^{K-M}=0, 
\end{equation}
and the  following relation holds in $A$: 
\begin{equation} \label{eq_linear3}
 (x^M - e_1(\gamma) x^{M-1}+\dots +(-1)^M e_M(\gamma))x^{K-M}=0. 
\end{equation}
Denote by $r_{M,N}$ the left hand side of this equation, 
\begin{equation}\label{eq_r_N_M}
    r_{M,N} =  (x^M - e_1(\gamma) x^{M-1}+\dots +(-1)^M e_M(\gamma))x^{K-M}. 
\end{equation}
The following proposition follows.  
\begin{prop} \label{eq_prop_surj} 
There is a surjective ring homomorphism 
\begin{equation}
    R_{M,N}[x]/(r_{M,N}) \lra A,
\end{equation}
and elements $1,x,\dots, x^{K-1}$ span $A=\alpha_{M,N}(\SS^1)$ as a module over $R_{M,N}$, where $K = \max(N+1,M).$
\end{prop} 

The  inclusion $R_{M,N}\subset R'_{M,N}$ makes $R'_{M,N}$ a free module over $R_{M,N}$ of rank $N!M!$. From this we can conclude that $A'$ is a free $A$-module of rank $N!M!$, with an isomorphism  
$A'\cong A\otimes_{R_{M,N}} R'_{M,N},$
and Proposition~\ref{eq_prop_surj} holds with $R_{M,N}$  replaced by $R'_{M,N}$ and $A$ replaced by $A'=\alpha'_{M,N}(\SS^1)$. 

When $K=M$, that is, $N<M,$ the monomial $x^{K-M}$  in (\ref{eq_linear3}) vanishes and $r_{M,N}$ becomes a monic degree $M$ polynomial in $x$ with coefficients - elementary symmetric functions in $\gamma_i$'s. Informally, one can think of it as a 'generic' monic  degree $M$ polynomial. 

\vspace{0.1in} 

Coefficients of the power series (\ref{eq_Z_rat_1}) are the so-called \emph{complete supersymmetric} functions 
\begin{equation} 
h_n(\gamma/\beta) = \sum_i h_{n-i}(\gamma)e_i(\beta) = \sum_{i=0}^{\min(n,N)} h_{n-i}(\gamma)e_i(\beta).
\end{equation} 
\emph{Supersymmetric} here is a bit of a misnomer. Rather, these  functions  and their generalizations are  characters of certain irreducible representations of the Lie superalgebra $\mathfrak{gl}(M|N)$. We refer the reader to references~\cite{Mo,MJ1} for an introduction to \emph{supersymmetric  Schur} functions  and  to~\cite[Introduction]{MJ2} for a quick summary of formulas for supersymmetric Schur functions as well as a brief history of this subject and references  to  the original papers, see also~\cite{HSP}. Supersymmetric Schur functions are also called \emph{hook Schur} functions~\cite{HSP}. Lascoux~\cite{L} calls them \emph{Schur functions in  difference  of alphabets}; see \"Ozt\"urk-Pragacz~\cite[Section 2]{OP} for another introduction and relation to singularity  theory. 

\vspace{0.07in} 

Consider the  Hankel matrix  $H_K(\gamma/\beta)$ of the  spanning set $\{1,x,\dots, x^{K-1}\}$ of $A(1)$. Its $(i,j)$ entry is $h_{i+j}(\gamma/\beta)$. Multiplying it by the matrix $J_K$ of the longest permutation results in the matrix  
with the $(i,j)$-term $h_{K-1-i+j}(\gamma/\beta).$ Recall the rectangular partition 
\begin{equation}
    \lambda_K = ((K-1)^K) =  (K-1,K-1,\dots,  K-1). 
\end{equation}
We see that the matrix $J_K H_K(\gamma/\beta)$ has the $(i,j)$-entry 
\begin{equation}
h_{K-1-i+j}(\gamma/\beta)= h_{\lambda_K^i-i+j}(\gamma/\beta).
\end{equation}
Consequently, it's the Jacobi-Trudi matrix for the partition $\lambda_K$. 

\begin{prop} 
\begin{equation} 
\det(H_K(\gamma/\beta)) = (-1)^{K(K-1)/2} s_{\lambda_K}(\gamma/\beta) .
\end{equation} 
\end{prop}  
Proposition says that the determinant of $H_K(\gamma/\beta)$ is, up to a sign, the  supersymmetric  Schur function for the partition $\lambda_K$. 
The  Jacobi-Trudi formula for the supersymmetric Schur function can either be taken as the definition of the latter, or derived, if  the supersymmetric Schur function is defined as the character of the corresponding irreducible representation of $\mathfrak{gl}(N|M)$. The Jacobi-Trudi formula for the general supersymmetric Schur function is written down  in  Section~\ref{sec_SP_formula} below, see~(\ref{eq_ss_jt}).  

\vspace{0.1in} 

For the diagram $\lambda_K=((K-1)^K)$ with $K=\max(N+1,M)$ we distinguish two cases
\begin{enumerate}
    \item $N<M$, then $K=M$ and $\lambda_K=((M-1)^M)$.  
    \item $N\ge M$, then $K=N+1$ and $\lambda_K=(N^{N+1})$. 
\end{enumerate}
To understand the  supersymmetric Schur function $s_{\lambda_K}(\gamma/\beta)$ we compare partition $\lambda_K$ to the rectangular partition $(N^M)$, see Figure~\ref{fig_4_1}. In each of these two cases $\lambda_K$ contains the rectangle $M\times N$, with the complement -- itself a rectangular partition. 

\begin{figure}[h]
\begin{center}
\includegraphics[scale=1.0]{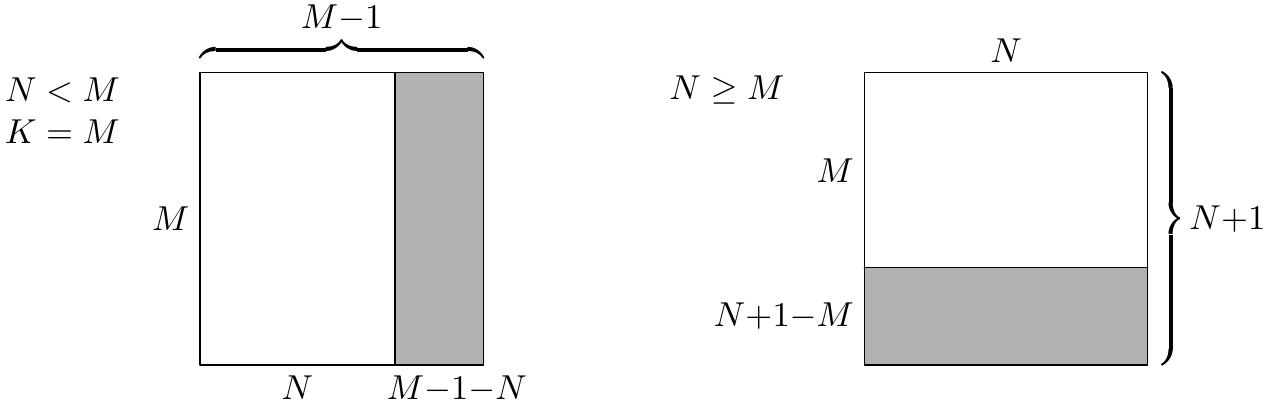}
\caption{Comparing rectangular partitions $\lambda_K$ and $(N^M)$.  Left: $N<M$ case, $\tau=((M-N-1)^M), \eta = \emptyset$. Right: $N\ge M$ case, $\tau=\emptyset, \eta=(N^{(N+1-M)})$. In both cases $(N^M)\subset \lambda_K$, with the complement shaded. For $\tau, \eta$ notation see Section~\ref{sec_SP_formula}.}
\label{fig_4_1}
\end{center}
\end{figure}

When $(M,N)$-supersymmetric partition $\lambda$ contains $M\times N$ rectangle, the supersymmetric  Schur function   simplifies to the product 
\[ s_{\lambda}(\gamma/\beta) = s_{\tau}(\gamma) s_{\eta'}(\beta)\cdot \prod_{i=1}^N\prod_{j=1}^M (\beta_i+\gamma_j)  
\] 
where $\tau$ and   $\eta$ are part of $\lambda$ to the right and down  of the $M\times N$-rectangle. 

Furthermore, since partitions $\eta'$ and  $\tau$  are, in some order, a rectangular  and the empty  partition, see Figure~\ref{fig_4_1}, and  the rectangular  partition is of  the  maximal height for that number of variables, we obtain the following simple  formulas for the supersymmetric Schur function that describes our determinant. 

Case (1): 
\begin{equation}
\label{eq_det_case1}
s_{\lambda}(\gamma/\beta) = (\gamma_1\dots\gamma_M)^{M-N-1}\cdot \prod_{i=1}^N\prod_{j=1}^M (\beta_i+\gamma_j) . 
\end{equation}

Case  (2): 
\begin{equation}
    \label{eq_det_case2}
 s_{\lambda}(\gamma/\beta) =  (\beta_1\dots\beta_N)^{N+1-M}\cdot \prod_{i=1}^N\prod_{j=1}^M (\beta_i+\gamma_j) .
\end{equation}

Over  $R_{M,N}$ these determinants are  not zero and, consequently, there are no linear relations on $1,x,\dots, x^{K-1}$. 
Together  with Proposition~\ref{eq_prop_surj} this implies the following result. 

\begin{theorem}\label{thm_2} 
  The state space of the circle $A(1) =\alpha_{M,N}(\SS^1)$ is a free $R_{M,N}$-module with a basis $\{1,x,\dots, x^{K-1}\}$, where $K = \max(N+1,M).$ Multiplication induced by the pants cobordism turns $A(1)$ into a commutative algebra 
  \begin{equation}\label{eq_monic_x}
      A\cong R_{M,N}[x]/((x^M - \ove_1 x^{M-1}+\dots +(-1)^M \ove_M)x^{K-M}), 
  \end{equation}
  where  $\ove_i=e_i(\gamma)$. 
\end{theorem}

Algebra $A(1)$ is commutative Frobenius over $R_{M,N}$, being  generated by a single element with the monic minimal polynomial over the ground  ring, see~(\ref{eq_monic_x}). 
Nethertheless, the natural trace that $A(1)$ inherits from the topological theory structure is not Frobenius, that is, it does not induce an isomorphism  $A(1)\stackrel{\cong}{\lra}A(1)^{\ast}$, only an injection.  This is  due to the Hankel determinants (\ref{eq_det_case1}) 
and (\ref{eq_det_case2}) not being invertible in $R_{M,N}$. The  same phenomenon appears in Section~\ref{subsec_constant} below, where $A(1)\cong R$, with $\{1\}$ as the basis,  but  $(1,1)=\beta\in  R$, so that the  trace map  $\epsilon(a)=(a,1)$ is unimodular iff  $\beta$  is invertible in $R$. 

\vspace{0.1in}

If  we  specialize  to  a field  $\kk$ via a  $\kk$-linear homomorphism 
\[
\psi:\kk[\gamma_1,\dots,\gamma_M,\beta_1,\dots, \beta_N]\lra \kk
\]
such that  
\begin{equation}\label{eq_psi_no_zero}
\psi(\gamma_j)\not= 0, \ \psi(\beta_i)\not=0,\  \psi(\beta_i+\gamma_j)\not=0, \ 1\le j \le M, \
1\le i \le N, 
\end{equation}
then the resulting  theory over a  field $\kk$ has  no linear relations  on $1,x,\dots, x^{K-1}$ either. This theory has the generating function 
\begin{equation}\label{eq_Z_rat_3} 
    Z_{\psi}(T)= \frac{(1+\psi(\beta_1)T)\dots (1+\psi(\beta_N) T)}{(1-\psi(\gamma_1) T)\dots (1-\psi(\gamma_M) T)} .
\end{equation}
The state space $A_{\psi}(1)$ of the circle in the induced theory is a $\kk$-vector space with a basis  $\{1,x,\dots, x^{K-1}\}$. Over a field the determinants (\ref{eq_det_case1}) and (\ref{eq_det_case2}) are invertible, subject to condition (\ref{eq_psi_no_zero}), and the usual 
trace  $\epsilon(a)=(a,1)$ turns  $A_{\psi}(1)$ into a commutative Frobenius algebra. 

\vspace{0.07in}

\emph{Remark:} The rational function in (\ref{eq_Z_rat_1}) satisfies  $Z(0)=1$. To avoid this restriction, one may change $P(T)$ to either  
\begin{itemize} 
\item  $\beta_0(1+\beta_1 T)\dots (1+\beta_N T)$, where $\beta_0$ is another  formal variable, or 
\item $(\beta'_1 +T)\dots (\beta'_N + T)$, where $\beta'_i$ are formal variables of degree $-2$, opposite  to  that of  $\beta_i$. 
\end{itemize}
The computations above would need to be modified. We leave the details  to the reader. Some special cases are considered below in relation to state spaces $A(k)$ of several circles. 

\vspace{0.1in} 


\subsection{One-circle state space for a  polynomial generating function.} \label{subsec_one_c} 
$\quad$

\vspace{0.07in} 

We now specialize  this setup to $M=0$, that is, $Z(T)$ being a polynomial,  but also change from terms  $1+\beta_i T$ to $T-\beta_i$.
Let the generating function $Z(T)$ be a generic monic polynomial $P(T)$, factoring into the product 
\[ Z(T) = \prod_{i=1}^N(T-\beta_i) = T^N-e_1(\beta)T^{N-1}+\dots + (-1)^N e_N(\beta). 
\] 
The Hankel determinant (or the Gram determinant) for the first $N+1$ vectors $1, x, \dots, x^N$ has the form 
\begin{equation} \label{eq_hankel_poly}
      H_{[0,N]} =  
  \left( {\begin{array}{ccccc}
   (-1)^N e_N & (-1)^{N-1}e_{N-1} & (-1)^{N-2}e_{N-2} & \dots & 1 \\
   (-1)^{N-1}e_{N-1} & (-1)^{N-2}e_{N-2} & \dots & \dots & 0 \\
   (-1)^{N-2}e_{N-2} & \dots & \dots & \dots &  0 \\
   \dots & \dots & \dots & \ddots & \dots \\
   1 & 0 & 0 & \dots & 0 \\
  \end{array} } \right),
\end{equation}
with $e_i = e_i(\beta)$. 
Every antidiagonal element is $1$, and there are only zeros below the antidiagonal, so 
that $\det H_{[0,N]} =(-1)^{\frac{N(N+1)}{2}}.$ Adding any additional element $x^i$ for $i>N$ to this set of vectors results in the matrix which has only zeros in the last row and column. In particular, it has the zero determinant and shows that $x^i=0$ in $A$  for $i\ge N+1$.  

Recall that we use two different ground rings, which are $R_{0,N}$ and $R'_{0,N}$ in  this case:   
\begin{equation}
    R'_{0,N}=\kk[\beta_1,\dots,\beta_N], \  \ 
    R_{0,N}= (R'_{0,N})^{S_N} = 
    \kk[e_1(\beta),\dots, e_N(\beta)],
\end{equation}
with $R_{0,N}$ the ring of symmetric functions in $N$ variables $\beta_1,\dots, \beta_N$. For these  theories we denote the state space of $\SS^1$ by $A$ and $A'$, correspondingly. 

\begin{corollary}  $A\cong R_{0,N}[x]/(x^{N+1})$ and 
$A'\cong R'_{0,N}[x]/(x^{N+1})$. The bilinear pairing on $A=\alpha_{0,N}(\SS^1)$ is unimodular and turns $A$ into a commutative Frobenius algebra over $R_{0,N}$, ditto for $A'$  and  $R'_{0,N}$. Algebra $A$ is a free module of rank $N+1$ over $R_{0,N}$, same for $A'$ and $R'_{0,N}$.   
\end{corollary}    

If keeping track of the grading, we set $\deg(T)=\deg(\beta_i)=-2$, so that $Z(T)$ is homogeneous of  degree $-2N$. Degree of $\beta$ is the  opposite to that in Section~\ref{subsec_semiu}. 

\vspace{0.07in} 

Starting with any commutative ring $R$ and elements $\alpha_0,\dots, \alpha_{N-1}\in R$, we can do the universal construction for the generating function $Z(T)=T^N+\sum_{i=0}^{N-1} \alpha_i T^i$. Then the state space $A_{\alpha}\cong R[x]/(x^{N+1})$ and the trace $\epsilon(x^i)= \alpha_i,$ $i<N$, $\epsilon (x^N)=1$. 

\vspace{0.07in}

The Frobenius extension $R\subset A_{\alpha}$ can be obtained via base change from $R_{0,N}\subset A$, so that 
$A_{\alpha}= R\otimes_{R_{0,N}} A$:  

\[ 
\begin{CD}
R @>>> A_{\alpha} \\
@AAA    @AAA \\
R_{0,N}  @>>> A 
\end{CD}
\] 


\vspace{0.1in}
 
To  write down  the dual basis in $A$ to the monomial basis, consider the upper-triangular $(N+1)\times( N+1)$ matrix $U$, a relative of $H_{[0,N]}$ in equation (\ref{eq_hankel_poly}),  with ones on the main diagonal  and  $(i,j)$-entry $\alpha_{N+i-j}$ for $i<j$.  It's easy to write  down  its inverse matrix $U^{-1}$, and the dual basis to  $\{1,x,\dots, x^{N}\}$ relative  to  the trace map $\epsilon$ can be read off from $U^{-1}$, as the  coefficients of $U^{-1}(1,x,\dots, x^N)^T$. For instance, for $N=1$ matrices $U$ and $U^{-1}$ are  shown below 
\begin{equation}
    U= \begin{bmatrix} 1 & \alpha_0 \\
    0 & 1 \end{bmatrix}, \ \ 
    U^{-1} =\begin{bmatrix} 1 & -\alpha_0 \\
    0 & 1 \end{bmatrix} ,
\end{equation}
and
the dual basis to $\{1,x\}$ is $\{x,1-\alpha_0 x\}$. For $N=2$ 
the matrices are 
\begin{equation}
    U= \begin{bmatrix} 1 & \alpha_1 & \alpha_0 \\
    0 & 1 & \alpha_1  \\
    0  &   0 & 1 \end{bmatrix}, \ \ 
    U^{-1} = \begin{bmatrix} 1 & -\alpha_1 & \alpha_1^2-\alpha_0 \\
    0 & 1 & -\alpha_1  \\
    0  &   0 & 1 \end{bmatrix},
\end{equation}
and the dual basis to $\{1,x,x^2\}$ is  $\{x^2,x-\alpha_1x^2,1-\alpha_1 x  + (\alpha_1^2-\alpha_0)x^2 \}$. 

As we've already mentioned, non-degeneracy  of $\epsilon$ does  not guarantee that the associated topological theory is one-multiplicative, that is, admits  a neck-cutting relation. For example, consider for $N=1$ the difference of  the tube and the sum of basis elements and their duals $1\sqcup x  + x \sqcup  (1-\alpha_0 x)$, see Figure~\ref{fig_4_2}.  

\begin{figure}[h]
\begin{center}
\includegraphics[scale=1]{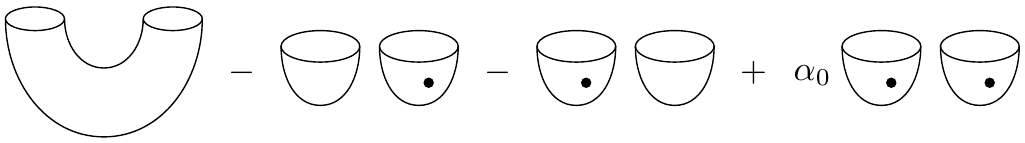}
\caption{A dot is used to denote a handle, see Figure~\ref{fig_3_7} below. This linear combination of cobordisms with  boundary $\SS^1\sqcup  \SS^1$  evaluates to  zero if closed by a cobordisms with different connected components at the two boundary circles.  Closing up by a tube evaluates  to 
$ 1 - 1 - 1 +0 = 1\not= 0$, since in this evaluation, with $Z(T)=\alpha_0+T$, torus  evaluates  to $1$, sphere to $\alpha_0$, and higher genus components  to $0$.}
\label{fig_4_2}
\end{center}
\end{figure}

This  difference
evaluates to zero in  $R_{0,2}$ if we cap off the two boundary circles by a disjoint union of two 2-manifolds, each  with one circle as  the  boundary, but not if  cap off the  boundary  circles by a tube  connecting them. 

\vspace{0.07in}

{\it Non-monic polynomial function:}
For a  minor modification of the above example, suppose that $\alpha_N\not=0$ and $\alpha_n=0$ for $n>N$, and that $R$ is an integral domain. This is still the case of a polynomial function 
\begin{equation} \label{eq_poly_non}
Z(T)=\alpha_N T^N + \alpha_{N-1}T^{N-1}+\dots + \alpha_0, 
\end{equation}
which is not necessarily monic. 
The Hankel matrix $H_{[0,N]}$ has zeros below the antidiagonal and $\alpha_N$ as each antidiagonal entry, so that 
\[ \det H_{[0,N]} = (-1)^{\frac{N(N+1)}{2}}\alpha_N^{N+1}.
\] 
 The determinant is non-zero and there are no $R$-linear relations on $\{1, x, \dots, x^N\}$, while $x^{N+1}$ is in the kernel of the bilinear form. Therefore, $A\cong R[x]/(x^{N+1})$, as  before. Notice that $A$ comes with a trace map $\epsilon$ that takes  values $\alpha_0, \dots, \alpha_{N-1}$ on $1, x, \dots, x^{N-1},$ and a nonzero value $\alpha_N$ on $x^N$. The trace map gives a perfect (unimodular) pairing on $A$ iff $\alpha_N$ is invertible in $R$. 
 If $\alpha_N$ is not invertible (but not zero), the  pairing is non-degenerate on $A$ but not perfect. 

\vspace{0.05in} 

In this example of a polynomial generating function, the parameters modify the trace only, while the algebra structure of $A(1)$ is fixed with $x^{N+1}=0$ and all lower powers of $x$ constituting a basis. 

\vspace{0.1in} 

Dropping the requirement that $R$ is an integral domain  may  result  in  state  spaces not being $R$-projective modules.  For  instance, if $N=0$ in the above non-monic example (\ref{eq_poly_non}), with $\alpha_0\not= 0$ and $\alpha_n=0$ for $n>0$, then $A$ is isomorphic to the subspace of $R$ which is the kernel of the multiplication by $\alpha_0$: 
\[ A \cong \ker(m), \  m: R\lra R, \ m(y) = \alpha_0 y. 
\] 
If $\alpha_0$ is not a zero divisor, $A\cong R$. If $\alpha_0$ is a zero divisor, $A$ may not a projective $R$-module. 

\vspace{0.1in}

\emph{Remark:} To have 1-multiplicativity, or the neck-cutting relation, when $A=A(1)$ is a free $R$-module of rank $n$ and the trace form is unimodular on $A(1)$, a necessary condition is that $\alpha_1=n$,  that  is, the two-torus evaluates to $n$. One  can  see this  by taking the potential neck-cutting formula  decomposing the tube as the sum  $\sum_{i=1}^n x_i\cup y_i$, over basis elements $\{x_1,\dots, x_n\}$ and dual basis elements $\{y_1,\dots, y_n\}$  and then capping off the  two boundary circles by a tube.  Terms  $x_i\cup y_i$ each  evaluate to $1$  and $x_i \cup y_j$ to $0$ for $i\not=j$, implying that the 2-torus must  evaluate  to $n$. 

\begin{figure}[h]
\begin{center}
\includegraphics[scale=1]{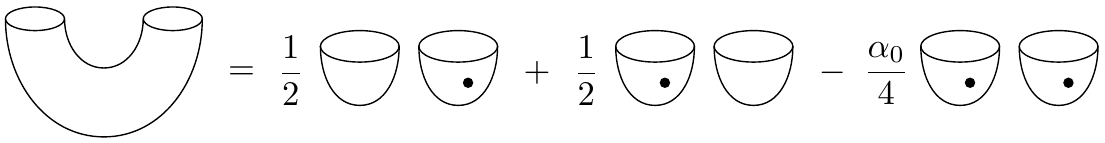}
\caption{Neck-cutting relation when $Z(T)=\alpha_0+2T$ and $2$ is invertible in $R$. Dot denotes a handle.}
\label{fig_4_3}
\end{center}
\end{figure}

For example, take $N=1$ in (\ref{eq_poly_non}) and further specialize to $Z(T)=\alpha_0 + 2T$. Assuming that $2$ is invertible in $R$, the basis $\{1,x\}$ of $A(1)$ has the dual basis 
$\{ x/2,(2-\alpha_0 x)/4\}$.
This theory is 1-multiplicative, the natural map $A(1)^{\otimes k}\lra A(k)$ is an isomorphism for all $k$, and the neck-cutting relation is shown in Figure~\ref{fig_4_3}. For a deformation of this example see  Section~\ref{subsub_rank_2}. 


\subsection{State spaces for unions of circles} \label{subset_union_circle}
$\quad$
\vspace{0.07in}

{\it Monoid $\Cobk$:} For an upper bound on the size of state spaces $A(k)$ of the theories for various  $\undalpha$ consider the commutative monoid $\Cobk$ of all oriented surfaces $S$ with boundary the union  $\sqcup_k \SS^1$ of $k$ circles such that every component of $S$ has nonempty boundary, under the usual multiplication via the $k$-pants cobordism. The unit element of $\Cobk$ consists of $k$ discs, visualized as $k$ cups with boundary  $\sqcup_k \SS^1.$

\vspace{0.1in} 

For small values of $k$ commutative monoid $\Cobk$ is given by 
\begin{itemize}
    \item $\Cobkk{0}= 1$, with the empty cobordism as the unique (and unit) element. 
    \item $\Cobkk{1}=\langle x \rangle$, the  monoid of non-negative powers of  $x$; generator $x$ is given by the 2-torus with one boundary component (one-holed 2-torus), earlier denoted $\SS^1_1$, also see Figure~\ref{fig_ell_cob3}.
    \item $\Cobkk{2}=\langle x_1,x_2,y\rangle /(y^2=x_1y,y^2=x_2y)$. Generators $x_1,x_2$ are given by the 1-holed 2-torus $\SS^1_1$ bounding the first, resp. second circle and the disk bounding the other circle. Generator $y$ is the tube cobordism $1_{tu}$ in Figure~\ref{fig_ell_cob2}.
\end{itemize}
Notice that generators $x_1,x_2$ of $\Cobkk{2}$ have two  connected components each, while $y$ has one. The square of $y=1_{tu}$ is the tube cobordism with a handle. That handle can be positioned near either of the two circles, giving the defining relations above. Figure~\ref{fig_ell_cob4} cobordisms are $x_1^2x_2,$ $y$, and $yx_1^2=yx_1x_2=yx_2^2=y^3$ left to right.

\begin{prop} \label{prop_monoid_rel}
For any $k$, commutative monoid $\Cobk$ has generators $x_i$, $1\le i\le k$, $y_{ij}$, 
$1\le i < j \le k$ and defining relations
\begin{eqnarray*}
    y_{i,j}^2 & = & y_{i,j}x_i = y_{i,j}x_j,  \\
    y_{i,j}y_{j,k} & = & y_{i,j}y_{i,k}  = y_{i,k}y_{j,k}, \ i<j<k.
\end{eqnarray*}
\end{prop} 
Here $x_i$ denotes the cobordism which is the union of $k-1$ disks and one-holed two-torus bounding the $i$-th circle out of $k$. Cobordism $y_{ij}$ consists of $k-2$ disks and  a tube, the later connecting $i$-th and $j$-th circles.  

\emph{Proof} is left to the reader. We did not list commutativity relations on the generators. 
$\square$

\vspace{0.07in} 

More  generally, for a subset $I=(i_1,\dots,i_m)$ define
\[ y_I =y_{i_1,i_2}y_{i_2,i_3}\dots y_{i_{m-1}i_m}.
\] 
This element is an $m$-punctured  2-sphere  bounding $m$ boundary circles $i_1,i_2,\dots, i_m$ together with $k-m$ disks bounding the remaining circles, see examples in Figure~\ref{fig_3_4}. 



Multiplying $y_I$ by $x_i^g$ for any $i\in  I$ gives a genus $g$ surface bounding the same $m$ circles, together with the disks  for the other   circles. 
Elements of $\Cobk$ are parametrized by decompositions of the $k$-element set into non-empty subsets
\[ \{1,2,\dots, k\}=I_1\sqcup I_2\sqcup \dots \sqcup I_m 
\]
together with a choice of  a non-negative number $n_j$, $1\le j\le m$  for each subset (genus of  the corresponding component). The corresponding element of the monoid is 
\[ y_{I_1}x_{\ell_1}^{n_1} y_{I_2}x_{\ell_2}^{n_2}\dots 
y_{I_m}x_{\ell_m}^{n_m},
\] 
where  $\ell_j$ is any element of $I_j$. 

For instance, for $k=2$ there are two decompositions and corresponding elements: 
\begin{itemize}
    \item $\{1,2\}= \{1\}\sqcup  \{2\}$, with elements $x_1^{\ell_1}x_2^{\ell_2}$, $\ell_1,\ell_2\ge 0$, 
    \item $\{1,2\}=\{1,2\}$, with elements $x_1^{\ell}y_{12}=x_2^{\ell}y_{12}=y_{12}^{\ell+1}$, $\ell\ge 0$.
\end{itemize}

\vspace{0.07in}

Define $R$-algebra $A_k$ as the monoid algebra $R\mathrm{Cob}_{2,k}$ with coefficients in $R$,
\begin{equation}\label{eq_def_Ak}
    A_k = R\mathrm{Cob}_{2,k}.
\end{equation}

As a free $R$-module, $A_k$ has a basis given by cobordisms $S$ with boundary $\sqcup_k \SS^1$ such that each connected  component of $S$ has nonempty boundary. Multiplication in $A_k$ is given by putting two cobordisms in parallel and merging $2k$ boundary circles into $k$ circles via $k$ pants cobordisms.

Repeating our earlier monoid examples, for small values of $k$ the algebra $A_k$ is given by 
\begin{itemize}
    \item $A_0= R$, with the empty cobordism as the unique basis element. 
    \item $A_1=R[x],$ the generator $x$ given by the 2-torus with one boundary component, earlier denoted $\SS^1_1$.
    \item $A_2\cong R[x_1,x_2,y]/(y^2-x_1y,y^2-x_2y)$. Generators $x_1,x_2$ are given by the 2-torus $\SS^1_1$ bounding the first, resp. second circle and the disk bounding the other circle. Generator $y$ is the tube cobordism $1_{tu}$. 
\end{itemize}

\begin{prop} \label{prop_rel_Ak} 
Commutative $R$-algebra $A_k$ has generators $x_i$, $1\le i\le k$, $y_{i,j}$, $1\le i< j\le k$ and defining relations 
\begin{eqnarray*}
    y_{i,j}^2 & = & y_{i,j}x_i = y_{i,j}x_j,  \\
    y_{i,j}y_{j,k} & = & y_{i,j}y_{i,k}  = y_{i,k}y_{j,k}, \ i<j<k.
\end{eqnarray*}
\end{prop}
This follows at once from Proposition~\ref{prop_monoid_rel}. 
Subalgebra inclusion $R[x_1,\dots, x_k]\subset A_k$ is split, via the two-sided ideal $J=(y_{i,j})_{i<j}$ in $A_k$, with $A_k\cong R[x_1,\dots, x_k]\oplus J$.

\vspace{0.1in} 

{\it Grading:} 
Give a surface $S$ with boundary the union of $k$ circles and no closed components degree  
\begin{equation} \label{eq_degree_cob} \deg(S)=k-\chi(S).
\end{equation} 
This degree is a non-negative even integer for each $S$ and equips monoid  
$\mathrm{Cob}_{2,k}$ with a  $2\Z_+$ grading. In particular,
\begin{equation} \label{eq_degree_cob_2} \deg(1)=0,   \ \deg(x_i)=\deg(y_{ij}) = 2. 
\end{equation}
If $R$ is $\Z$-graded, resp. $\Z_+$-graded, then algebra $A_k$ is $\Z$-graded, resp. $\Z_+$-graded as well. 

To make gradings of $\mathrm{Cob}_{2,k}$ and $R$  compatible (if $R$ is graded), we want $\deg(\alpha_g)=2g-2$, for $g\ge 0$. Then the degree formula  (\ref{eq_degree_cob}) 
extends to cobordisms $S$ with closed components, when these cobordisms are viewed as elements of $A_k$. With these assumptions, degree  of the 2-sphere is $-2$, of  the 2-torus is zero, while all higher genus components have positive degree. In this sense, 'most' generators of $A_k$ have positive degree (generators $x_i$, $y_{ij}$ and $\alpha_g$ for $g>1$, when we consider $R$ as generated by $\alpha$'s over some smaller commutative base ring), with the exception of $\alpha_0$ and  $\alpha_1$,  of degrees $-2$ and $0$, respectively. 

\vspace{0.1in} 

{\it Multiplications and surjections:} There are algebra homomorphisms 
\begin{equation} \label{eq_algA_homs}
A_{k_1}\otimes_R A_{k_2}\lra A_{k_1+k_2} 
\end{equation}
given by putting cobordisms in  parallel. These homomorphisms are inclusions but not isomorphisms for $k_1,k_2>0$, missing cobordisms with a connected component having a boundary circle among the first $k_1$ and the last $k_2$ circles. For instance, $y_{k_1,k_1+1}$ is not in the image of the homomorphism. 

For any $\undalpha$ as in Section~\ref{subset_state_circle},
commutative algebra $A(k)$ generated by cobordisms with boundary the union of $k$ circles modulo the relations given by $\undalpha$ is naturally a quotient of $A_k$, via 
the homomorphism 
\begin{equation} \label{eq_As_map}
    A_k \lra A(k)
\end{equation}
sending any cobordism $S$ in $\Cobk$ to itself, viewed as an element of $A(k)$, and extending $R$-linearly. 
These surjections intertwine multiplications (\ref{eq_algA_homs})  with the corresponding maps (\ref{eq_hom_parallel}) for $A(k)$'s:

\[ 
\begin{CD}
A_{k_1}\otimes_R A_{k_2} @>>> A_{k_1+k_2}  \\
@VVV    @VVV \\
A(k_1)\otimes_R A(k_2)  @>>> A(k_1+k_2) .
\end{CD}
\] 
\vspace{0.1in}

{\it A finite-dimensional non-multiplicative example.}
Analogously to the universal setup above, 
homomorphisms (\ref{eq_hom_parallel}) may not be isomorphisms even when $A(k)$ have finite rank over $R$.
Let us give such an example with $R$ an integral domain and the generating function 
\begin{equation} \label{eq_linear_fun}
Z(T)=\beta_0+\beta_1 T,
\end{equation}
where $\beta_0,\beta_1\in R$, $\beta_1\not=0$. 

The ring $A(1)\cong  R[x]/(x^2)$, where $x=\SS^1_1$ is the one-holed 2-torus. Indeed, the Gram matrix of $(1,x)$ has determinant $-\beta_1^2$ while  $x^2=0$ in this  theory.

\begin{figure}[h]
\begin{center}
\includegraphics[scale=1.0]{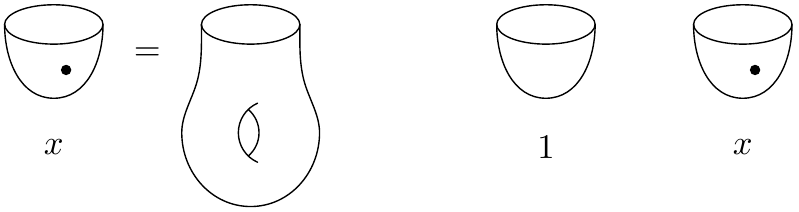}
\caption{Left: dot notation for a handle on a surface; right: basis  $\{1,x\}$ of $A(1)$.}
\label{fig_3_7}
\end{center}
\end{figure}

To simplify diagrams, we use a dot to denote a handle added to a surface, see Figure~\ref{fig_3_7} left. 
A basis of $A=A(1)$ is given by a cup and a dotted cup, see Figure~\ref{fig_3_7} right. In this theory, two dots on the same connected component evaluate to zero, since any closed connected surface of genus two or higher evaluates to zero. Diagrams in Figure~\ref{fig_3_8} constitute a spanning set of $A(2)=\alpha(\SS^1\sqcup \SS^1)$. They  can  be  written as 
\[ 1, \ x_1, \ x_2, \ x_1x_2,  \ y, \ x_1 y, 
\] 
with the understanding that we  mean the image of the  corresponding element  of $A_2$ under the  homomorphism (\ref{eq_As_map}). 

\begin{figure}[h]
\begin{center}
\includegraphics[scale=1.0]{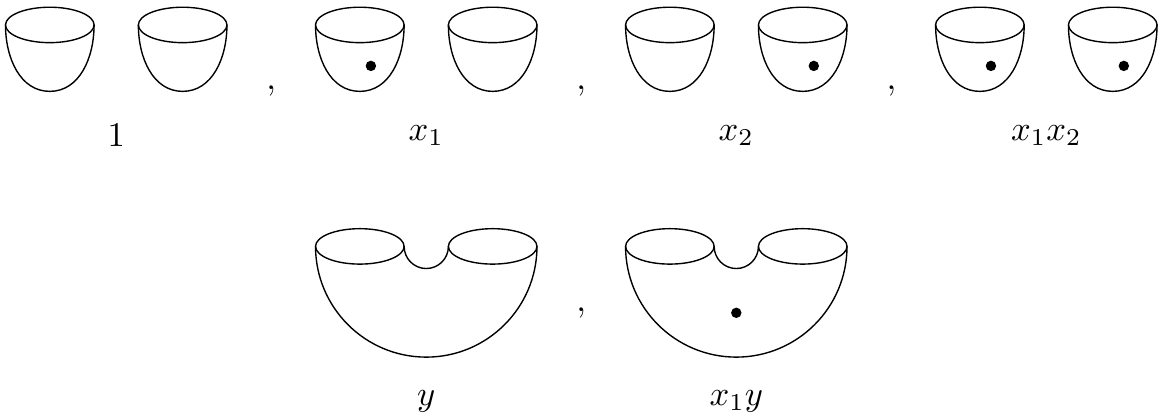}
\caption{Spanning  set for $A(2)$.}
\label{fig_3_8}
\end{center}
\end{figure}

The following is the Gram matrix for this  spanning set: 
\begin{equation}
    G =\begin{bmatrix} 
     \beta_0^2 & \beta_0\beta_1 & \beta_0\beta_1 & \beta_1^2 & \beta_0 & \beta_1 \\
     \beta_0\beta_1 & 0     & \beta_1^2 & 0 & \beta_1 & 0\\
     \beta_0\beta_1 & \beta_1^2 & 0     & 0 & \beta_1 & 0 \\
     \beta_1^2 & 0 & 0 & 0     & 0 & 0 \\
     \beta_0   & \beta_1     & \beta_1     & 0     & \beta_1 & 0 \\
     \beta_1   & 0     & 0     & 0     & 0 & 0 
     \end{bmatrix}  . 
\end{equation}
We see that 
\begin{equation} \label{eq_xx_beta}
x_1 x_2 - \beta_1 x_1 y = 0 
\end{equation} 
in $A(2)$, since the fourth column is $\beta_1$ times the last column. Removing  $x_1x_2$ from the above list of six vectors and downsizing to $5\times 5$ matrix by removing  the fourth row and column yields determinant $\beta_1^6(\beta_1-2).$ In particular, if $\beta_1-2$ is  not a zero divisor, the set $\{1,x_1,x_2,y,x_1y\}$ is a basis of $A(2)$, and the latter is a free $R$-module of rank five. 

If $\beta_1=2$, relations
\[ x_1 x_2 - 2 x_1 y = 0, \ x_1+x_2-2y -\beta_0 x_1 y = 0 
\] 
hold. The first relation is  specialization of (\ref{eq_xx_beta}) for $\beta_1=2$. The second relation, in fact, implies the  first  relation. In this case $A(2)$ is a free $R$-module of rank four with a basis, for instance, $\{ 1, x_1, y, x_1 y\}$; this  is the theory considered at the end of Section~\ref{subsec_one_c}, with parameter $\alpha_0$ there relabeled into $\beta_0$ here. 

\vspace{0.1in}

It's an interesting question to find bases in $A(k)$ for this theory and for more general theories with the polynomial and  rational generating functions. In Section~\ref{subsec_constant} below we discuss but don't fully work out the case of the constant generating function. 

\vspace{0.08in}

{\it Finite generation of $A(k)$:} As a first step in the  classification of theories for various $\undalpha$'s, we can separate them into those with finitely- versus infinitely-generated $R$-module $A(1)$. What does being in the first class entail for $A(k)$ for $k>1$? 

\begin{prop} Let  $R$ be any  commutative ring. Assume that $A(1)$ is a finitely-generated $R$-module. Then $A(k)$ is finitely-generated as well, for all $k>1$. 
\end{prop} 

\begin{proof} Choose finitely many  generators $u_1,\dots, u_m$ of $R$-module $A(1)$. Any power of $x\in  A(1)$ can be written as a linear combination of these generators. Interpreting this topologically, a one-holed surface of any genus reduces to a linear combination of $u_i$'s in $A(1)$.

$R$-module $A(k)$ is spanned by  cobordisms $S$ in $\Cobk$. Any component $S'$ of $S$ has some genus and  some number $j$ of boundary components. It can be reduced in $A(j)$ to a linear combination of the genus zero cobordism $S_0'$ with the same set of boundary components as $S'$ and $u_i$'s written on the cobordism. By writing $u_i$ on the  cobordism $S_0'$ we mean multiplying with $u_i$. Since $S_0'$ is connected, it's not important at which circle the multiplication happened. 

Equivalently, if $u=\sum_{\ell} a_{\ell} x^{\ell}$, multiplication of a connected cobordism $L$ by $u$ equals the 
sum 
$\sum_{\ell} a_{\ell} L_{\ell}$, where by  $L_{\ell}$ we mean $L$ with $\ell$ handles attached at its unique connected component. To interpret  this  operation as multiplication in $A(j)$ we need $L$ to have at least one boundary component, so $j>0$. 
For $S\in \Cobk$, viewed as an element of $A(k)$, any connected component reduces to a linear combination of the corresponding genus zero surface with one of the $u_{\ell}$'s floating  on it. This gives us a spanning set that runs over all decompositions $I$ of $\{1,2,\dots, k\}$. For a decomposition $I$ with $t$ parts, there are $m^{t}$ choices for putting $u_i$'s on connected components. Consequently, there's a finite spanning set in $A(k)$ parametrized by such decompositions and choices of labels $u_i$ for each connected component in the decomposition. 
\end{proof}

For example, with assumptions and notations as in the above proof, $R$-module $A(2)$ admits a set of $m^2+m$ generators, with $m^2$ generators corresponding to disconnected cobordisms and $m$ generators for the cobordism $y=1_{tu}$ with a label $u_i$ floating on it. 

\vspace{0.1in}


\subsection{Example of the constant generating function} \label{subsec_constant}

$\quad$ 

\vspace{0.07in} 

Consider the generating function which is just a constant,  
\begin{equation}\label{eq_Z_const} 
Z(T)=\beta,
\end{equation} 
for $\beta\in R$. Let's assume that $R$ is an integral domain and $\beta\not= 0$. The theory evaluates a 2-sphere to $\beta$ and all higher genus surfaces to zero, 
see Figure~\ref{fig_3_1}. 

\begin{figure}[h]
\begin{center}
\includegraphics[scale=1.0]{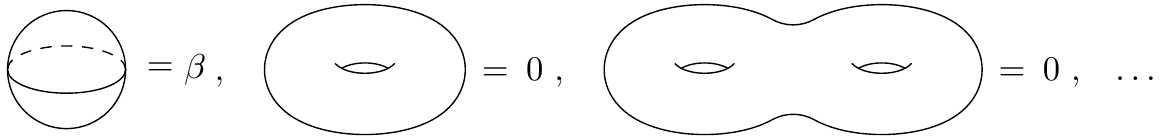}
\caption{Values of $\alpha$ on closed connected cobordisms for the constant generating function $Z(T)=\beta$.}
\label{fig_3_1}
\end{center}
\end{figure}

The  space $A(1)\cong R[\SS^1_0]$ is generated by the disk cobordism $\SS^1_0$, but the bilinear form on it is unimodular only if $\beta$ is invertible in $R$. 

A closed surface that  contains a handle will evaluate to zero. 
Consequently, space $A(k)$ is spanned by cobordisms into  $k$ circles 
with all components  of zero  genus, that is, 2-spheres with holes. 
Homeomorphism classes of such cobordisms may be parametrized by decompositions of $k$; denote this set by $D(k)$. 

\vspace{0.07in} 

Take the algebra $A_k$ 
and mod out by the ideal generated by cobordisms with a handle. Denote the resulting algebra by $A_{0,k}$. It's a free $R$-module with a basis given by decompositions of $k$. A basis element $y_{\lambda}$ corresponds to $\lambda\in D(k)$ being a partition of the set $\{1,\dots, k\}$ into a disjoint union of non-empty subsets. 
The unit element of $A_{0,k}$ is in the basis, as the decomposition $\{\{1\},\{2\},\dots, \{k\}\}$ into one-element sets. 

\vspace{0.07in} 

We can identify $A_{0,k}$ with the quotient of the algebra $A_k$ in Section~\ref{subset_union_circle} of all cobordisms with that boundary by the ideal $(x_1,\dots, x_k)$. 
We can write basis elements of $A_{0,k}$ as products $y_{J_1}\dots y_{J_m}$ over all decompositions $J_1\sqcup  J_2\sqcup \dots \sqcup J_m$ of $\{1,\dots, k\}$. One-element subsets $\{\ell\}$ can be dropped from the product, since $y_{\{\ell\}}=1$. 

\vspace{0.1in} 

It's convenient to assume that $R$ is graded as well, with $\deg(\beta)=-2$. This extends the convention on the degree of cobordisms in $A_k$ to closed cobordisms. Note that $A_{0,k}$ is $\Z_+$-graded (if tacitly assuming that $R$ lives entirely in degree $0$),  but with $\deg(\beta)=-2$ the ring $A_{0,k}$ has both positive and negative terms. 

\vspace{0.1in} 

Algebra $A_{0,k}$ surjects onto the  algebra $A(k)$ for our theory with the  generating  function (\ref{eq_Z_const}). Basis of $A_{0,k}$ produces a spanning set of $A(k)$. 
The spanning sets of $A(2)$ and $A(3)$ are shown in Figures~\ref{fig_3_2} and~\ref{fig_3_3}, respectively. 

\begin{figure}[h]
\begin{center}
\includegraphics[scale=1.0]{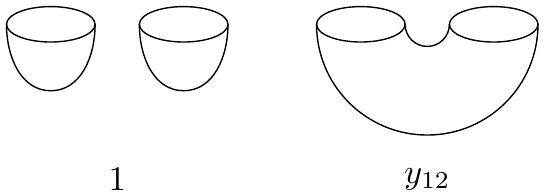}
\caption{Basis of $A(2)$: unit cobordism $1$ and tube cobordism $y_{12}=1_{tu}$.}
\label{fig_3_2}
\end{center}
\end{figure}

\begin{figure}[h]
\begin{center}
\includegraphics[scale=1.0]{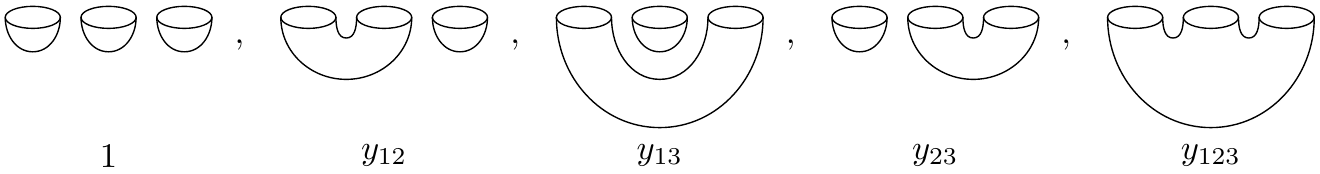}
\caption{A spanning set of $A(3)$: the unit elements, tubes $y_{ij}$, and connected cobordism $y_{123}=y_{12}y_{13}$.}
\label{fig_3_3}
\end{center}
\end{figure}

Gram matrices for these spanning sets are 
\begin{equation}
    G_2= \begin{bmatrix} \beta^2 & \beta \\
    \beta & 0 \end{bmatrix}, \ \ 
    G_3 =\begin{bmatrix} 
     \beta^3 & \beta^2 & \beta^2 & \beta^2 & \beta \\
     \beta^2 & 0     & \beta & \beta & 0 \\
     \beta^2 & \beta & 0     & \beta & 0 \\
     \beta^2 & \beta & \beta & 0     & 0 \\
     \beta   & 0     & 0     & 0     & 0 
     \end{bmatrix}  . 
\end{equation}
Determinant of the  first matrix is $-\beta^2$. Since $\beta\not=0$ and $R$ is an integral domain, $A(2)$ is a free $R$-module with basis $\{1,y_{12}\}$. 

The second matrix has determinant $2\beta^5$. 
\begin{itemize}
\item If $2\not=0$ in $R$, the five-element set in Figure~\ref{fig_3_3} is a basis of the free module $A(3)$. 
\item 
If $2=0$ in $R$, there is a linear relation
\begin{equation}
y_{12}+y_{13}+y_{23}+\beta y_{123}=0
\end{equation}
in $A(3)$, which can be thought of as a skein relation on genus zero surfaces with three boundary components  in this theory. Then $A(3)$ is a free $R$-module with a basis, for instance, $\{1,y_{12},y_{13},y_{123}\}.$ 
\end{itemize}

If $2=0$ in $R$ and $\beta$ is invertible, we can instead write
\begin{equation}
y_{123}=\beta^{-1}(y_{12}+y_{13}+y_{23}) 
\end{equation}
so that a sphere with 3 holes simplifies to a linear combination of unions of a tube and a disk. This allows to simplify any  $y_J$ with $J$ of cardinality three or higher. In this case ($2=0$ and invertible $\beta$) space $A(k)$ is spanned by products $y_{i_1,j_1}y_{i_2,j_2}\dots y_{i_m,j_m}$ such that $i_{\ell}<j_{\ell}$, $\ell =1,\dots, m$, sequence  $(i_1,\dots, i_m)$ is strictly increasing, and the $2m$ indices are all distinct. 
Diagrammatically, the spanning set consists of diagrams of disks and tubes, that is, no component containing more than two boundary circles. 

\begin{figure}[h]
\begin{center}
\includegraphics[scale=1.0]{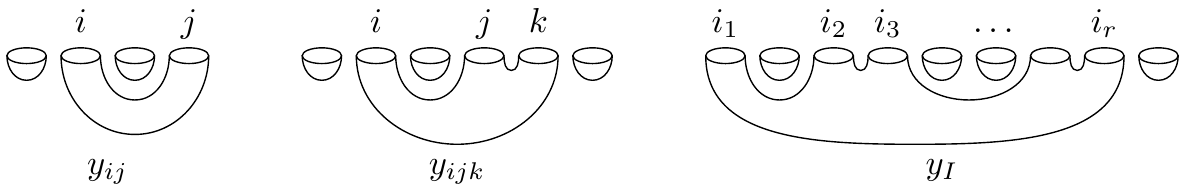}
\caption{Elements $y_{ij},y_{ijk}$ and $y_I$, for $I=\{i_1,\dots, i_r\}$ of $A(k)$.}
\label{fig_3_4}
\end{center}
\end{figure}

\begin{figure}[h]
\begin{center}
\includegraphics[scale=1.0]{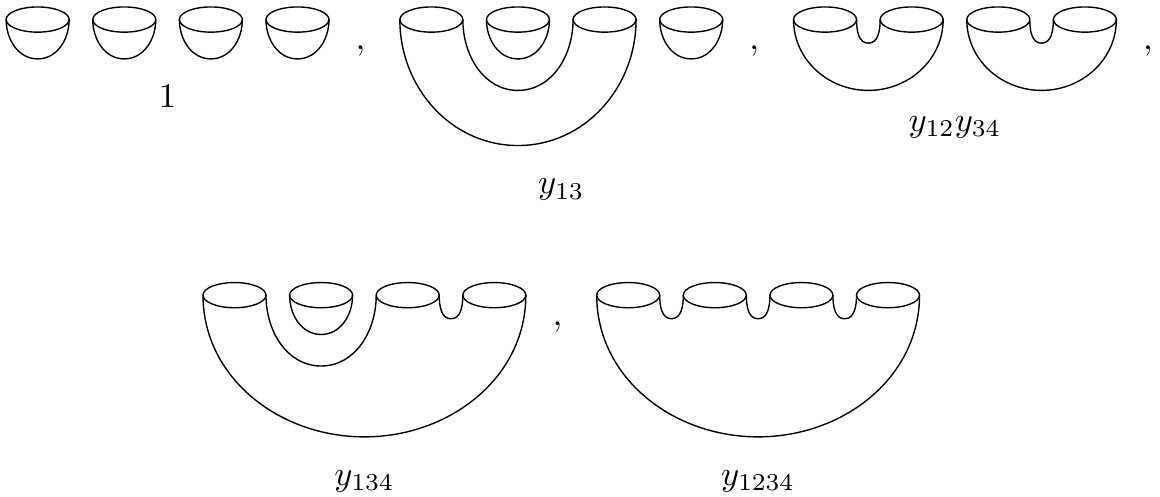}
\caption{Elements $1, y_{1234}$ and examples of elements of types $y_{ij},$ $y_{ij}y_{kl}$, $y_{ijk}$ in $A(4)$.}
\label{fig_3_5}
\end{center}
\end{figure}

We can separate elements of the spanning set for $A(4)$ into five types, see Figure~\ref{fig_3_4}.
\begin{enumerate}
    \item Unit cobordism $1$. 
    \item Six cobordisms $y_{ij}$, $i<j$. 
    \item Three cobordisms $y_{ij}y_{kl}$ with $i,j,k,l$ distinct: $y_{12}y_{34},y_{13}y_{24},y_{14}y_{23},$ each a disjoint union of two tubes. 
    \item Four cobordisms $y_{ijk}$, each a union of a 3-holed sphere and a disk. 
    \item Cobordism $y_{1234}$, which is a 4-holed sphere.  
\end{enumerate}
There are seven elements of types 1-2 and eight elements of types 3-5. 
The inner product between any two elements of types 3-5 is zero due to handle presence. Since there are 8 elements of these types and 7 elements of types 1-2, there is at least one linear relation on 8 elements of types 3-5. We can look for such a relation on linear combinations that are invariant under the permutation action of $S_4$ on the four boundary circles. This quickly yields the following relation
\begin{equation} \label{eq_eight_els}
 (y_{12}y_{34}+y_{13}y_{24}+y_{14}y_{23})-(y_{123}+y_{124}+y_{134}+y_{234}) +\beta y_{1234} = 0 
\end{equation}

To see whether (\ref{eq_eight_els}) is the only relation, we write down part of the matrix of the bilinear form, with rows labelled by the eight vectors in (\ref{eq_eight_els}) and columns by the seven vectors  $\{1,y_{12},y_{13},y_{14},y_{23},y_{24},y_{34}\}$ of types 1 and 2: 
\begin{equation}
    G'_4 = \beta \begin{bmatrix} 
    \beta  & 0 & 1 & 1 & 1 & 1 & 0 \\
    \beta  & 1 & 0 & 1 & 1 & 0 & 1 \\
     \beta & 1 & 1 & 0 & 0 & 1 & 1 \\
     \beta & 0 & 0 & 1 & 0 & 1 & 1 \\
     \beta & 0 & 1 & 0 & 1 & 0 & 1 \\
     \beta & 1 & 0 & 0 & 1 & 1 & 0 \\
     \beta & 1 & 1 & 1 & 0 & 0 & 0 \\
     1     & 0 & 0 & 0 & 0 & 0 & 0
     \end{bmatrix}  . 
\end{equation}
We took the  common multiple $\beta$ of all entries out to the front of the matrix. Relation (\ref{eq_eight_els}) is  the linear relation  on the rows of this matrix:  the sum  of the first  three  rows minus the sum of the next four rows plus $\beta$ times  the last  row  is the  zero  vector. 

Remove one of the first  seven  rows to get
a square $7\times 7$ matrix $G''_4$ that describes part of the  bilinear form on the corresponding spanning set of 14 vectors. It has determinant $8\beta^7$ (having  removed row $7$). The entire $14\times 14$ matrix of the bilinear form on this spanning set is the  block matrix 
\begin{equation*}
    \begin{bmatrix} \ast & (G''_4)^T \\
    G''_4 & 0 \end{bmatrix}. 
\end{equation*}
with determinant $-64\beta^{14}$. 
Consequently, If $2\not= 0$ in $R$, these 14 elements (the above 15 elements without  $y_{234}$) constitute  a  basis in the free $R$-module $A(4).$

If $\beta$ is invertible, we  can instead express $y_{1234}$ as the linear combination of the other  seven elements  in  (\ref{eq_eight_els}). 

\vspace{0.07in} 
  
\begin{table}
\centering
 \begin{tabular}{| c | c | c | } 
 \hline 
  $k$  & rank  & graded rank  \\ [0.5ex] 
 \hline 
  0 & 1 & 1 \rule{0pt}{2.5ex}\\ [0.5ex] 
 \hline
  1 \rule{0pt}{2.5ex}& 1 & 1 \\
 [0.5ex] 
 \hline
  2 \rule{0pt}{2.5ex} & 2 & $1+q^2$ \\
  [0.5ex] 
 \hline
  3 \rule{0pt}{2.5ex} & 5 & $1 + 3q^2 + q^4$ \\
   [0.5ex] 
 \hline
 4 \rule{0pt}{2.5ex} & 14 & $1 + 6q^2 + 6 q^4 + q^6$ \\
 [0.7ex] 
 \hline
\end{tabular}
\vspace{0.2cm} 
\caption{Ranks and graded ranks of $R$-modules $A(k)$, $k\le 4$, when $2$ is invertible in $R$.}
\label{tab_ranks}
\end{table}

Assuming that $2$ is invertible in $R$, we obtain the list of ranks and graded ranks of $A(k)$ as free $R$-modules for $k\le 4$ in Table~\ref{tab_ranks}. 
To make sense of graded dimensions, it's convenient to assume that $R$ is non-positively graded, with $\deg(\beta)=-2$. Elements $y_J$, see Figure~\ref{fig_3_4}, will be positively graded, though, and graded ranks of $A(k)$ have non-negative powers of $q$ in its monomial terms. Rings $A(k)$ are $2\Z$-graded, having  elements of both positive and negative even degrees.  

The sequence  of degrees up  to $k=4$ matches the sequence of Catalan numbers.  The sequence of $q$-degrees matches a particular refinement of Catalan  numbers into polynomials in $q^2$. The  $n$-th Catalan number counts shortest paths in the square lattice from $(0,0)$ to $(n,n)$ that don't go below the main  diagonal. Counting a path with the  coefficient $q^{2r}$, where $r$ is the  number of right-to-up turns in the path results in the polynomials in the rightmost column of Table~\ref{tab_ranks}.

\begin{table}
\centering
 \begin{tabular}{| c | c | c | } 
 \hline 
  $k$  & rank  & graded rank  \\ [0.5ex] 
 \hline 
  5 \rule{0pt}{2.5ex} & 42 & $1+10q^2+20q^4+10q^6+q^8$ \\
  [0.5ex] 
 \hline
  6 \rule{0pt}{2.5ex} & 132 & $1 + 15 q^2 +50q^4 + 50q^6+15q^8+q^{10}$ \\
   [0.5ex] 
 \hline
 7 \rule{0pt}{2.5ex} & 429 & $1 + 21 q^2 + 105 q^4 + 175 q^6+105q^8 + 21 q^{10}+q^{12}$ \\
 [0.7ex] 
 \hline
\end{tabular}
\vspace{0.2cm} 
\caption{Data from~\cite{Ko} on ranks and graded ranks of $R$-modules $A(k)$, $k=5,6,7$, when $2$ is invertible in $R$}
\label{tab_ranks_2}
\end{table} 

Yakov Kononov pointed  out to the author that this $q$-refinement of Catalan numbers into polynomials in $q^2$ is known as the Narayana numbers deformation. The latter count the number of Dyck paths with exactly $r$ right-to-up turns. Kononov also ran computations and confirmed~\cite{Ko} that $A(k)$ are free graded $R$-modules of graded ranks $\sum_{r=0}^{k-1} N(k,r+1)q^{2r}$ for $k=5,6,7$ as well, where $N(k,r+1)$ is the Narayana number, assuming $\Q\subset R$. His results are shown in Table~\ref{tab_ranks_2}. We plan to prove~\cite{KKo} that the pattern holds for all $k$ if $R$ contains $\Q$. 

\vspace{0.07in}

 Recall algebras $A_{0,k}$ defined earlier in this section as the quotients of $A_k$ by the ideal generated by cobordisms which have a component of degree greater that zero. We see that surjective homomorphism
\begin{equation}\label{eq_Ak_hom}
    A_{0,k} \lra A(k) 
\end{equation}
stops being an isomorphism for $k=4$, due to the  relation (\ref{eq_eight_els}), which holds in $A(4)$ for our theory with $Z(T)=\beta$ but not in $A_{0,4}$. In particular, $A_{0,4}$ is a free $R$-module of rank 15, not 14, and the left hand side of (\ref{eq_eight_els}) belongs  to the kernel of the homomorphism (\ref{eq_Ak_hom}) for $k=4.$

\vspace{0.1in} 

{\it Truncation of a fully multiplicative theory:}
Already  this example, of the constant generating function  $Z(T)=\beta$, leads to a non-trivial theory, with interesting skein relations on genus zero cobordisms with four boundary circles, even for invertible $\beta$. While the space $A(1)$ is one-dimensional, spaces $A(k)$ are bigger than $A(1)^{\otimes k}$, as we've just seen. For invertible $\beta$, the theory $Z(T)=\beta$ is essentially the simplest truncation  of the theory with the generating function
\begin{equation}\label{eq_Z1_function} 
Z_1(T)=\frac{\beta}{1-\beta^{-1}T}=\beta+ T + \beta^{-1} T^2 + \dots \ .
\end{equation}
The latter theory evaluates a closed connected surface $S$ of genus $g$ and Euler characteristic $2-2g$ to $\beta^{1-g}=\beta^{\chi(S)/2}$. The state space $A(1)$ in this theory has rank one over $R$, and so are the spaces $A(k)$ for all $k$. The theory is 1-multiplicative, as defined at the end of Section~\ref{sec_uni_n}, with the maps in (\ref{eq_hom_parallel}) isomorphisms, that is, satisfies the Atiyah tensor product axiom. Equivalently, this theory has a neck-cutting formula, essentially the simplest one possible, see Figure~\ref{fig_3_6}.  

\begin{figure}[h]
\begin{center}
\includegraphics[scale=1.0]{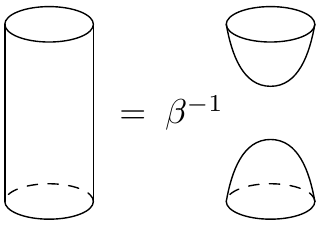}
\caption{Neck-cutting formula, for the theory with the generating function  $Z_1(T)$ in (\ref{eq_Z1_function}).}
\label{fig_3_6}
\end{center}
\end{figure}

It corresponds  to the Frobenius algebra of rank one over $R$, with the trace $\epsilon(1)=\beta$. 

We want to point out that the simplest truncation $Z(T)$  of the generating function $Z_1(T)$ by resetting evaluations at genus one and greater to zero leads to a theory with nontrivial combinatorics of spaces $A(k)$ and nontrivial skein relations for cobordisms, unlike that of $Z_1(T)$. 

\vspace{0.1in}


\subsection{Free theory} \quad 
\vspace{0.05in}

Consider the 'free' theory,  where 
\begin{equation} \label{eq_free_theory}
    R  = R'[\alpha_0,\alpha_1,\dots]
\end{equation}
is the ring of  polynomials in the  evaluations  $\alpha_0,\alpha_1,\dots$ of  surfaces of all genera with coefficients in  a commutative ring $R'$. 
In other words, there  are no polynomial relations on the $\alpha_i$'s in $R$. 
Recall that for any theory there is a surjective map 
\begin{equation}\label{eq_map_again}  
A_k \lra A(k)
\end{equation} 
of commutative algebras, see (\ref{eq_def_Ak}), (\ref{eq_As_map})  and the discussion in  Section~\ref{subset_union_circle}. 

\begin{prop} \label{prop_free_theory} 
Map (\ref{eq_map_again}) is an isomorphism for the free theory (\ref{eq_free_theory}) for all $k$. 
\end{prop} 
\begin{proof} 
The proposition says that there are no skein relations on cobordisms with boundary beyond the  relations on
closed cobordisms. The latter relations say that a closed surface of genus $g$  evaluates to  $\alpha_g$. The proof is rather straightforward. The   reader may  wish to compare with the proof of Theorem~2.1 in Freedman et al.~\cite{FKNSWW}, who show that their universal pairing is positive-definite in dimension two. Our proposition is simpler. 

In the  proof, we can restrict to the $k>0$ case. 

For an oriented 2-manifold $S$ with boundary $\sqcup_k\SS^1$ denote by $S_{\partial}$ the union of components of $S$ with non-empty boundary. Equivalently, $S_{\partial}$ is obtained from $S$ by  removing all closed components. 

To $S$ assign its genus $g(S)$, which is the sum of genera of its connected components, both closed and with boundary. Then $g(S_{\partial})$ is the genus of the non-closed portion of $S$. 

\vspace{0.07in}

Notice that $A_k$ is a free $R$-module with a basis consisting of homeomorphism classes rel boundary of 2-manifolds $S'$ with $\partial(S')=\sqcup_k \SS^1$ and no closed components. 
Consequently, for our choice  of $R$, $A_k$ is a free $R'$-module with basis consisting of homeomorphism classes rel boundary of 2-manifolds $S$ with $\partial(S)=\sqcup_k \SS^1$. 
Write a non-zero element of $A_k$ as
\begin{equation}\label{eq_sum_S}
    a = \sum_S  a_S \cdot S, \ a_S \in R', a_S  \not= 0,
\end{equation}
the sum over finitely many $S$. To show that its  image in $A(k)$ is not zero, we need to find a 2-manifold $T$ with the same boundary $B=\sqcup_k \SS^1$ such that 
\begin{equation}\label{eq_sum_S_alpha} 
    \sum_S a_S \cdot \alpha(-T \cup_B S)  
\end{equation}
is not $0$ in $R'$, where $\alpha(-T \cup_B S)$ is the  evaluation of the closed surface $-T \cup_B S$ given by  gluing $T$ and  $S$ along common boundary $B$. 

We fix an order on the boundary circles and label them $1$ through $k$. 
Let $g$ be the maximum  of genera  of $S$  among  all  $S$ in the sum.   
Choose a large integer $K$ with the  property that in the set of powers  $\{K,K^2,\dots, K^k\}$ any two distinct sums of its elements differ by more that $g+k$ (taking $K>g+k$ should  suffice). 

Consider the cobordism $T$ which is the  disjoint union of connected surfaces of genera $K,K^2,\dots, K^k$ with one boundary component each, corresponding to circles $1,2,\dots, k$ in this order. Take a surface $S$ from the above sum and  consider $ \alpha(-T \cup_B S) $, which is a product of $\alpha_i$'s. Looking at this product, we can single out components of $-T \cup_B S$ that came from closed components of  $S$ -- these are the components of genus at most $g$, contibuting $\alpha_i$'s with $i \le g$. Looking at terms in the product of genus at least $K$, we can figure out which circles belong to the same connected  component of $S$ as well as the genus of that connected component. Since we can reconstruct each $S$ uniquely from $ \alpha(-T \cup_B S) $, this means that there are no cancellations in the sum (\ref{eq_sum_S_alpha}), and it has as many terms as the original sum (\ref{eq_sum_S}). In particular, the sum in (\ref{eq_sum_S}) is not zero in $A(k)$. Proposition follows. 
\end{proof} 

{\emph  Remark:} 
It may be interesting to investigate theories where the quotient map $A_1\lra A(1)$ in (\ref{eq_map_again}) is an isomorphism, but $A_k\lra A(k)$ is not, for some $k>1$. In other words, there are no new skein relations when transitioning from closed surfaces to those with one boundary component, but there are skein relations on surfaces with several boundary components. 

\vspace{0.1in}   


\subsection{More  examples} 
\label{subsec_poly_inv} 



\subsubsection{Inverse of a polynomial generating function.}
Now consider the case that $Z(T)=Q(T)^{-1}$, for $Q(T)=(1-\gamma_1 T)\dots (1-\gamma_M T)$ as above. Then 
\[ Z(T) = \sum_{n\ge 0} h_n(\gamma) T^n = \sum_{n\ge 0}\ovh_n T^n.
\]
The Hankel determinant for the first $K$ vectors $1,x,\dots, x^{K-1}$ has the form 
\begin{equation} \label{eq_hankel_poly_2}
      H_{[0,K-1]} =  
  \left( {\begin{array}{ccccc}
   1 & \ovh_1 & \ovh_2 & \dots & \ovh_{K-1} \\
   \ovh_1 & \ovh_2 & \ovh_3 & \dots & \ovh_{K} \\
   \ovh_2 & \ovh_3 & \ovh_4 & \dots &  \ovh_{K+1} \\
   \dots & \dots & \dots & \ddots & \dots \\
   \ovh_{K-1} & \ovh_{K} & \ovh_{K+1} & \dots & \ovh_{2K-2} \\
  \end{array} } \right)
\end{equation}
Multiplying by the longest permutation matrix  $J_K$, as in formula (\ref{eq_toeplitz_N}), produces the Jacobi-Trudi matrix  
$J_K H_{[0,K-1]}$ for the partition $\lambda_K=((K-1)^K)$ with the determinant equal to the corresponding Schur function. Notice that the number of variables $\gamma_1,\dots, \gamma_M$ is $M$. When $K=M$, the determinant 
\begin{equation}
    \det( H_{[0,M-1]}) = (-1)^{M(M-1)/2}(\gamma_1\dots \gamma_M)^{M-1}= (-1)^{M(M-1)/2} \cdot (\ove_M)^{M-1}, 
\end{equation}
see remark at the end  of Section~\ref{subsec_state_hankel}. 
Theorem~\ref{thm_2} specializes in this case to the following  result.  

\begin{prop} \label{prop_state_inverse} 
The state space $A=\alpha(\SS^1)$ for the generating function $Z(T)= Q(T)^{-1}$, with $Q(T)=(1-\gamma_1 T)\dots (1-\gamma_MT)$ over the  base ring $R_{M,0}$ is a free $R_{M,0}$-module of rank $M$ with a basis $1,x,\dots, x^{M-1}$. As an $R_{M,0}$-algebra, it is given by   
\begin{equation}\label{eq_monic_M_only}
      A\cong R_{M,N}[x]/(x^M - \ove_1 x^{M-1}+\dots +(-1)^M \ove_M), 
  \end{equation}
  where $\ove_i$  is  the $i$-th elementary symmetric function in $\gamma_1,\dots, \gamma_M$. 
\end{prop} 

\vspace{0.1in} 


\subsubsection{Rank two Frobenius extensions.}
\label{subsub_rank_2} For information on Frobenius extensions of rank two  we refer the reader to~\cite{Kh2,TT}; more references can be found in~\cite{KR}. 
Consider Frobenius extension  of rank two  with 
\begin{eqnarray*} 
 &  &  R= \Z[E_1,E_2], \ A = R[X]/(X^2-E_1X + E_2),\\ 
 &   & \epsilon(X)=1, \ \epsilon(1)=0, \\
 &    &  \Delta(1) = X\otimes 1 + 1\otimes X - E_1\cdot  1\otimes 1
\end{eqnarray*} 
(comultiplication $\Delta$ is uniquely determined by the  trace $\epsilon$). Let $\mcD=E_1^2-4 E_2$. In this 2D TQFT the  value of  the  closed surface of genus $2n+1$ is $2\mcD^n$, for $n\ge 0$, while closed surfaces of even genus evaluate to zero. Consequently, the generating function is
\begin{equation}\label{eq_gen_f_1}
    Z(T) = 2T(1+\mcD T^2+ \mcD T^4 + \dots ) = \frac{2T}{1-\mcD T^2}. 
\end{equation}
When the cofficient ring is enlarged from $R$ to $R_{u}=\Z[u_1,u_2]$ with $E_1=u_1+u_2$, $E_2=u_1u_2$, the  discriminant becomes a square, 
\[
\mcD=E_1^2-4 E_2 =  (u_1-u_2)^2,
\]
and the generating function can be further factored 
\begin{equation}
    Z(T) = \frac{2T}{1-(u_1-u_2)^2 T^2} = 
    \frac{2T}{(1-(u_1-u_2) T)(1+(u_1-u_2)T)}, 
\end{equation}
with linear in $T$ terms in the denominator.  
We changed from $\alpha_1,\alpha_2$ in~\cite{KR} to $u_1,u_2$ here to avoid a clash of notations, since $\alpha_i$'s are already in use in this paper. 

The trace map in  this extension can  be deformed to \begin{equation*}
   \epsilon(1) = \rho_0, \  \epsilon(X)=\rho_1,
\end{equation*}
using  notations in~\cite[Section~4.3]{KK}, also 
see~\cite{TT,V}. 
Condition that   $\{1,X\}$ is a basis and the pairing  is  unimodular in this basis translates into the  invertibility of the Gram determinant (denoted  $\rho$) for these vectors 
\[ \rho = 
\det \begin{bmatrix}   \rho_0 & \rho_1 \\
    \rho_1 & E_1 \rho_1 - E_2 \rho_0 \end{bmatrix} = 
    - (E_2 \rho_0^2 - E_1 \rho_0 \rho_1 + \rho_1^2) = - (\rho_1-u_1\rho_0)(\rho_1-u_2\rho_0) , 
\] 
or, equivalently, invertibility of $\epsilon(X-u_1)=\rho_1-u_1\rho_0$ and $\epsilon(X-u_2)=\rho_1-u_2\rho_0$ when the ground ring  is $R_u$. Let $\rho'=E_1\rho_1-E_2\rho_0$. 

With the invertibility assumption, we can write down the formula for comultiplication 
\begin{equation}\label{eq_rank_2_comult}
\Delta(1)  =   \rho^{-1}\left( (\rho' \cdot 1 - \rho_1 \cdot X)
\otimes 1+ (-\rho_1 \cdot  1 + \rho_0\cdot X)\otimes X \right),
\end{equation}
Consequently, 
\begin{eqnarray*}
    m (\Delta(1)) & = & \rho^{-1}(E_1\rho_1 - E_2\rho_0 - 
    2\rho_1 X + \rho_0 X^2)  \\ 
    &  = &  \rho^{-1}((E_1\rho_1 - 2 E_2\rho_0)\cdot  1 
    + (\rho_0 E_1 - 2 \rho_1)\cdot X) \\
    & = &  \rho^{-1}((\rho_0 u_2-\rho_1)(X-u_1)+  (\rho_0 u_1 -\rho_1)(X-u_2)) \\
    & = & -\left( (\rho_0 u_1-\rho_1)^{-1}(X-u_1)+  (\rho_0 u_2 -\rho_1)^{-1}(X-u_2) \right) .
\end{eqnarray*}
The connected surface $\SS^1_n$ of  genus  $n$  with one boundary circle will represent the  element  $[\SS^1_n]$ in $A(1)$. The above formula make this element easy to compute, when we enlarge the ground ring from $R$ to $R_u$ and the state space from $A(1)$ to $A_u(1)\cong A\otimes_R R_u$. Computation is easy  since in $R_u$ 
\[ 
(X-u_1)(X-u_2)=0, \  (X-u_1)^2=(u_2-u_1)(X-u_1), \ 
(X-u_2)^2=(u_1-u_2)(X-u_2),
\] so that 
\begin{eqnarray*}
 (m\Delta(1))^n & = & (\rho_1-\rho_0 u_1)^{-n}(X-u_1)^n+  (\rho_1-\rho_0 u_2 )^{-n}(X-u_2)^n \\
 &  = & (\rho_1-\rho_0 u_1)^{-n}(u_2-u_1)^{n-1}(X-u_1)+  (\rho_1-\rho_0 u_2)^{-n}(u_1-u_2)^{n-1}(X-u_2), \\
 & = & (u_2-u_1)^{n-1}\left( (\rho_1-\rho_0 u_1)^{-n}(X-u_1)+(-1)^{n-1}  (\rho_1-\rho_0 u_2)^{-n}(X-u_2)\right) 
\end{eqnarray*}
and the trace 
\begin{eqnarray*}
\epsilon((m\Delta(1))^n) & = & (u_2-u_1)^{n-1}\left( (\rho_1-\rho_0 u_1)^{-n+1}+(-1)^{n-1}  (\rho_1-\rho_0 u_2)^{-n+1}\right) .
\end{eqnarray*}

We can now compute the generating function for this  theory:  
\begin{eqnarray*}
Z(T) & = & \sum_{n\ge 0} \epsilon((m\Delta(1))^n) T^n 
\\
 & = & \sum_{n\ge 0}\left( \left(\frac{u_2-u_1}{\rho_1-\rho_0u_1}\right)^{n-1} T^n
 + (-1)^{n-1} \left(\frac{u_2-u_1}{\rho_1-\rho_0 u_2}\right)^{n-1} T^n \right)
   \\
 & = &   \frac{1}{u_1-u_2} \left( \frac{(\rho_1-\rho_0 u_2)^2}{\rho_1-\rho_0 u_2-(u_1-u_2)T)} - \frac{(\rho_1-\rho_0 u_1)^2}{\rho_1-\rho_0 u_1-(u_2-u_1)T)}
  \right)  \\
  & = &  \frac{\rho_0\rho + (2\rho-\rho_0^2\mcD )T}{\rho -\rho_0 \mcD T +\mcD T^2}
\end{eqnarray*} 
To recover the special case (\ref{eq_gen_f_1}) we specialize 
$\rho_0=0$, $\rho_1=1$, giving  $\rho=-1$. 

\vspace{0.07in} 

This theory is 1-multiplicative, with $1_{tu}$ decomposable into an element of   $A(1)^{\otimes 2}$. The neck-cutting relation can be read  off  the formula  (\ref{eq_rank_2_comult})  for $\Delta(1)$. Note that the coefficient at $T$ in the  power series is $2$, the dimension of the Frobenius algebra.

\vspace{0.1in} 


\subsubsection{Generating function $\beta/(1-\gamma T)$.} 
Consider the generating  function
\begin{equation}  \label{eq_Z_rat1}
    Z(T) = \frac{\beta}{1-\gamma T} = \beta + \beta\gamma T + \beta\gamma^2 T^2 + \dots ,
\end{equation}
where $\beta,\gamma\in R$, $\beta,\gamma\not=0$ and $R$ is an  integral domain. The evaluation of genus $g$ closed surface  in this theory is  $\beta\gamma^g$. Skein relation in 
Figure~\ref{fig_3_10} holds.

\begin{figure}[h]
\begin{center}
\includegraphics[scale=1.0]{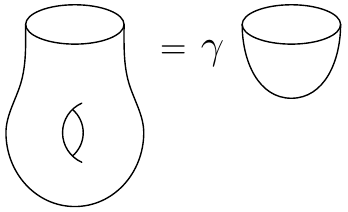}
\caption{Skein relation for the generating function  in 
(\ref{eq_Z_rat1}).}
\label{fig_3_10}
\end{center}
\end{figure}

The  space $A(1)\cong R$, with disk  $[\SS^1_0]$  as the  generator, 
and  $[\SS^1_n]=x^n=\gamma^n \cdot  1$. Since genus can be reduced to zero in  this  theory, elements $1$ and  $y_{12}$ in Figure~\ref{fig_3_2} span  $A(2)$, and elements  $1,y_{12},y_{13},y_{23},y_{123}$ in Figure~\ref{fig_3_3} span $A(3)$. The Gram matrices are 

\begin{equation}
    G_2= \beta\begin{bmatrix} \beta & 1 \\
    1 & \gamma \end{bmatrix}, \ \ 
    G_3 = \beta\begin{bmatrix} 
     \beta^2 & \beta & \beta & \beta & 1 \\
     \beta & \beta\gamma     & 1 & 1 & \gamma \\
     \beta & 1 & \beta\gamma     & 1 & \gamma \\
     \beta & 1 & 1 & \beta\gamma    & \gamma \\
     1   & \gamma     & \gamma    & \gamma     & \gamma^2 
     \end{bmatrix}  . 
\end{equation}
Each coefficient of $G_2$ and $G_3$ is divisible by $\beta$, and we took it out of both matrices. The  determinants are 
\begin{equation}
    \det(G_2) = \beta^2(\beta\gamma-1), \ \ 
    \det(G_3)=\beta^5 (\beta\gamma-1)^4 (\beta\gamma-2). 
\end{equation}

We see that, unless $\beta\gamma=1$ or $\beta\gamma=2$, the state  space $A(3)$ is five-dimensional with the above basis. The space $A(2)$ is two-dimensional,  unless  $\beta\gamma=1$.   
If  $\beta\gamma=1,$ that is, $\gamma=\beta^{-1}$, we get a genuine 2D TQFT with the generating function (\ref{eq_Z1_function})   discussed at the end of  Section~\ref{subsec_constant}, with  one-dimensional state spaces $A(k)\cong A(1)^{\otimes k}$. The theory is 1-multiplicative  only  if  $\beta\gamma=1$.

If $\beta\gamma=2$, the generating function  
can be written as 
\begin{equation}  \label{eq_Z_rat2}
    Z(T) =  \beta + 2 T + 2 \gamma T^2 + \dots =\beta + 2 T\sum_{n\ge 0} (\gamma T)^n .
\end{equation}
The state space $A(2)$ has rank two and a basis $\{1,y_{12}\}$, while there is a linear  relation on the above five spanning elements  of  $A(3)$:  
\begin{equation}
    \gamma^2\cdot 1  - \gamma(y_{12}+y_{13}+y_{23}) + 2y_{123}=0. 
\end{equation}
If 2 or $\gamma$  is  invertible  in $R$, we can conclude that $A(3)$ is  a free $R$-module of rank four. 


\subsubsection{Maps into  multiplicative  theories.} 
Consider a standard  example of  a  commutative Frobenius algebra, the even-dimensional cohomology ring $B=\mathrm{H}^{2\ast}(X,\kk)$ of an oriented $2m$-dimensional manifold  $X$ over a field  $\kk$ with  the standard trace $\epsilon: \mathrm{H}^{2m}(X,\kk)\stackrel{\cong}{\lra}\kk.$  
This  theory can be made $\Z$-graded, by shifting the grading of $B$ down by $m$  so that it sits between degrees $-m$ and $m$. 
Then the map $b(S)$ between tensor powers of $B$  associated in this theory to a 2D cobordism $S$ has degree  $\deg(b(S))= - m \chi(S)$,  proportional  to the Euler characteristic of  $S$. 
For a closed $S$ the map $b(S)$ is the multiplication by an element of the ground field $\kk$. Consequently, for a closed connected surface $S$ of genus $g$ the invariant $b(S_g)=0$ if $g\not= 1$ and $b(S_1)=\dim(B)$. Here we assume that $m>0$. 

Consider the universal theory over $\kk$ for this generating  function, assuming that $\dim(B)\not=0$ in $\kk$, that is, $\mathrm{char}(\kk)$ is not a divisor of $\dim(B)$. This universal theory depends only on the choice of field $\kk$ and on the dimension of $B$.  It has the generating function  $Z(T)=\dim(B)\cdot T.$ 

As usual, denote the state space of $k$ circles in this  theory by $A(k)$.
There are natural $\kk$-linear maps
\begin{equation}\label{eq_A_to_B} 
    A(k) \lra B^{\otimes k}
\end{equation}
that preserve bilinear forms on $A(k)$ and $B^{\otimes k}$, respectively, where the bilinear form on the latter is the tensor power of that on $B$ given  by the trace $\epsilon$. This map is induced by the  natural homomorphism $\phi_k: A_k \lra B^{\otimes k}$ sending a cobordism from the  empty one-manifold $\emptyset_1$ to the  union of $k$ circles to the corresponding map in the 2D TQFT $(B,\epsilon),$ so that there is a factorization of $\phi_k$ as 
\begin{equation}\label{eq_A_to_B_factor} 
    A_k \lra A(k) \lra B^{\otimes k}.
\end{equation}

Since the collection of these maps $A(k)\lra B^{\otimes k}$, over all $k\ge 0$, respects the bilinear forms on $A(k)$ and $B^{\otimes k}$, and bilinear forms are non-degenerate on $A(k)$, we see that homomorphisms $\phi_k$ are injective. In particular, $A(k)$ is isomorphic to its image in $B^{\otimes k}$, for any $B$ with fixed $n=\dim(B)$ in $\kk$. In particular, to determine $A(k)$ given $n$ and $Z(T)=nT$, one can choose $B$ to be $\mathrm{H}^{\ast}(\mathbb{CP}^{n-1},\kk)$, the cohomology of $\mathbb{CP}^{n-1}$, or any other commutative  Frobenius graded algebra $B$ such that 
\begin{equation}
   B= \oplus_{i=0}^m B^{-m+2i}, \   B^{-m}\cong \kk \cdot 1, \ B^m\cong \kk, 
\end{equation}
for some $m>0$, with 
$\dim(B)=n$, and structure  maps of $B$ associated to cobordisms $S$ having degree $-m\chi(S)$. 

A connected oriented 2-manifold with boundary $\sqcup_k \SS^1$ gives the zero vector in $A(k)$ and $B^{\otimes k}$ if it has at least two handles, since  the evaluation is trivial on any closed surface of genus at least two, and genus two cobordism $\SS^1_2$ with one boundary component  gives the zero vector in $B$. Notice that genus one cobordism $\SS^1_1$ gives the top degree vector $w\in B$ scaled so that  $\epsilon(w)=\dim(B)$. Then the  invariant of  $\SS^1_2$ is $w^2=0\in B$. 

There is a skein relation in either theory, in Figure~\ref{fig_3_9}. 

\begin{figure}[h]
\begin{center}
\includegraphics[scale=1.0]{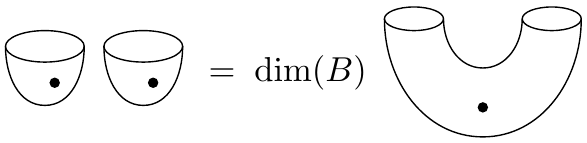}
\caption{Skein relation in $A(2)$ and $B^{\otimes 2}$.}
\label{fig_3_9}
\end{center}
\end{figure}

To summarize, if a component carries two dots, the cobordism evaluates to zero. If two genus zero components each carry a dot, they can be merged into  a single genus zero component with a dot. Hence, any cobordism reduces, up to scalars $\dim(B)^{\ell}$, to a disjoint union of connected  genus zero cobordisms, at most one of which may carry a dot. Such cobordisms give a spanning set for  $A(k)$. They are enumerated by  decompositions of $k$, together with a choice of at most one set in the decomposition. This spanning set is not a basis of $A(k)$, admitting  other skein relations that depend on $n$.  


\vspace{0.07in} 

The above is a  rather  degenerate  example due  to a very special  form of  the generating function, $Z(T)=\dim(B)T.$ To modify this  construction, one can, for instance, consider   manifolds $X$
with  an action  of a compact Lie group $G$ on $X$. One would want the equivariant cohomology $\mathrm{H}^{\ast}_G(X,\kk)$ to be  a Frobenius  algebra  over  $\mathrm{H}^{\ast}_G(p,\kk)$, where $p$  is a point, so it may make sense to restrict  to $G$-formal  manifolds, see~\cite{GKM} and follow-up papers. Equivariant cohomology $\mathrm{H}^{\ast}_G(p,\kk)$, which is the cohomology of  the classifying space  $BG$, is usually nontrivial in various positive even degrees, leading to non-vanishing of $\alpha_g$ for at least some $g>1$ and producing more general power series $Z(T)$. 
Example in  Section~\ref{subsub_rank_2} above  is  of that type, with $X=\SS^2$  and $G=U(2)$ or   $G=U(1)\times U(1)$ with the standard action on $\SS^2$ via the identification of the latter with $\mathbb{CP}^1$. 
One may further generalize from cohomology to more general $G$-equivariant complex-oriented cohomology theories, subject to suitable $G$-formality assumptions on $X$. 

\vspace{0.07in}

Maps (\ref{eq_A_to_B}) are $S_k$-equivariant under the natural action of the  symmetric  group $S_k$ on both spaces induced by permutations of boundary circles. To understand  spaces $A(k)$ for  more general generating functions $Z(T)$ one may search for similar $S_k$-equvariant maps, over all $k$, that respect bilinear forms, where $B$ is just a $\kk$-vector space with a bilinear form rather than a commutative Frobenius algebra.  

\vspace{0.1in}


\subsubsection{Non-injectivity of disjoint union maps.} \label{subsub_nonin}

Let us briefly discuss possible non-injectivity of maps  $\alpha_{N_0,N_1}$ in $(\ref{eq:can_map_tensor})$ when  $R$ is not a field, also see Proposition~\ref{prop_injective}. 

On  the algebraic level, one reason for non-injectivity is the following.  Suppose given free $R$-modules  $F'_1$, $F'_2$ with symmetric bilinear  forms $(,)_1,(,)_2$. Form quotient modules 
$F_i  = F_i'/\mathrm{ker}((,)_i)$, $i=1,2$. The free $R$-module $F_{12}':= F'_1\otimes_R F'_2$ comes with the bilinear form $(,)_{12}$ given by the tensor product of forms $(,)_1,(,)_2$, and  one  can consider the quotient module $F_{12}=F_{12}'/\mathrm{ker}((,)_{12})$. The natural map  
\begin{equation}\label{eq_bilin_not_surj}
    F_1\otimes_R F_2 \lra F_{12}
\end{equation}
is surjective but not, in general, injective. 

An example is given by taking $R=\kk[\beta]/(\beta^2)$ and  one-dimensional modules  $F_i'=Rv_i$, $i=1,2$ with inner  products $(v_i,v_i)=\beta$. Then on $F_{12}'=R( v_1\otimes  v_2) $  the  bilinear form is identically  zero, so that $F_{12}=0$, while $F_i\cong \kk v_i$ and  $F_1\otimes_R F_2\cong \kk (v_1\otimes v_2)$, with the trivial action of $\beta$. The map  in (\ref{eq_bilin_not_surj}) is not injective, taking  $\kk$ to the  zero module. 

In  the above example  $R$ is not an integral domain. For an example with $R$ an  integral domain, take  $R=\kk[a_{11},a_{12},a_{22}]/(a_{11}a_{22}-a_{12}^2)$ and two-dimensional free $R$-modules $F_1'=F_2'=Rv_1\oplus  Rv_2$ with the Gram matrix $ \begin{bmatrix} a_{11} & a_{12}\\
    a_{12} & a_{22} \end{bmatrix} $ describing the bilinear  form in  the basis $\{v_1,v_2\}$. 

Ring $R$ has a maximal ideal $\mathbf{m}=(a_{11},a_{12},a_{22})$.
$R$-modules  $F_1\cong F_2$ that are quotients of $F_1'=F_2'$ above by the kernel of the bilinear form surject onto a     two-dimensional  $k$-vector space 
$W=\kk v_1\oplus  \kk v_2$ via the quotient map $F_1\lra F_1/\mathbf{m}F_1\cong W$. This  follows from the observation that only the trivial $\kk$-linear  combination of $v_1$ and $v_2$ is in the kernel of this bilinear form on $F_1'$. Space $W$ is naturally an $R$-module with the trivial action of $\mathbf{m}$ and the  quotient map is an $R$-module  map. 
Consequently, there is a surjection of $R$-modules 
$F_1\otimes_R  F_2 \lra W\otimes_{\kk} W$. 
Vector $v=v_1\otimes v_2 - v_2\otimes v_1 \in F_1\otimes_R F_2$ is nonzero, since its image in $W\otimes_{\kk} W$ is nontrivial. 
The  image of $v$ in  $F_{12}$ is trivial, since all inner products $(v_i\otimes v_j,v)_{12}$, $i,j=1,2$ are zero. 

\vspace{0.07in}  
    
The first example above can perhaps be lifted to the level of topological theories for $n=4$, by selecting $\alpha$ to be non-zero (and equal $\beta$, as above) only on two  carefully selected connected $4$-manifolds  $K_i=(-M_i)\cup_{N_i} M_i$, $i=0,1$ so that $\alpha_{N_0,N_1}$ is not injective. Manifolds $K_i$ and $N_i$ would need to satisfy a number of properties, including uniqueness of an embedding $N_i\subset K_i$ up to diffeomorphisms of $K_i$.    

Alternatively, if we allow suitable decorations of 2-manifolds, the  first  example can be lifted  to topological theories  in dimension two. Namely, one should require each component of a cobordism and of its boundary to carry a color, $0$ or $1$, so that the colors of a component and its boundary match. This is a modification of a cobordism category to non-interacting 2-color case. Closed oriented 2-manifolds are parametrized by  their genus and color. We set the evaluation to be $\beta$ on 2-spheres of either color and zero on all higher genus closed connected  surfaces.   
If  $N_i$, $i=0,1$ are circles of color $0$ and $1$,  respectively, 
the product map $\alpha_{N_0,N_1}$ can be identified with (\ref{eq_bilin_not_surj}) for the bilinear forms as above. The map is not injective then. 

Our introduction of two colors is a hack to disallow the tube  cobordism that connects two circles, since that would make the topological counterpart of the bilinear form $F_{12}$ nontrivial in the above example. A variation of this trick can be used to  implement the second example in decorated 2-dimensional topological theories as well, by introducing seamed circles on connected components and allowing colors of facets  to change upon crossing a  seamed circle.  

We feel compelled to point out possible non-injectivity of maps $\alpha_{N_0,N_1}$, even without producing interesting examples.

\vspace{0.1in} 


\subsection{Recursive sequences} \label{sec_recur}
$\quad$
\vspace{0.07in}

Below  we relabel $T$ into $z^{-1}$ making  $Z(T)=Z(z^{-1})$ a  power series in the latter variable. 

Consider the ring $R[z]$ of polynomials in a formal variable $z$ as well as the following $R[z]$-modules:
\begin{itemize}
    \item $R[z^{-1}]\undone$, with $z\undone=0$. These are polynomials in $z^{-1}$  with $z$ killing the element $\undone$ and otherwise acting  by  increasing the degree  of $z$ in $z^{-n}\undone$ by one. 
    \item $R((z^{-1}))$ the ring of Laurent power series in $z^{-1}$ with $z$  acting by the usual multiplication. As an $R[z]$-module it has a submodule $z R[z]$, with the quotient module that we write as  
    $R[[z^{-1}]]\undone,$ so there's an exact sequence of $R[z]$-modules 
    \begin{equation}
        0 \lra zR[z]\lra R((z^{-1})) \lra R[[z^{-1}]]\undone \lra 0 
    \end{equation}
    An element of $R((z^{-1}))$ decomposes uniquely as 
    $f= f_+ + f_-$, with $f_+\in zR[z]$ and  $f_-\in R[[z^{-1}]]$. When  we treat the latter as an $R[z]$-module, not as a ring, we write it as $R[[z^{-1}]]\undone.$ There is an inclusion of $R[z]$-modules 
    \begin{equation}
        R[z^{-1}]\undone \subset  R[[z^{-1}]]\undone  .
    \end{equation}
\end{itemize}
We can identify $R[z]$-modules 
\[
R[[z^{-1}]]\undone\cong (R[z])^{\ast}= \Hom_{R}(R[z],R).
\]
An element  of  the latter is determined  by a sequence $\ualpha=(\alpha_0,\alpha_1,\dots)$, $\alpha_n\in R$, describing the value of the  functional on $z^n$, $n\ge 0$. To $\ualpha$ we assign power series 
\[
Z(z^{-1})=\sum_{n\ge 0}\alpha_n z^{-n}.
\] 

An $R[z]$-module homomorphism 
\[ \gamma: R[z] \lra R[[z^{-1}]]\undone 
\] 
is determined by the image  of $1$, which can be any power series $Z(z^{-1})$ as above. We denote this homomorphism by $\gamma_{\undalpha}$. Theorem~\ref{thm_rational} from Section~\ref{subset_rat_th} can be restated as follows. 

\begin{prop}\label{prop_image_fd} Let $R$ be a field $\kk$. The image of $\kk[z]$ under $\gamma_{\undalpha}$ is a finite-dimensional $\kk$-vector space iff the power series $Z(z^{-1})$ is a rational function in $z^{-1}$. 
\end{prop}

We refer to~\cite[Section 8.3.1]{Fh} for a proof in this language. 
Replacing $z$ with $T^{-1}$, the image  $\gamma_{\undalpha}(\kk[z])$ is naturally isomorphic to $A(1)$ in the theory assigned to $\undalpha$, so that Proposition~\ref{prop_image_fd} is equivalent 
to Theorem~\ref{thm_rational}. 

\vspace{0.07in} 

An equivalent  characterization of such homomorphisms is that the  sequence $\undalpha=(\alpha_0,\alpha_1,\dots)$ is eventually  recurrent, that is, there exist $N,M>0$ and $q_0,\dots, q_{N-1}\in \kk$  with 
\begin{equation}\label{eq_LRS}
    \alpha_{m+N} = q_0\alpha_m + q_1 \alpha_{m+1}+ \dots + q_{N-1}\alpha_{m+N-1} 
\end{equation}
for all $m\ge M$. Let us  call homomorphisms satisfying this  finite-dimensionality condition \emph{recurrent homomorphisms}. 
Sequences as in (\ref{eq_LRS}) are known as \emph{linearly recursive sequences}. They are important in number theory, in control theory, and in many other  fields~\cite{EPSW,Fa}. 

\vspace{0.07in} 

The  space of recurrent homomorphisms $\gamma_{\undalpha}$, in  addition to the Kronecker classification via rational functions, can be  described as the space  of representative  functions on $\kk[z]$. 
Given an associative algebra $B$ over a field $\kk$ with  the multiplication
$m:B\otimes B \lra B$, it can be dualized to obtain a map 
\[ m^{\ast}: B^{\ast}\lra (B\otimes B)^{\ast}
\] 
The latter space contains $B^{\ast}\otimes B^{\ast}$ as a subspace, which is proper iff $B$ is infinite-dimensional. We say that $f\in B^{\ast}$ is \emph{representative} if 
$m^{\ast}(f)$ is in $ B^{\ast}\otimes B^{\ast} $ rather  than just being an element of the bigger space $(B\otimes B)^{\ast}$: 
\[ m^{\ast}(f)\in B^{\ast}\otimes B^{\ast} \subset (B\otimes B)^{\ast}
\] 
The set of representative functions is denoted $B^{\circ}$. It's a vector subspace of $B^{\ast}$ and is naturally a coalgebra. $B^{\circ}$ consist  of functionals on $B$ whose kernel contains an ideal of finite codimension in $B$. If $B$ is a bialgebra (or a Hopf algebra), $B^{\circ}$ is a bialgebra (resp. a Hopf algebra) as well~\cite[Section 1.5]{DNR}, known as  the \emph{Sweedler dual} of $B$, or the \emph{finite  dual}, the \emph{dual bialgebra}, or the  \emph{bialgebra  of representative functions}~\cite[Section 9.1]{Mn}. 

When  $B=\kk[z]$ is a polynomial algebra on one  generator, representative functions are in a bijection with recurrent  sequences, thus also in a bijection with rational functions. 

Bialgebra structures  on $\kk[z]$ (with $\Delta(z)=1\otimes z + z\otimes 1$ or $\Delta(z)=z\otimes z$) give rise to  bialgebra structures on the coalgebra  $\kk[z]^{\circ}$ of representative functions~\cite{CG,PT,LT}.  
When $\kk$ is algebraically closed, the numerator  and denominator  of a rational function factor into linear terms, and a basis in $\kk[z]^{\circ}$ can be written down  explicitly, see~\cite[Example 9.1.7]{Mn}, \cite{VS} and references therein.   

The notion of the finite dual is more subtle for a bialgebra over a commutative ring $R$ than over a field; the case of $R[x]^{\circ}$ is discussed in~\cite{AGTW,Ku}, see also~\cite[Section 9]{Ha}. 

Monograph~\cite{EPSW} details arithmetic properties or  recurrent sequences, also see~\cite{Al} and follow-up papers for connections to K-theory and~\cite{LB} and references therein for connections to arithmetic geometry. 

\vspace{0.07in} 

Notice that the polynomial $r_{M,N}$ in (\ref{eq_r_N_M}),  which describes the extension $R\subset A(1)$, consists of two factors, the second of which is the monomial $x^{\max(N+1,M)-M}$ that depends  only on degrees $M$ and $N$. It's natural to distinguish two  cases  
\begin{itemize}
    \item $N<M$, that is, when the numerator of  the  rational series $Z(T)$ has lower degree than than the denominator, $\deg(P(T))<\deg(Q(T))$. The second  factor  is not present. 
    \item $N\ge M$, which is the opposite case, $\deg(P(T))\ge \deg(Q(T))$. The second factor  is $x^{N+1-M}$. 
\end{itemize}

The first case is the one  more commonly encountered to date in the literature on commutative Frobenius extensions. 
In particular, it  appears in Jouve and  Rodrigues Villegas~\cite[Section 2]{JRV} via the notion of a monogenic Frobenius algebra. The authors also provide a classification of isomorphism classes of such Frobenius extensions $\kk\subset A(1)$  via rational functions that are zero at infinity. We  refer the reader to the latter paper, Furhmann~\cite{Fh}, and various papers in control theory,  see~\cite{Fa} and references therein, for details and also for the connection to the Bezoutians. 

Note, though, that a rational series $Z(T)$ determines not only a Frobenius extension  $\kk\subset  A(1)$ but   the entire  collection of $A(k)$, over $k\ge 0$, with various maps between them induced  by 2D cobordisms,  including  the symmetric group $S_k$ action on $A(k)$ and tensor  product  maps (\ref{eq_hom_parallel}).

\vspace{0.1in} 

%
%

\section{Overlapping theta-foams and the Sergeev-Pragacz formula}\label{sec_SP_formula}

{\it  Sergeev-Pragacz formula.}
We recall the basics  of  the supersymmetric Schur  functions and  the Sergeev-Pragacz formula,  following~\cite{MJ1,Mo}  and  using  variables $x_i$ and $y_j$'s in place of $\gamma_i$ and $\beta_j$. Fix $M,N\ge 0$ and consider the polynomial ring  $R'_{M,N}=\kk[x_1,\dots, x_M,y_1,\dots,y_N]$. The complete supersymmetric function 
\[ h_n(x/y) = \sum_i h_{n-i}(x)e_i(y) 
\] 
is  defined in terms of complete functions in the variables $x$ and elementary  functions in the $y$'s.  By a supersymmetric (or an $(M,N)$-hook) partition $\lambda$ we mean a partition 
\[ \lambda= (\lambda_1,\dots, \lambda_p) 
\]
that fits into an $(M,N)$-hook, see Figure~\ref{fig_5_1}. This means $\lambda_{M+1}\le N$.  Partitions that don't fit into this hook can also  be  considered, but then the corresponding  supersymmetric  Schur  function equals zero. 

\begin{figure}[h]
\begin{center}
\includegraphics[scale=1.0]{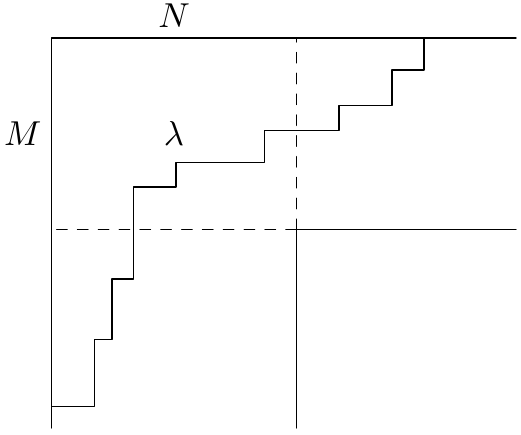}
\caption{An  $(M,N)$-hook partition $\lambda$.}
\label{fig_5_1}
\end{center}
\end{figure}

The hook (or supersymmetric) Schur function  $s_{\lambda}(x/y)$ is defined as the  determinant
\begin{equation} \label{eq_ss_jt}
  s_{\lambda}(x/y) = \det(h_{\lambda_i-i+j(x/y)})_{i,j\le \ell(\lambda)}.
\end{equation}  
This definition works for  any partition  but gives the zero function if $\lambda$ does not  fit  into the hook. Alternatively, one can define $s_{\lambda}(x/y)$ as the  character of a suitable irreducible representation of the Lie superalgebra $\mathfrak{gl}(M|N)$, and then (\ref{eq_ss_jt}) can be viewed as the supersymmetric Jacobi-Trudi formula. 

For $\lambda$ as above, define partitions $\kappa$, $\tau$ and  $\eta$ as shown in Figure~\ref{fig_5_2}.
Denote by $\kappa$ the intersection of $\lambda$  with the $M\times N$ rectangle. What's left after deleting  $\kappa$ from $\lambda$ are the two partitions $\tau$ and $\eta$. Partition $\tau$ is the part of $\lambda$ to the right of the rectangle, while $\eta$ is the part of $\lambda$ below the rectangle. Either one or both may be the  empty partition. 

\begin{figure}[h]
\begin{center}
\includegraphics[scale=1.0]{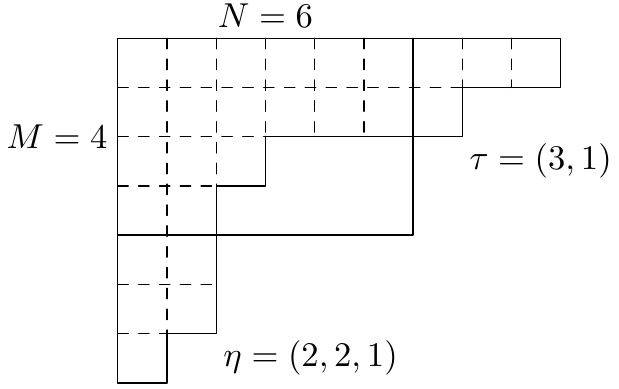}
\caption{Partitions $\kappa$, $\tau$ and  $\eta$ associated to $\lambda=(9,7,3,2,2,2,1)$, with $M=4$ and  $N=6$. Partition  $\kappa=(6,6,3,2)$ is the intersection of $\lambda$ with the  $M\times N$ rectangle. }
\label{fig_5_2}
\end{center}
\end{figure}
For example, for $M=4,N=6$ and $\lambda=(9,7,3,2,2,2,1)$ we have $\kappa=(6,6,3,2)$, $\tau = (3,1),  \eta=(2,2,1),$ see Figure~\ref{fig_5_2}. 

The  Sergeev-Pragacz formula for $s_{\lambda}(x/y)$ is the following, see~\cite{MJ1,Mo}: 
\begin{equation}\label{eq_sp} 
    s_{\lambda}(x/y) = D_0^{-1} \sum_{\sigma \in S_M\times S_N} 
    \mathrm{sgn}(\sigma) \cdot 
    \sigma\left( 
   x^{\tau+\delta_M}y^{\eta'+\delta_N}\prod_{(i,j)\in \kappa} (x_i + y_j) \right) . 
\end{equation}
Here $(i,j)\in \kappa$ iff the box with the row  index $i$ and column  index $j$ belongs  to $\kappa$, and 
\[ D_0 = \prod_{1\le i<j\le M}(x_i-x_j)  
   \prod_{1\le i<j\le N}(y_i-y_j).  
\]
Partition $\delta_M=(M-1,M-2,\dots, 1,0)$ and likewise for $\delta_N$. We denote 
\[ x^{\tau+\delta_M} = x_1^{\tau_1+M-1}x_2^{\tau_2+M-2}
\dots x_M^{\tau_M}
\]
and likewise for $y^{\eta'+\delta_N}$, where  $\eta'$ is the conjugate partition of $\eta$. For $\eta=(2,2,1)$ as above, $\eta'=(3,2)$.

The formula  simplifies when  $\lambda_M\ge N$ and becomes 
\begin{equation}\label{eq_ss_prod} 
    s_{\lambda}(x/y) = s_{\tau}(x)s_{\eta'}(y)\prod_{i=1}^M \prod_{j=1}^N (x_i+y_j). 
\end{equation}
If $\lambda_{M+1}>N$, that is, $\lambda$ does not fit into the $(M,N)$-hook, the supersymmetric Schur function $s_{\lambda}(x/y)=0$. 

\vspace{0.1in}

{\it Overlapping foams.} We assume familiarity with $GL(N)$ foam evaluation, as developed  by Robert and Wagner~\cite{RW1}, also see~\cite[Section 1.2]{KK} for an introduction. In Robert-Wagner theory, a  closed $GL(N)$  foam in $\R^3$ evaluates to a symmetric polynomial in $x_1,\dots, x_N$. Let us explain a naive extension of Robert-Wagner evaluation to two sets  of variables that produces a supersymmetric Schur function for 
a configuration of two overlapping theta-foams.  

Consider a configuration  $F=(F',F'')\subset \R^3$ of a $GL(M)$ foam $F'$ and a $GL(N)$ foam $F''$ in  $\R^3$ that may intersect generically. By a generic intersection we mean the  following.  Choose  any admissible coloring $c'$  of $F'$ and  $c''$ of $F''$. To the coloring $c'$ there is associated a closed surface $F'_i(c')\subset \R^3$ for $1\le i\le M$ which consist of all facets  of $F'$ that contain color $i$ in the coloring  $c'$. Likewise, to $c''$ there is associated closed  surface $F''_j(c'')$, the  union of all facets of $F''$ that contain color $j$ in $c''$. We require that surfaces $F'_i(c')$ and $F''_j(c'')$ intersect generically, along a finite union of circles, for all $i,j,c',c''$ as above.  

An alternative definition requires first defining $GL(N)$ foams $F$ so that at each point $p\in F$  there is a well-defined  tangent plane, including when  $p$  is on  a  seam of $F$ or $p$ is a singular vertex of $F$. The definition  can  be found in~\cite{RW2}, for  instance. Conceptually, one  requires that along each seams of $F$, the thicker facet splits smoothly into two thinner  facets, so that near the seam the  thinner  facets stay  infinitesimally close  to each other.  The same definition ensures 'smoothness' and  a well-defined tangent plane near each singular vertex of $F$. 

With the alternative definition at hand, by an $(M,N)$-foam pair $F=(F',F'')$ we mean a configuration of possibly overlapping $GL(M)$, respectively $GL(N)$, foams $F', F''$ such  that at each intersection  point $p\in F'\cap  F''$ the tangent  planes $T_p(F')$ and $T_p(F'')$ are in general position, that  is, intersect along a  line, 
\[ \dim_{\R}(T_p(F')\cap T_p(F'')) = 1. 
\] 
Given an $(M,N)$-foam $F$, we allow dots of $F'$ and $F''$ to float smoothly on facets of  $F'$ and $F''$ and cross over intersection lines $F'\cap  F''$, as long as each dot stays on  its own facet. Likewise, we allow deformations of $F'$ and $F''$ relative to each other, as long  as at each moment  of the deformation $F'$ and  $F''$ intersect generically as defined above. 

Given two smooth closed surfaces $S',S''\subset \R^3$ that intersect generically in this sense, the intersection $S'\cap  S''$ is a union of  finitely many circles. Define the intersection index $\iind(S',S'')$  as the number of circles in the intersection. The intersection index is symmetric and additive with respect to decomposing $S'$ and  $S''$  into their connected components. Isotopy of  $S'$ and $S''$ that keeps them  intersecting generically at each moment does not change the index. 

We now define evaluation $\angf{F,c}$, where $c=(c',c'')$ consists of a $GL(M)$ coloring $c'$  of $F'$ and a $GL(N)$ coloring $c''$ of $F''$,  by  
\begin{equation} \label{eq_overlap_eval}
    \angf{F,c} = \angf{F',c'} \angf{F'',c''}\cdot \prod_{i=1}^M\prod_{j=1}^N (x_i+y_j)^{\iind(F'_i(c'),F''_j(c''))}.
\end{equation}
Here $\angf{F',c'}$, respectively  $\angf{F'',c''}$,  is the Robert-Wagner evaluation of the $GL(M)$ foam $F'$ at its coloring $c'$, respectively evaluation of the $GL(N)$ foam $F''$ at coloring $c''$. The new term in the formula counts the number of intersection circles of surfaces $F'_i(c')$ and $F''_j(c'')$ and puts it in the exponent  of $x_i+y_j$. 

We also call such $c=(c',c'')$ an (admissible) coloring  $c$ of $F$ or a $GL(M|N)$ coloring of $F$. 
Define the evaluation of $F$ by 
\begin{equation}
    \angf{F} =\sum_c \angf{F,c}, 
\end{equation}
the sum over all admissible coloring $c$ of $F$, that is, over all pairs $(c',c'')$ of colorings as above. 

\vspace{0.1in} 

{\it Theta foam and Schur functions.}
Let $\mu=(\mu_1,\dots, \mu_M)$ be a partition with at most $M$ parts. 
By a $GL(M)$ theta-foam $\Theta_{\mu}$ we mean the $GL(M)$ foam with one disk $M$-facet and $M$ disk $1$-facets attached to it along the common singular circle~\cite{KK}, see Figure~\ref{fig_5_3} left. On the right of the figure we depicted the central cross-section of the foam. The foam can be reconstructed from its cross-section  by taking the suspension of this  diagram and "smoothing out" the north and  south poles of the resulting 2-dimensional CW-complex.  Upon suspension, each interval becomes a disk facet in  the  foam and  the pair of vertices of the cross-section where  thin edges meet the edge  of thickness  $M$ turn into the singular circle of the foam.  

\begin{figure}[h]
\begin{center}
\includegraphics[scale=1.0]{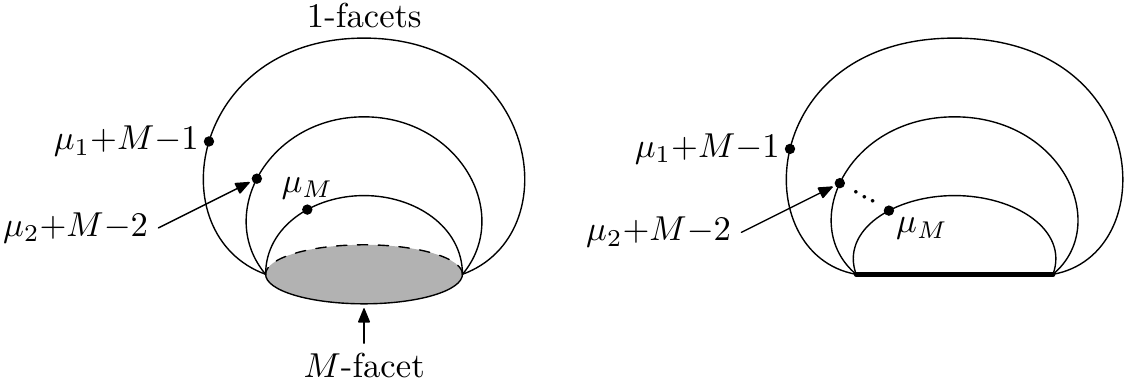}
\caption{$GL(M)$ theta-foam $\Theta_{\mu}$ and  its central cross-section.}
\label{fig_5_3}
\end{center}
\end{figure}
We put $\mu_1+M-1,\mu_2+M-2,\dots, \mu_M$ dots on thin facets as we go cyclically around the singular circle. 
It's immediate to see that, up  to overall sign $(-1)^{M(M-1)/2}$ that depends on the orientation of the foam, theta-foam evaluates to the Schur polynomial $s_{\mu}(x_1,\dots, x_M)$, 
\begin{equation}
    \angf{\Theta_{\mu}} = \pm\sum_{\sigma\in S_M}
    \mathrm{sgn}(\sigma)\frac{x_{\sigma(1)}^{\mu_1+M-1} 
    x_{\sigma(2)}^{\mu_2+M-2}\dots x_{\sigma(M)}^{\mu_M}}{\prod_{i<j}(x_i-x_j)}=\pm s_{\mu}(x_1,\dots, x_M).
\end{equation}
This follows  at once  from the Robert-Wagner evaluation  formula.  Theta-foam admits $M!$ colorings, over all ways to label $M$ thin facets by $\{1,\dots, M\}$. Each surface $F_{ij}(c)$ is homeomorphic to the   2-sphere, contributing  $x_i-x_j$ to the  denominator. The formula follows. 

\vspace{0.1in} 

{\it Overlapping theta-foams.}
Take an  $(M,N)$-supersymmetric partition $\lambda$ as above and consider associated partitions $\kappa,\tau,\eta$ as described earlier. To a partition $\kappa$ we associate a configuration of overlapping $GL(M)$ and  $GL(N)$ theta-foams $F'$ and $F''$ as follows. Foams overlap only along 1-facets. Label 1-facets by $f'_1,\dots, f'_M$ and $f''_1,\dots, f''_N$ going around singular circles of $F'$ and $F''$ in the opposite directions for the two foams. Form the intersection where facets $f'_i$ and  $f''_j$ intersect along a circle iff  the square $(i,j)$ belongs to the  partition $\kappa$. 

The  network of intersections is depicted schematically in Figure~\ref{fig_5_4}, via the intersection  of $GL(M)$ and $GL(N)$ theta graphs. Facets  are represented  as edges, and circles of intersection correspond to  pairs of  opposite  intersection points. 

\begin{figure}[h]
\begin{center}
\includegraphics[scale=1.0]{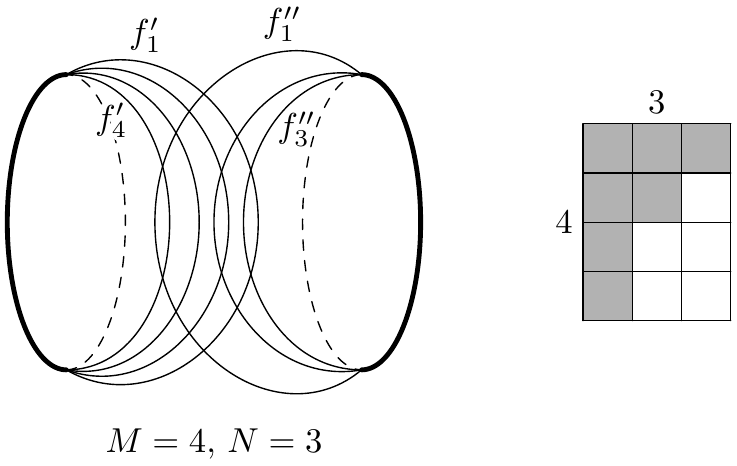}
\caption{Middle cross-section of the $GL(4,3)$ theta foam with the partition $\kappa=(3,2,1,1)$. Pairs  of intersecting edges are  in a bijection with squares of  $\kappa$. Edge representing facet $f'_1$ intersects edges  for  the  three facets $f''_1,f''_2,f''_3$, corresponding to  $\kappa_1=3.$  Edge for $f'_2$  intersects $f''_1,f''_2$, which corresponds to $\kappa_2=2$.  Edges for $f'_3,f'_4$ each intersect $f''_1$ only, and  $\kappa_3=\kappa_4=1$.   }
\label{fig_5_4}
\end{center}
\end{figure}

The foam can be reconstructed from this cross-section by taking the suspensions of the $GL(M)$ and $GL(N)$ theta-graphs and "smoothing  out" two north and two  south poles of the suspensions. Each edge becomes a disk facet and pairs of opposite intersection points turn  into singular circles along which foams $F'$ and $F''$ overlap.  The number of singular circles  equals the size of the  partition $\kappa$. 

\vspace{0.1in}

Recall the  partition 
\[  \tau+\delta_M = (\tau_1 +M-1,\tau_2+M-2, 
\dots, \tau_{M-1}+1,\tau_M)
\]
with at most $M$ parts. We put $\tau_i+M-i$ dots on facets  $f'_i$, $1\le i \le M$. Notice that these $M$ facets necessarily carry distinct numbers of dots. 
Likewise, the partition
\[\eta'+\delta_N = (\eta'_1 +N-1,\eta'_2+N-2, 
\dots, \eta'_{N-1}+1,\eta'_N)
\] 
has at most $N$ parts. We put $\eta'_j+N-j$ dots on facets $f''_j$, $1\le j\le N$. Again, these  $N$ facets 
all carry different numbers of dots. 
An example of overlaps and distribution of dots on facets is shown in Figure~\ref{fig_5_5} for the partition from Figure~\ref{fig_5_2}. 
\begin{figure}[h]
\begin{center}
\includegraphics[scale=1.0]{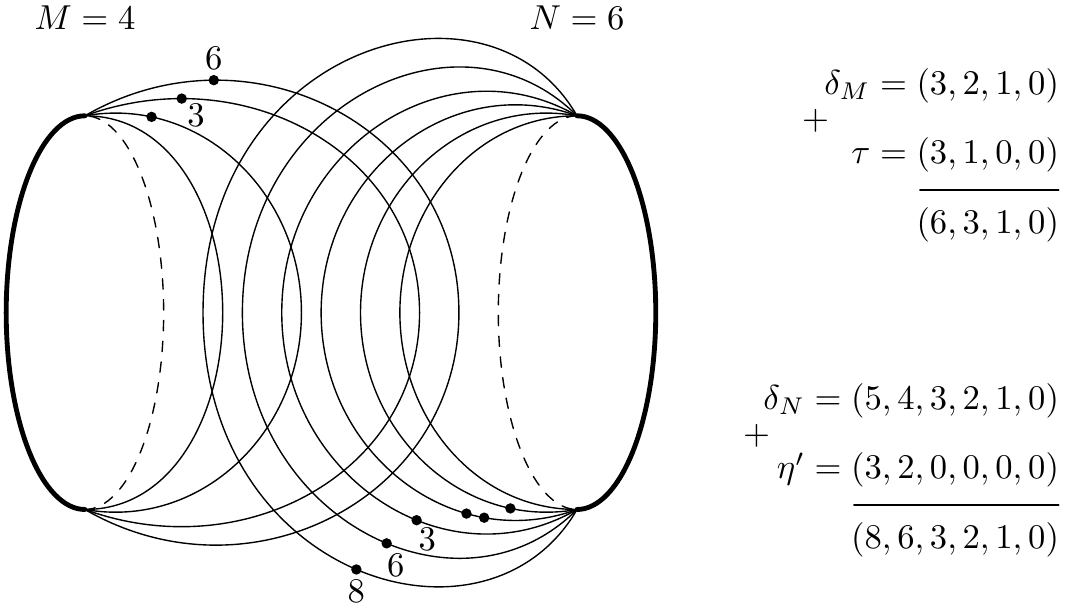}
\caption{$GL(4,6)$ theta-foam for the  partition $\lambda=(9,7,3,2,2,2,1)$ in Figure~\ref{fig_5_2}; $M=4,N=6$. }
\label{fig_5_5}
\end{center}
\end{figure}

Denote this $GL(M|N)$ foam by $\Theta_{\lambda}$. 
\begin{prop} Up to an overall sign, foam $\Theta_{\lambda}$ evaluates to the supersymmetric Schur function, 
  $\angf{\Theta_{\lambda}} = \pm s_{\lambda}(x/y). $
\end{prop}
\emph{Proof} is immediate from  the definition of evaluation. The  sum is  over  colorings  $c'$ and $c''$ of $F'$ and $F''$, respectively. 
Colorings $c'$ correspond  to  elements $\sigma'$ of the  symmetric  group  $S_M$, with facet $f'_i$ colored by $\sigma'(i)$. 
Colorings  $c''$ correspond to elements $\sigma''\in S_N$, with facet  $f''_j$ colored  by $\sigma''(j)$. Intersecting facets $f'_i$ and $f''_j$ contribute $x_{\sigma'(i)}+y_{\sigma''(j)}$ to the expression. 
All surfaces $F'_{ik}(c')$, $1\le i<k\le M$ and $F''_{j\ell}(c'')$, $1\le j<\ell\le N$ are spheres, contributing denominator  $D_0$ to the evaluation, see formula (\ref{eq_sp}). 
The sign $\mathrm{sgn}(\sigma)=\mathrm{sgn}(\sigma')\mathrm{sgn}(\sigma'')$ for $\sigma=(\sigma',\sigma'')$ as in (\ref{eq_sp}) comes from the  count of positive seam circles in these surfaces. 
Consequently, the Sergeev-Pragacz expression coincides with  the evaluation of $GL(M|N)$ foam  $F=(F',F'')$, implying that $\angf{F}=\pm s_{\lambda}(x/y).$ 
Choice of sign depends  on orientations of $F',F''$. Reversing the orientation of $F'$, respectively $F''$, multiplies  the evaluation by  $(-1)^{M(M-1)/2}$, respectively $(-1)^{N(N-1)/2}$. 
$\square$

\vspace{0.1in} 

{\it Overlapping 2-spheres.}
When $\lambda$ contains the $M\times N$ rectangle the formula simplifies and gives the product of Schur functions for $\tau$ and  $\eta$ and linear terms $x_i+y_j$, see equation (\ref{eq_ss_prod}). This  condition on $\lambda$ is equivalent  to the condition that foams $F'$ and $F''$ have maximal overlap. That  is, every pair of 1-facets $(f'_i,f''_j)$ intersects in a circle. When this happens, 
foams $F'$ and  $F''$ can be deformed relative to each other without changing the evaluation so that they intersect along a single circle in the $M$-facet  of $F'$ and $N$-facet of  $F''$. The stack of 1-disks on $F'$ with dots reduces to a single dot on the $M$-facet labelled by the Schur function $s_{\tau}(x)$. The corresponding  stack of 1-disks  on $F''$ reduces to the dot on the $N$-facet labelled by $s_{\eta'}(y)$. Overlapping foams  then reduce to overlapping 2-spheres of thickness $M$  and  $N$ with Schur functions dots on them, see Figure~\ref{fig_5_6} left. In general, a dot on an  $M$-facet may be labelled by a symmetric function in  $M$ variables, usually homogeneous so that the foam has a  well-defined degree. 

\begin{figure}[h]
\begin{center}
\includegraphics[scale=1.0]{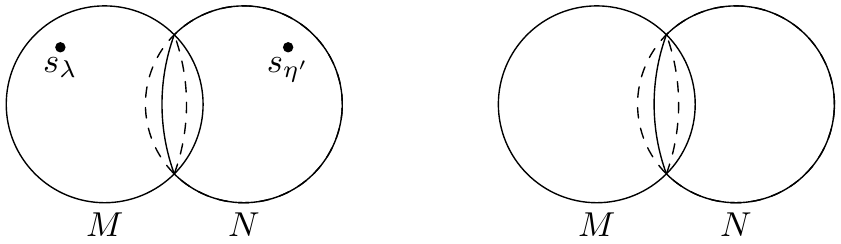}
\caption{On the left: overlapping 2-spheres of thickness $M$ and  $N$ with dots. On the right: same 2-spheres without dots. }
\label{fig_5_6}
\end{center}
\end{figure}

\vspace{0.07in}

{\it Foam evaluations and resultants.}
Without dots, two overlapping $GL(M)$ and $GL(N)$ 2-spheres of maximal thickness ($M$ and  $N$, correspondingly) will evaluate to the product 
\begin{equation}\label{eq_product} 
    \angf{(F',F'')} = \prod_{i=1}^M \prod_{j=1}^N (x_i+y_j). 
\end{equation}
If we change variables $y_j$ to $-y_j$ (or keep the variables and instead change contributions of circle overlaps between surfaces from $x_i+y_j$ to $x_i-y_j$), 
the product on the right can be  interpreted as the resultant of two polynomials. 

Namely, the evaluation takes value in the tensor product ring $\symf_M(x)\otimes \symf_N(y)$ of symmetric polynomials. 
Elements $x_1,\dots, x_M$ of the ring $\kk[x_1,\dots, x_M]$ are roots of the following degree $M$ polynomial with coefficients  in $\symf_M(x)$
\[  f_M(x)= x^M - e_1 x^{M-1}+e_2 x^{M-2}- \dots +(-1)^M e_M,
\] 
where $e_k$ is the  $k$-th  elementary symmetric  function in $x_1, \dots, x_M$.
Elements $y_1,\dots, y_N$ of the ring $\kk[y_1,\dots, y_N]$ are roots of the following degree $N$ polynomial with coefficients  in $\symf_N(y)$
\[  \overline{f}_N(y)= y^N - \ove_1 y^{M-1}+\ove_2 y^{N-2}\dots -\dots +(-1)^N \ove_N,
\] 
where $\ove_k$ is the  $k$-th  elementary symmetric  function in $y_1, \dots, y_N$.

The resultant of $f_M$ and $\overline{f}_N$ is 
\[ \mathrm{Res}(f_M,\overline{f}_N) = \prod_{i=1}^M \prod_{j=1}^N (x_i-y_j) \in \symf_N(x)\otimes \symf_M(y), 
\] 
the product  of differences of pairs of roots of $f_M$ and $\overline{f}_N$.  Individual terms in the product 
belong to the larger ring $\kk[x_1,\dots,x_M,y_1,\dots,y_N] .$ Adding a minus sign to indicate that we use $x_i-y_j$ factors in our evaluation of overlapping foams instead of  $x_i+y_j$ above, we can write 
\begin{equation}\label{eq_resultant} 
    \angf{(F',F'')}_- =\mathrm{Res}(f_M,\overline{f}_N). 
\end{equation}

It would be interesting to see whether this analogy between foam evaluation and resultants can be pushed further. 

\vspace{0.07in} 

There are  well-known analogies between knots and 3-manifolds on one side and primes and number fields (and functional fields) on the other side, see~\cite{KaS,Mr1,Mr2,Ma} and references therein, sometimes referred  to as \emph{arithmetic topology}. 
In that analogy, resultants of pairs of polynomials, quadratic residues, and their generalizations play the 
role of linking numbers. In formula (\ref{eq_resultant}) the resultant is the evaluation of overlapping foams that are 2-spheres of maximal thickness ($M$ and $N$, correspondingly). These two 2-spheres in $\R^3$ overlap along  a circle, and there is an intuitive way in which they may be viewed as linked. One  may hope that developing connections between this type of foam evaluations for linking of foams in $\R^3$, resultants, and related structures may provide an additional outlook on fascinating yet mysterious 3-manifolds vs number fields analogy (the observation that objects in both stories have homological dimension three provides partial but not fully satisfactory  
explanation for some of the similarities). 

\vspace{0.07in} 

A more immediate question for the evaluation in (\ref{eq_overlap_eval}) is whether it's integral for more general foams. A modification of the Robert-Wagner evaluation studied in~\cite{KKKo} restricts to an interesting integral evaluation for the theta-foam, but is not integral on all foams, and this situation may occur with evaluation (\ref{eq_overlap_eval}) as well. In both cases one can then ask whether integrality can be restored for evaluations of arbitrary foams via a more subtle formula. 

\vspace{0.1in} 

{\it Multi-type overlapping foams and their evaluations:}
From the definition of overlapping foam evaluation given by formula (\ref{eq_overlap_eval}) one derives the skein relation in Figure~\ref{fig_5_7} for pulling apart 1-facets of different types. 

\begin{figure}[h]
\begin{center}
\includegraphics[scale=1.0]{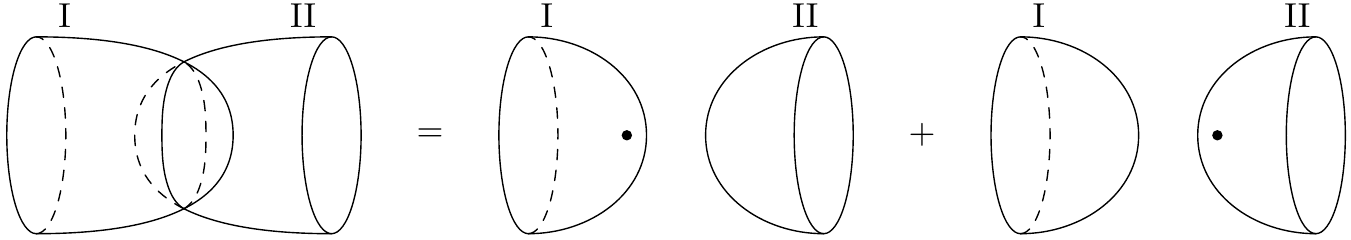}
\caption{Skein relation for putting apart overlapping 1-facets of different types. }
\label{fig_5_7}
\end{center}
\end{figure}

Figure~\ref{fig_5_7} relation is reminiscent of one of the defining relations in the  categorified quantum groups for strands associated to simple roots connected by a single edge in the Cartan graph, see Figure~\ref{fig_5_8}.

\begin{figure}[h]
\begin{center}
\includegraphics[scale=1.0]{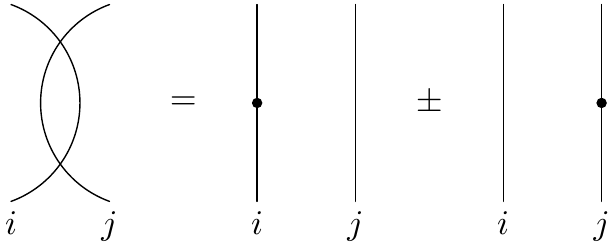}
\caption{A defining relation in the categorified quantum $sl(3)$, see~\cite{KL1,R}. If the minus sign is used, the edge connecting $i$ and $j$ in the Coxeter graph needs to be oriented~\cite{BK}.}
\label{fig_5_8}
\end{center}
\end{figure}

This is likely not a coincidence. There is a relation between the quotients of the categorified quantum $GL(n)$ calculus and $GL(k)$ foam theory, for different $n$ and $k$, see~\cite{MSV,LQR,QR}. We expect a similar lifting of categorified quantum $GL(n)$ diagrammatical calculus into a version of $GL(M|N)$ foams and their generalizations. 

Notice that components in $GL(M|N)$ foams are split into two types and colorings of components happen independently, with $M$ and $N$ colors used in components of type $I$ and $II$, respectively. There is a natural generalization where one starts with a simply-laced graph $G$ and an assignment of non-negative integers $N_i$ to vertices $i$ of $G$.
One then considers overlapping foams where components are labelled by vertices $i$ of graph $G$, and  facets of $i$-components may have thickness from $1$ to $N_i$. One uses $N_i$ colors $\{1,\dots, N_i\}$ and subsets of this set to  color facets of $i$-components. In the evaluation one uses variables $x_{i,1},\dots, x_{i,N_i}$ to write down contributions from $i$-components. 

\vspace{0.1in} 

Taking the union of $i$-components'  facets colored by $u$, $1\le u\le  N_i$, results in  a closed surface $F_{i,u}(c)$ in $\R^3$. 
If vertices $i$ and $j$ are connected by an edge in $G$, then each circle in the intersection $F_{i,u}(c)\cap  F_{j,v}(c)$ of surfaces $F_{i,u}(c)$ and $F_{j,v}(c)$ from $i$ and $j$ components of foam $F$, with $1\le u \le N_i$ and  $1\le v\le N_j$ contributes the term $x_{i,u} + x_{j,v}$ to the product for the evaluation $\angf{F,c}$. If edges of $G$ are oriented, with $(i,j)$ edge oriented from $i$ to  $j$, it's natural to use the  contribution $x_{j,v}-x_{i,u}$ instead. 

If vertices $i$ and $j$ are not connected by an edge, overlaps of $i$ and $j$ components make no contribution to the evaluation. 
These components in the embedded foam may be  arbitrarily deformed against each other without changing the evaluation. 

Evaluation (\ref{eq_overlap_eval}) admits this straightforward extension to an arbitrary  simply-laced Coxeter graph $G$. It may need to be further tweaked to achieve integrality. One can then  expect  to interpret KLR algebras for simply-laced diagrams and related structures via overlapping foam evaluation. 

Once $G$ has at least three vertices, for each such triple $(i,j,k)$ of vertices one can include in the evaluation the count of triple  intersections of surfaces $F_{i,u}(c),F_{j,v}(c), F_{k,w}(c)$ from $i$, $j$, and $k$ components of foam $F$, via an additional variable $t_{i,j,k}$ for each such triple, perhaps encoded by a power series in $x_i,x_j,x_k$ (similar power series in two variables are used in foam evaluation deformations in~\cite{KK,KKKo}). Triple intersections of $i,j,$ and  $k$ surfaces of colors $u,v,$ and $w$, respectively, would then contribute $t(x_{i,u},x_{j,v},x_{k,w})^{s/2}$ to the  product for a given coloring $c$, where $s$ is the unsigned count of these triple intersections. Scott Carter and Masahico Saito pointed out that the number of such intersections of a triple of closed surfaces embedded in $\R^3$ is even, hence $s/2$ in the exponent above, and intersection points come with signs once surfaces are oriented~\cite{CS}. Simply-laced Coxeter graph $G$ can  then  be upgraded to a decorated 2-dimensional CW-complex with 2-simplices that encode parameters for these new variables. 


\vspace{0.07in} 

Evaluation (\ref{eq_overlap_eval}) may also be  generalized by replacing the two-variable polynomial $x\pm y$ (specializing to $x_i\pm y_j$ in  the evaluations) by more general polynomials $g(x,y)$ in two variables, see Figure~\ref{fig_5_9}, which are  more  complicated reductions of innermost circle overlaps between facets of different types.  
In the language of Coxeter-Dynkin graphs (or diagrams), this corresponds to allowing multiple edges between a pair of vertices.  On the categorified quantum groups side, generalizing the relation in Figure~\ref{fig_5_8} to more general two-variable polynomials, see Figure~\ref{fig_5_10}, leads to the non-simply-laced case~\cite{KL2,R,KaK}.  

\begin{figure}[h]
\begin{center}
\includegraphics[scale=1.0]{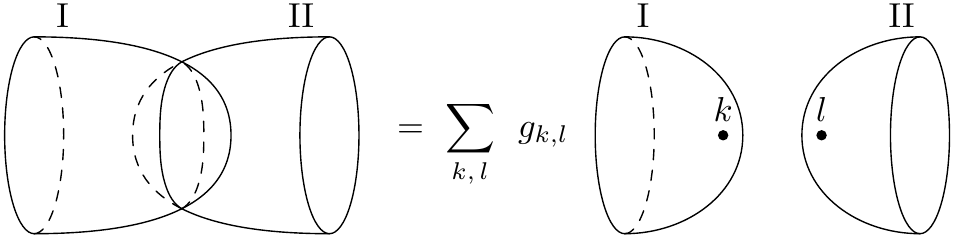}
\caption{Pulling apart 1-facets of different types in a more general case. Coefficients $g_{k\ell}$ encode the polynomial $g(x,y)=\sum_{k,\ell} g_{k\ell} x^k y^{\ell}$.}
\label{fig_5_9}
\end{center}
\end{figure}

\begin{figure}[h]
\begin{center}
\includegraphics[scale=1.0]{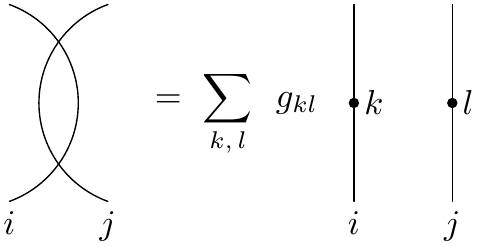}
\caption{Pulling apart $i$ and $j$ strands in categorified quantum group diagrammatics in the non-simply-laced case.}
\label{fig_5_10}
\end{center}
\end{figure}

These observations hint at a substantial theory of generalized foam evaluations waiting to be developed, beyond the one-type case that has been heavily used over the last 15 or so years to understand $GL(N)$ link homology for $N>2$ (when $N=2$ foams can be avoided, for the most part) and got a fully combinatorial description via the Robert-Wagner formula~\cite{RW1}. We refer the reader to~\cite{KK} for more references. 

Strands of different types stand out in the construction of Webster algebras and associated link homology theories~\cite{We1}. Relation in Figure~\ref{fig_5_8} with a minor variation appears in the redotted Webster algebra case, see~\cite[Section  4.2]{KLSY} and~\cite{We2}, and an approach to these algebras, bimodules and associated link invariants via multi-type overlapping foam evaluations may exist as well.

Overlapping foams in the two-type case are also expected to relate to knot Floer homology and Heegaard-Floer homology. Various DG-algebra approaches to the latter exhibit analogues of the skein relations in Figures~\ref{fig_5_8} and~\ref{fig_5_10} and likely admit a suitable foam description. Looking further ahead, basic examples in  this section indicate that such overlapping foams may help to categorify $GL(M|N)$ quantum invariants beyond $M=N=1$ case. 

\vspace{0.1in}

{\it Michael Day's formula and foams.} Closely related to the   structures discussed in this paper is the formula due to Michael Day for the Toeplitz determinant of the  Laurent expansion  of a rational function~\cite{Da,HJ}. Toeplitz determinants  of rational functions appear  throughout  Section~\ref{sec_univ_con_two} above. Day's formula can be interpreted via overlapping foam evaluation as well. Below we use notations from papers~\cite{Da,HJ}, which differ from our earlier notations.  

\begin{theorem}(\cite[Theorem 3.1]{Da} and~\cite[Theorem 4.1]{HJ}) Let $R_1$ and $R_2$ be real numbers such that $0\le R_1<R_2$. Let $D(z)$ be a complex polynomial of degree $k$ with roots $\delta_1,\dots, \delta_k$ satisfying $|\delta_i|\le R_1$, and $F(z)$ a polynomial of degree $h$ with roots $\rho_1,\dots, \rho_h$ satisfying $|\rho_j|\ge R_2$. Let $G(z)$ be a polynomial of degree $p$ with distinct roots $r_1,\dots, r_p$. Normalize the polynomials so that 
\[ D(z)=\prod_{j=1}^k (z-\delta_j), \ \ F(z) = \prod_{j=1}^h (1- \rho_j^{-1}z), \ \ G(z)= \prod_{j=1}^p (z-r_j).
\]
Let $\sum_{\nu=-\infty}^{\infty}a_{\nu}z^{\nu}$ be the Laurent expansion of 
\[ f(z) =\frac{G(z)}{F(z)D(z)}
\]
in the annulus $\{z\in\C | R_1<|z|<R_2\}.$ 
Let $T_n(f)=(a_{i-j}),$  $i,j=0,1,\dots, n$ be the Toeplitz  matrix. Then  if $p=k+m$, $m\ge h$, 
\begin{equation}
  \label{eq_det_day}
    \det T_n(f) =  (-1)^{m(n+1)} \sum_I  \left( T(I)\cdot \prod_{i\in I } r_i^{n+1} \right) ,\ 
\end{equation}
where the sum is over all $m$-element subsets $I$ of $\{1,2,\dots, k+m\}$, $\overline{I}=\{1,2,\dots, k+m\}\setminus I$, and 
\begin{equation}
  \label{eq_eval_d}
    T(I) :=   
    \prod_{\substack{i\in I \\ j \in \overline{I} }}(r_i-r_j)^{-1}\cdot 
    \prod_{\substack{i\in I \\  s\in \{1,\dots, k\}}} 
    (r_i-\delta_s) \cdot   
    \prod_{\substack{j\in \overline{I} \\ t\in \{1,\dots, h\} }}(\rho_t-r_j)\cdot 
    \prod_{\substack{t\in \{1,\dots, h\} \\ s\in \{1,\dots, k\}}} (\rho_t - \delta_s)^{-1}.
\end{equation}
\end{theorem} 
Notice that  cross-ratios  
\[(r_i,\delta_s,\rho_t,r_j) \ := \  \frac{(r_i-\delta_s) (\rho_t - r_j )}{(r_i-r_j)(\rho_t - \delta_s)}
\] 
feature prominently in  this formula. The product $T(I)$ can be interpreted as a sort of distributed cross-ratio, where indices of each of the  four families of variables  $r_i,r_j,\delta_s,\rho_t$ are parametrized by elements of finite sets $I,\overline{I},\{1,\dots,k\},\{1,\dots, h\}$, respectively, and one takes a product of differences of variables or their inverses over the four edges of the square below. For each edge, the  product is over all ways  to select a pair of elements, one from each set assigned to vertices of the square. This  can  be depicted diagrammatically by a decorated square in Figure~\ref{fig_6_1}. 

\begin{figure}[h]
\begin{center}
\includegraphics[scale=1.0]{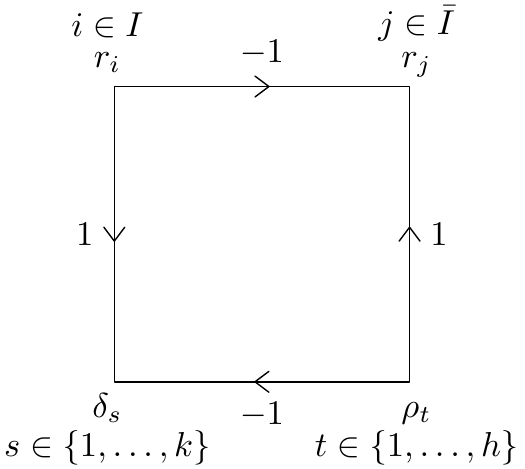}
\caption{Square that encodes the product $T(I)$. Orientations point toward variables  that appear with the minus sign  in  the corresponding differences. Sets to choose the index from are written next to each vertex. Numbers $1$ and  $-1$ on edges of the square indicate the  exponent with which the difference appears}
\label{fig_6_1}
\end{center}
\end{figure}

To give an overlapping foam interpretation of Day's formula, start with a   version of the  $GL(k+m)$ theta-foam which  consists of three disks of  thickness $k,m,k+m$ glued together along the common circle and  use variables $r_1, \dots, r_{k+m}$ instead of the customary $x_1, \dots, x_{k+m}$. Place a dot labelled by the power $e_m(r)^{n+1}$ of the $m$-th elementary symmetric function $e_m(r)$ in variables $r_i$ on the $m$-facet. Consider this  foam $F'$ as being of type I, see Figure~\ref{fig_5_11} left.  

Now add a foam $F''$ of  type II  which is a 2-sphere with two  disk facets  of thickness $1$, glued along a defect circle (or singular circle), see  Figure ~\ref{fig_5_11} right. To these two facets we associate variables  $\delta_1,\dots, \delta_k$ and $\rho_1, \dots, \rho_h$, respectively, and refer to the facets as the $\delta$-disk and the $\rho$-disk.  

\begin{figure}[h]
\begin{center}
\includegraphics[scale=1.0]{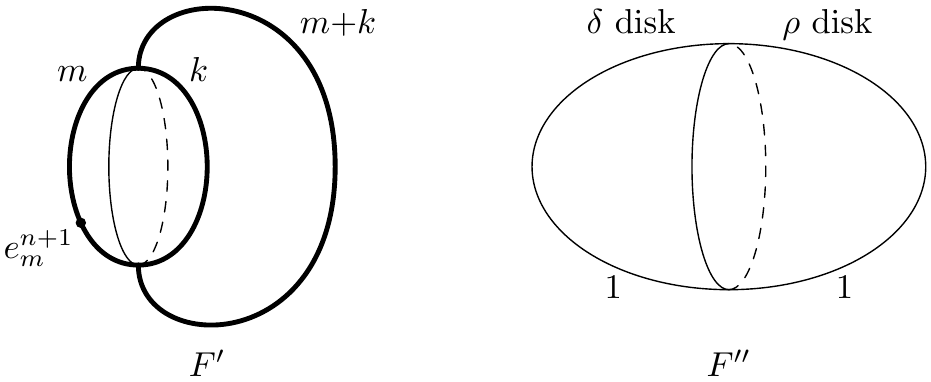}
\caption{Left:  foam $F'$.  Right: foam $F''$. Numbers next to  facets indicate their thickness. Foam $F'$ is shown schematically, via its theta graph cross section. It has a single dot on the $m$-facet carrying label $e_m^{n+1}$.}
\label{fig_5_11}
\end{center}
\end{figure}

Position foams $F'$ and $F''$ to  overlap along two circles so  that $m$-facet of $F'$ intersects $\delta$-facet of $F''$ and  $k$-facet of $F'$ intersects $\rho$-facet of  
$F''$, see Figure~\ref{fig_5_12}. Denote the resulting foam by $F$.   
\begin{figure}[h]
\begin{center}
\includegraphics[scale=1.0]{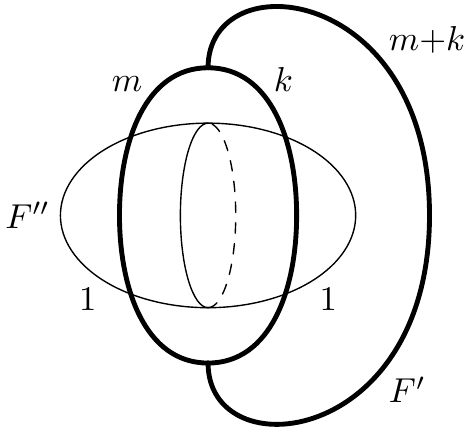}
\caption{Overlapping foams $F'$ and $F''$ form foam $F$.}
\label{fig_5_12}
\end{center}
\end{figure}

To evaluate $F$, we use variables $r_i$'s for $F'$ and variables $\delta_s$ and  $\rho_t$ for $F''$. A coloring $c$ of $F$ consists of a coloring $c'$ of $F'$ and a coloring $c''$ of $F''$, so that $c=(c',c'')$. 

The  largest disk of $F'$ has maximal thickness, so colorings $c'$ of  $F'$ are in a bijection with $m$-element subsets $I$ of $\{1,2,\dots, k+m\}$, indicating the subset assigned to the  $k$-facet of  $F'$. The complementary subset  $\overline{I}$ is assigned to the  $m$-facet of $F'$. For $i\in I$ and  $j\in \overline{I}$ the surface $F_{ij}(c')$ is a 2-sphere, contributing $(x_i-x_j)^{-1}$ to the evaluation $\angf{F,c}$, where $c=(c',c'')$ is a coloring of $F$.

\begin{figure}[h]
\begin{center}
\includegraphics[scale=1.0]{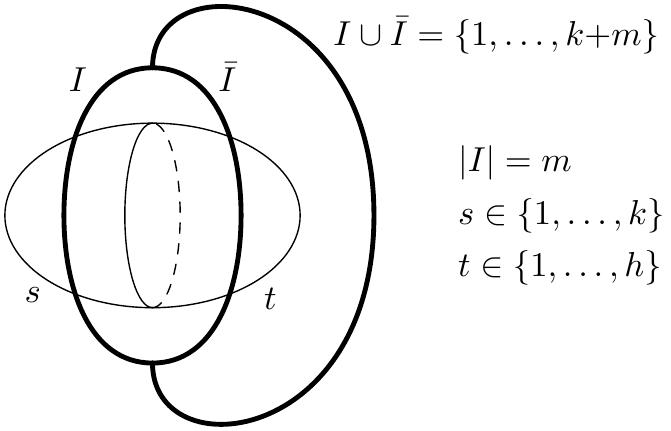}
\caption{Colorings of $F$.}
\label{fig_5_13}
\end{center}
\end{figure}

A coloring  $c''$ of the 2-sphere $F''$ consists of choosing
color $s\in \{ 1,\dots, k\}$ to assign to the $\delta$-disk of $F''$ and color $t\in \{1,\dots, h\}$ to assign to the $\rho$-disk. Both disks, glued along a singular circle, have thickness one. We declare that this coloring  contributes $(\rho_t-\delta_s)^{-1}$ to the evaluation of the foam. This is partially a guess to fit the Day's formula to the foam evaluation framework. It's also motivated by the observation that entries of the form $(x_i\pm y_j)^{-1}$ are common as matrix entries and contributions to related determinants that may carry foam  evaluation interpretation as well, including the Cauchy's double alternant~\cite{Kr} and Moens and Van der Jeugt's determinant for the supersymmetric Schur function~\cite[Theorem 3.4]{MJ1}. 

Furthermore, a similar situation occurs in~\cite{RW2}, where the authors convert an evaluation and state space from a set of variables $\{X_1,\dots, X_k\}$ to the set of variables $\{T_1,\dots, T_N\}$, via division by products $\prod_{j=1}^N(X_i-T_j)$, which can be interpreted with a residue formula having the product of these monomials in the denominator. One natural intepretation of their construction is via an evaluation of a 2-sphere foam glued out of two disks along a defect circle, with the  $X_i$ variables assigned to one disk and $T_j$ variables assigned to the other (in that example the two glued disk facets have thickness $k$ and $N$ rather  than thickness $1$ and $1$ as in Figure~\ref{fig_5_12}). Again, it results in denominators that are products of $X_i-T_j$. 
Putting $X_i-T_j$ in the denominator as the contribution from the 2-sphere with two  disks colored $i$ and  $j$, correspondingly, is also analogous to the original Robert-Wagner evaluation, where a 2-sphere component of the $F_{ij}(c)$ surface contributes $\pm (x_i-x_j)^{-1}$, with  additional sign contribution coming from the singular $(i,j)$-circles on the bicolored 2-sphere.

\begin{figure}[h]
\begin{center}
\includegraphics[scale=1.0]{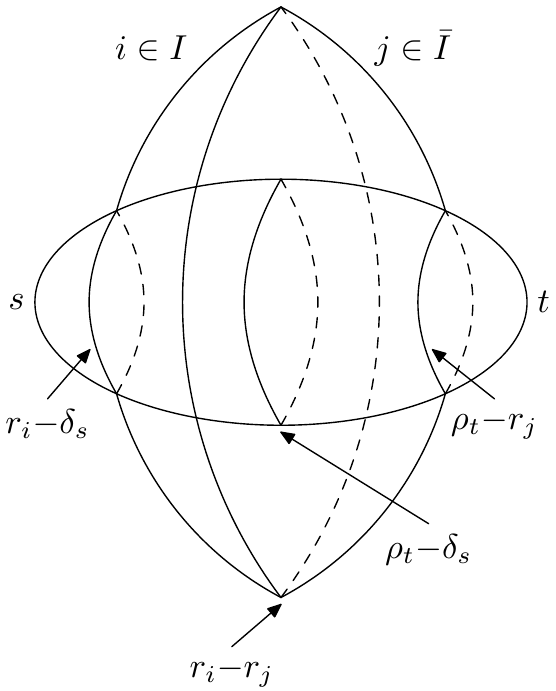}
\caption{Computing the evaluation.}
\label{fig_5_14}
\end{center}
\end{figure}

Going back to our evaluation and looking at Figure~\ref{fig_5_14},
intersection circle of the $m$-facet of $F'$ and the $\delta$-facet of $F''$ contributes the product 
\[\prod_{i\in I} (\delta_s-r_i)=(-1)^m \prod_{i\in I} (r_i-\delta_s)
\] 
to the evaluation $\angf{F,c}$ for a given coloring $c$. In  this normalization, intersection of a type II facet carrying variable $a$ and type I facet carrying variable $b$ contributes $a-b$ to  the product. The  opposite normalization, contributing $b-a$, would only modify the overall evaluation  by a sign, since for this foam the  number of (I,II) intersecting circles is the same for all colorings and equals $m+k$. 
Intersection circle of the $k$-facet of $F'$ and the $\rho$-facet of $F''$ contributes the product 
$\prod_{j\in \overline{I}} (\rho_t-r_j)$ to the evaluation. 
We chose to use differences $\delta_s-r_i$ and $\rho_t-r_j$ in this definition rather than their negations to match the term $(-1)^m$ in the sign of (\ref{eq_det_day}) instead of getting $(-1)^k$. 

Putting everything together, for a given coloring $c$ the evaluation is   
\begin{equation}
  \label{eq_eval_d2}
    \angf{F,c} =  (-1)^{m} \cdot T(I)\cdot 
    \prod_{{i\in I }}  r_i^{n+1} 
    ,
\end{equation}
with $T(I)$ given by formula (\ref{eq_eval_d}). Consequently, 
\begin{equation}
    \det T_n(f) = (-1)^{mn} \angf{F}, 
\end{equation}
recovering Day's expression for the Toeplitz determinant of a rational function Laurent series via foam evaluation as proposed here. 

\vspace{0.1in}




\begin{thebibliography}{alpha}
\raggedright

\bibitem[A]{A} L.~Abrams, {\sl Two-dimensional topological quantum field theories and Frobenius algebra}, J. Knot. Th. Ramif. {\bf 5}, (1996),  569–587. 
\bibitem[AGTW]{AGTW} J.~Y.~Abuhlail, J.~Gomez-Torrecillas, and R.~Wisbauer, {\sl Dual coalgebras of algebras over commutative rings}, Journal Pure Appl. Algebra {\bf 153} (2000), 107-120. 

\bibitem[Al]{Al} G.~Almkvist, {\sl K-theory of endomorphisms}, Journal  of Algebra {\bf  55} (1978), 308-340. 

\bibitem[At]{At} M.~F.~Atiyah, {\sl Topological quantum field theory}, Publ. Math. IHES, tome {\bf 68} (1988), 175-186. 

\bibitem[BHMV]{BHMV} C.~Blanchet, N.~Habegger, G.~Masbaum, and P.~Vogel, {\sl Topological quantum field theories derived from the Kauffman bracket}, Topology {\bf 34} 4 (1995), 883-927.  

\bibitem[BS]{BS} A.~B\"ottcher and B.~Silbermann, {\sl Introduction to large truncated Toeplitz matrices}, 
Universitext, Springer, 1999. 

\bibitem[BK]{BK} J.~Brundan and A.~Kleshchev, {\sl Blocks of cyclotomic Hecke algebras and Khovanov-Lauda algebras}, Invent. Math. {\bf 178} (2009), 451-484. 

\bibitem[CFW]{CFW} D.~Calegari, M.~H.~Freedman, and K.~Walker, {\sl Positivity of the universal pairing in 3 dimensions}, Journal of the AMS {\bf 23} 1 (2010), 107-188. 

\bibitem[CS]{CS} S.~Carter and  M.~Saito, {\sl Private communication}, 2020. 

\bibitem[Ch]{Ch} T.~S.~Chihara, {\sl An introduction to orthogonal polynomials}, Dover 2011, orig. publ. New York: Gordon and Breach, 1978. 

\bibitem[CG]{CG} W.~Chin and J.~Goldman, {\sl Bialgebras of linearly recursive sequences}, Comm. Algebra {\bf  21}, 11 (1993), 3935-3952. 

\bibitem[Da]{Da} K.~M.~Day, {\sl Toeplitz matrices generated by the Laurent series expansion of an  arbitrary rational  function}, Trans. AMS {\bf 206} (1975), 224-245. 

\bibitem[DNR]{DNR} S.~D\u{a}sc\u{a}lescu, C.~N\u{a}st\u{a}sescu, and \c{S}.~Raianu, {\sl Hopf algebras: An introduction}, Pure and Appl. Math. {\bf 235}, Marcel Dekker, 2001. 

\bibitem[Di]{Di} B.~Dickman, {\sl On a theorem of Dwork}, BA Thesis, Amherst College, 2008. 

\bibitem[Dw]{Dw} B.~Dwork, {\sl On the rationality of the zeta function of an algebraic variety}, Amer. J. Math. {\bf 82}, n.3 (1960), 631-648. 

\bibitem[EPSW]{EPSW} G.~Everest, A.~van der Poorten, I.~Shparlinski, T.~Ward, {\sl Recurrence Sequences}, Math. Surv.  Monographs  {\bf  104}, AMS 2003. 

\bibitem[Fa]{Fa} P.~Falb, {\sl Methods of algebraic geometry in control theory: Part I. Scalar linear systems and affine algebraic  geometry}, Birkh\"auser  1990, reprinted by Springer 2018. {\sl Part II. Multivariable linear systems and projective algebraic geeometry}, Birkh\"auser 1999, reprinted by Springer 2018. 

\bibitem[FKNSWW]{FKNSWW} M.~H.~Freedman, A.~Kitaev, C.~Nayak, J.~K.~Slingerland, K.~Walker and Z.~Wang, {\sl Universal manifold pairings and positivity}, Geometry and Topology {\bf 9} (2005), 2303-2317. 

\bibitem[Fr]{Fr} M.~H.~Freedman, {\sl Quantum gravity via manifold positivity}, In: Pardalos P., Rassias T. (eds) Essays in Mathematics and its Applications, 111-140, Springer 2012. 

\bibitem[Fh]{Fh} P.~A.~Fuhrmann, {\sl A polynomial approach to linear algebra}, 2nd ed., Universitext, Springer, 2010. 

\bibitem[G]{G} L.~Gemignani, {\sl Hankel matrix},  in ed.~M.~Hazewinkel, Encyclopedia of Mathematics, Supplement  vol.~III, Kluwer A.P. (2001), 185-187.

\bibitem[GKM]{GKM} M.~Goresky, R.~Kottwitz, and R.~MacPherson, {\sl Equvariant cohomology, Koszul duality and the localization theorem}, Invent. Math. {\bf 131} (1998), 25-83. 

\bibitem[Ha]{Ha} M.~Hazewinkel, {\sl Cofree coalgebras and multivariable recursiveness}, J. Pure  Appl. Algebra {\bf 183} (2003), 61-103. 

\bibitem[HJ]{HJ} T.~H{\o}holdt and J.~Justesen, {\sl Determinants of a class of Toeplitz  matrices}, Mathematica Scandinavica  {\bf 43}, n.~2 (1979),  
250-258. 

\bibitem[HSP]{HSP} {\sl Hook Schur polynomials}, \url{https://www.math.upenn.edu/~peal/polynomials/superSymmetricSchur.htm}. 

\bibitem[IK]{IK} A.~Iarrobino and V.~Kanev, {\sl Power sums, Gorenstein algebras, and determinant loci}, Lect. Notes in Math. {\bf 1721}, Springer  1999. 

\bibitem[JRV]{JRV} F.~Jouve and F.~Rodriguez Villegas, {\sl On the bilinear structure associated to Bezoutians}, J. of Algebra {\bf 400} (2014), 161-184. 

\bibitem[KaK]{KaK} S.-J.~Kang, M.~Kashiwara, {\sl Categorification of highest weight modules via Khovanov-Lauda-Rouquier algebras},  Invent. Math. {\bf 190} (2012), 699–742. 

\bibitem[KaS]{KaS} M.~Kapranov and  A.~Smirnov, {\sl Cohomology determinants and reciprocity laws: number field case,} Preprint Series 1912, Inst. für Experimentelle Mathematik, Essen, 1995. Available online.  

\bibitem[Kh1]{Kh1} M.~Khovanov, 
{\sl sl(3) link homology},  Alg. Geom. Top. {\bf 4} (2004), 1045–1081.

\bibitem[Kh2]{Kh2} M.~Khovanov, 
{\sl Link homology and Frobenius extensions}, Fundamenta Math. {\bf 190} (2006), 179-190. 

\bibitem[KK]{KK} M.~Khovanov and  N.~Kitchloo,  {\sl A deformation of Robert-Wagner foam evaluation and link homology}, arXiv:2004.14197. 

\bibitem[KKKo]{KKKo} M.~Khovanov, N.~Kitchloo, and Y.~Kononov, {\sl In preparation.}

\bibitem[KKo]{KKo} M.~Khovanov and Y.~Kononov, {\sl Work  in progress.}

\bibitem[KL1]{KL1} M.~Khovanov and  A.~D.~Lauda, {\sl A diagrammatic  approach to  categorification of quantum groups I}, Representation  Theory {\bf 13} (2009) 309-347. 

\bibitem[KL2]{KL2} M.~Khovanov and  A.~D.~Lauda, {\sl A diagrammatic  approach to  categorification of quantum groups II}, Transactions of the AMS {\bf 363}, n.~5 (2011), 2685-1700. 

\bibitem[KLSY]{KLSY} 
M.~Khovanov, A.~D.~Lauda, J.~Sussan, and  Y.~Yonezawa, {\sl Braid group actions from categorical symmetric Howe duality on deformed Webster algebras}, arXiv:1802.05358. 

\bibitem[KR]{KR} M.~Khovanov and L.-H.~Robert, {\sl Link homology and Frobenius extensions II}, arXiv:2005.08048. 

\bibitem[Kb]{Kb} N.~Koblitz, {\sl p-adic Numbers, p-adic Analysis, and Zeta-Functions}, Grad. Texts in Math. 2nd ed., Springer 1984. 

\bibitem[Kc1]{Kc1} J.~Kock, {\sl Frobenius algebras and 2d topological quantum field theories}, Cambridge U. Press, Cambridge, 2004.

\bibitem[Kc2]{Kc2} J.~Kock, {\sl Frobenius algebras and 2d topological quantum field theories (short version)}, \href{http://mat.uab.es/~kock/TQFT/FS.pdf}{http://mat.uab.es/~kock/TQFT/FS.pdf}. 

\bibitem[Ko]{Ko}  Y.~Kononov, {\sl Private communication}, June  2020. 

\bibitem[Kr]{Kr} C.~Krattenthaler, {\sl Advanced determinant calculus}, in Foata D., Han GN. (eds) The Andrews Festschrift, Springer (2001), 349-426. 

\bibitem[KT]{KT} M.~Kreck and P.~Teichner, {\sl Positivity of topological field theories in dimension at least 5}, Journal of topology {\bf 1} (2008), 663-670.

\bibitem[Ku]{Ku} V.~L.~Kurakin, {\sl Hopf algebra dual to a polynomial algebra over a commutative ring}, Mathematical Notes {\bf 71}, 5 (2002), 617-623. 

\bibitem[L]{L} A.~Lascoux, {\sl Symmetric functions and combinatorial operators on polynomials}, CBMS Regional Series in Math. {\bf 99}, AMS, 2003. 

\bibitem[LT]{LT} R.~G.~Larson and E.~J.~Taft, {\sl The algebraic structure  of linearly  recursive sequences under Hadamard product}, Israel J. Math. {\bf 72}, nos.1-2 (1990), 118-132. 

\bibitem[LQR]{LQR} A.~D.~Lauda, H.~Queffelec and D.~E.~V.~Rose, {\sl Khovanov homology is a skew Howe 2-representation of categorified quantum sl(m)}, Alg. Geom. Top. {\bf 15} (2015), 2515-2606. 

\bibitem[LB]{LB} L.~Le Bruiyn, {\sl Linearly  recursive sequences and $Spec(\Z)$ over $\mathbb{F}_1$}, Comm. Algebra {\bf 45}, 7 (2017), 3150-3158. 

\bibitem[Mc]{Mc} I.~G.~Macdonald, {\sl Symmetric functions and Hall polynomials}, Oxford University Press, 2nd edition, 1998. 

\bibitem[MSV]{MSV} M.~Mackaay, M.~Stosic and P.~Vaz, {\sl A diagrammatic categorification of the q-Schur algebra}, Quantum Topology {\bf 4} (2013), 1-75. 

\bibitem[MV]{MV} M.~Mackaay and  P.~Vaz, {\sl The universal sl3-link homology}, Alg. Geom. Top. {\bf 7} (2007), 1135-1169.

\bibitem[MMS]{MMS} E.~A.~Maximenko and M.~A.~Moctezuma-Salazar, {\sl Cofactors and eigenvectors of banded Toeplitz matrices: Trench formulas  via skew Schur polynomials}, Operators and Matrices, 1149-1169,  \url{https://doi.org/10.7153/oam-2017-11-79}. 

\bibitem[Ma]{Ma} B.~Mazur, {\sl Primes, Knots and Po}, Lecture notes for  the conference  'Geometry, Topology and Group Theory' in honor of the 80th birthday of Valentin Poenaru, (2012). 

\bibitem[Mo]{Mo} E.~M.~Moens, {\sl Supersymmetric Schur functions and Lie  superalgebra  representations}, PhD Thesis, Universiteit Gent, 2007. 

\bibitem[MJ1]{MJ1} E.~M.~Moens and J.~Van der Jeugt, {\sl A determinantal formula for supersymmetric Schur polynomials}, Journal of Alg.  Combinatorics {\bf 17} (2003), 283--307. 

\bibitem[MJ2]{MJ2} E.~M.~Moens and J.~Van der Jeugt, {\sl On dimension formulas for $\mathfrak{gl}(m|n)$-representations},  Journal of Lie Theory  {\bf 14} (2004), 523-535. 

\bibitem[Mn]{Mn} S.~Montgomery, {\sl Hopf algebras and their actions  on rings}, CBMS Regional Conf. Ser. Math. {\bf 82}, AMS, 1993. 

\bibitem[Mr1]{Mr1} M.~Morishita, {Analogies between knots and primes, 3-manifolds and number rings}, arXiv:0904.3399. 

\bibitem[Mr2]{Mr2} M.~Morishita, {Knots and primes: An introduction to arithmetic topology}, Springer 2012. 

\bibitem[O]{O} Mathoverflow discussion, \url{https://mathoverflow.net/questions/249541/formal-power-series-is-taylor-expansion-of-rational-function-iff-hankel-determin}.

\bibitem[OP]{OP} \"O.~\"Ozt\"urk and P.~Pragacz, {\sl On Schur function expansion of Thom  polynomials}, in Contributions to algebraic geometry: Impanga Lecture Notes, EMS Series of Congress Reports, 2012, 443-480,  arXiv:1111.6612. 

\bibitem[PT]{PT} B.~Peterson  and E.~J.~Taft, {\sl  The Hopf algebrra of linearly  recursive sequences}, Aequationes Mathematicae {\bf 20} (1980), 1-17.  

\bibitem[QR]{QR} H.~Queffelec and D.~E.~V.~Rose, {\sl The  $\mf{sl}_n$ foam 2-category: a combinatorial formulation of Khovanov-Rozansky homology via categorical skew Howe duality},
Adv. in Math. {\bf 302} (2016), 1251-1339. 

\bibitem[RW1]{RW1} L.-H.~Robert and E.~Wagner, {\sl A closed formula for the evaluation of $\mathfrak{sl}_N$-foams}, arXiv:1702.04140, to appear in Quantum Topology. 

\bibitem[RW2]{RW2} L.-H.~Robert and E.~Wagner, {\sl Symmetric Khovanov--Rozansky link homologies}, Jour. de l'École polytech. - Mathématiques {\bf 7} (2020), 573-651. 

\bibitem[R]{R} R.~Rouquier, {2-Kac-Moody algebras},  arXiv:0812.5023. 

\bibitem[S]{S} R.~Salem, {\sl Algebraic numbers and Fourier analysis}, D.~C.~Heath and Co., 1963, reprinted in The Wadsworth Mathematics Series, 1983, Wadsworth International Group. 

\bibitem[Sch]{Sch} K.~Schm\"udgen, {\sl The moment problem}, Grad. Texts Math. {\bf 277}, Springer, 2017. 

\bibitem[T]{T} U.~Tamm, {\sl Some aspects of Hankel matrices  in coding  theory  and combinatorics}, Electronic J. Comb. {\bf 8} is.~1 (2001), A1. 

\bibitem[TT]{TT} V.~Turaev, P.~Turner, {\sl Unoriented topological quantum field theory and link homology}, Alg. Geom. Top. {\bf 6}, no.~3 (2006), 1069--1093, arXiv:math/0506229. 

\bibitem[VS]{VS} L.~Verde-Star, {\sl Hopf algebras in analysis}, Int. J. of Theoretical Phys. {\bf 40} no.~1 (2001), 41-54. 

\bibitem[V]{V} P.~Vogel, {\sl Functoriality of Khovanov homology}, Journal Knot Theory Ramif., DOI: 10.1142/S0218216520500200, arXiv:1505.04545. 

\bibitem[W]{W} K.~Walker, {\sl Universal manifold pairings in dimension 3}, Celebratio Mathematica, Michael H.~Freedman, \url{https://celebratio.org/Freedman_MH/article/93/}, (2012). 

\bibitem[We1]{We1} B.~Webster, {\sl Knot invariants and higher representation theory}, Memoirs of the AMS {\bf 250}, 1991 (2017), AMS. 

\bibitem[We2]{We2} B.~Webster, {\sl Three perspectives on categorical symmetric Howe duality}, arXiv:2001.07584. 


\end{thebibliography}
\end{document}